\documentclass[a4paper,11pt]{amsart}
\usepackage{amsmath,amssymb,a4wide}
\usepackage{verbatim}
\newtheorem{theorem}{Theorem}[section]
\newtheorem{lemma}[theorem]{Lemma}
\newtheorem{proposition}[theorem]{Proposition}
\newtheorem{corollary}[theorem]{Corollary}
\newtheorem{claim}[theorem]{Claim}
\theoremstyle{definition}

\theoremstyle{remark}
\newtheorem{remark}{Remark}[section]
\numberwithin{equation}{section}
\DeclareMathOperator{\sech}{sech}
\DeclareMathOperator{\diag}{diag}

\DeclareMathOperator{\supp}{supp}
\newcommand{\R}{\mathbb{R}}
\newcommand{\C}{\mathbb{C}}
\newcommand{\N}{\mathbb{N}}
\newcommand{\Z}{\mathbb{Z}}
\newcommand{\eps}{\epsilon}
\renewcommand{\a}{\alpha} 
\newcommand{\tc}{\tilde{c}}
\newcommand{\tk}{\tilde{k}}

\newcommand{\tx}{\tilde{x}}
\newcommand{\ty}{\tilde{y}}
\newcommand{\tpsi}{\tilde{\psi}}
\newcommand{\la}{\langle}
\newcommand{\ra}{\rangle}
\newcommand{\pd}{\partial}

\newcommand{\mL}{\mathcal{L}}
\newcommand{\mF}{\mathcal{F}}
\newcommand{\mN}{\mathcal{N}}
\newcommand{\bb}{\mathbf{b}}
\newcommand{\bd}{\mathbf{d}}
\newcommand{\tb}{\tilde{b}}
\newcommand{\tbd}{\mathbf{\tilde{d}}}
\newcommand{\bM}{\mathbb{M}}
\newcommand{\bE}{\mathbf{E}}

\newcommand{\bW}{\mathbf{W}}
\newcommand{\bX}{\mathbf{X}}

\newcommand{\wP}{\widetilde{P}}
\newcommand{\ta}{\tilde{a}}
\newcommand{\wU}{\widetilde{U}}

\newcommand{\wS}{\widetilde{S}}

\newcommand{\oc}[1]{\overset{\circ}{#1}}
\newcommand{\obu}[1]{\overset{\bullet}{#1}}

\begin{document}
\title{Stability of Benney-Luke line solitary waves in 2D}
\author{Tetsu Mizumachi}
\address[T.~M.]{Division of Mathematical and
    Information Sciences, Hiroshima University, 739-8521, Japan.}
\email{tetsum@hiroshima-u.ac.jp}
\author{Yusuke Shimabukuro}
\address[Y.~S.]{Mathematical Engineering Group, Linea Co., Ltd.,
Tokyo, Japan.}
\email{yshimabukuro@lineainc.co.jp}
\begin{abstract} 
The $2$D Benney-Luke equation is an isotropic model 
which describes long water waves of small amplitude in $3$D whereas 
the KP-II equation is a unidirectional model for long waves with
slow variation in the transverse direction.
In the case where the surface tension is weak or negligible,
linearly stability of small line solitary waves of the $2$D Benney-Luke
equation was proved by Mizumachi and Shimabukuro [Nonlinearity, 30 (2017),
3419--3465]. In this paper, we prove nonlinear stability of the
line solitary waves by adopting the argument by Mizumachi
([Mem. Amer. Math. Soc. no. 1125],
[Proc. Roy. Soc. Edinburgh Sect. A., 148 (2018), 149--198] and
[https://arxiv.org/abs/1808.00809])
which prove nonlinear stability of $1$-line solitons
for the KP-II equation.
\end{abstract}
\maketitle
%\tableofcontents

\section{Introduction}
\label{sec:intro}
In this paper, we study nonlinear transverse stability of line solitary waves
for the Benney-Luke equation
\begin{equation}
  \label{eq:BL}
\pd_t^2\phi-\Delta\phi+a\Delta^2\phi-b\Delta\pd_t^2\phi
+(\pd_t\phi)(\Delta\phi)+\pd_t(|\nabla\phi|^2)=0
\quad\text{on $\R\times\R^2$.}
\end{equation}

The Benney-Luke equation is an approximation model of small amplitude
long water waves with finite depth originally derived by Benney and
Luke \cite{BL64} as a model for 3D water waves and its mathematically
rigorous derivation from the water wave equation was given
by \cite{Kano-Nishida}.
Here
$\phi=\phi(t,x,y)$ corresponds to a velocity potential of water waves.
We remark that \eqref{eq:BL} is an isotropic model for the propagation of
water waves whereas KdV, BBM and KP equations are unidirectional
models.  See e.g. \cite{BCG13,BCL05,BCS02} for the other bidirectional
models of 2D and 3D water waves. Since the Benney-Luke equation is isotropic
as the water wave equation, it could be more useful to describe nonlinear
interactions of waves at a high angle than the KP equations.
\par
The parameters $a$, $b$ are positive and satisfy $a-b=\hat{\tau}-1/3$, where
$\hat{\tau}$ is the inverse Bond number.
In this paper, we will assume $0<a<b$, which corresponds to the case where
the surface tension is weak or negligible.
\par

If we think of waves propagating in one direction, slowly evolving in
time and having weak transverse variation, then the Benney-Luke equation
can be formally reduced to the KP-II equation if $0<a<b$ and
to the KP-I equation if $a>b>0$. 
More precisely, the Benney-Luke equation \eqref{eq:BL} is reduced to
$$2f_{\tx\tilde{t}}+(b-a)f_{\tx\tx\tx\tx}+3f_{\tx}f_{\tx\tx}+f_{\ty\ty}=0$$
in the coordinate $\tilde{t}=\eps^3t$, $\tx=\eps(x-t)$ and
$\ty=\eps^2 y$ by taking terms only of order $\eps^5$,
where $\phi(t,x,y)=\eps f(\tilde{t},\tx,\ty)$.
See e.g. \cite{MilKel96} for the details.
On the other hand, the Benney-Luke equation \eqref{eq:BL} can realize the
finite time dynamics of the KP-II equation as in the case of $2$-dimensional
Boussinesq equations (see \cite{GaSc01} and \cite{Paumond}).
\par

The solution $\phi(t)$ of the Benney-Luke equation \eqref{eq:BL} formally
satisfies the energy conservation law
\begin{equation}
  \label{eq:conserve}
E(\phi(t),\pd_t\phi(t))=E(\phi_0,\phi_1)
\quad\text{for $t\in\R$,}
\end{equation}
where 
$$
E(f,g):=\frac12\int_{\R^2} \left\{|\nabla f|^2+a(\Delta
  f)^2+g^2+b|\nabla g|^2\right\}\, dxdy\,,$$
and \eqref{eq:BL} is globally well-posed in the energy class
$\bE:=(\dot{H}^2(\R^2)\cap \dot{H}^1(\R^2)) \times H^1(\R^2)$ (see
\cite{Q13}).  The Benney Luke equation \eqref{eq:BL} has a 3-parameter
family of line solitary wave solutions
\begin{equation}
\label{eq:lsw}
\phi(t,x,y)=\varphi_c(x\cos\theta+y\sin\theta-ct+\gamma)\,,
\quad \pm c>1\,,\quad \gamma\in\R\,,\quad \theta\in[0,2\pi)\,,
\end{equation}
where 
$$\varphi_c(x)=\beta(c)\left\{\tanh\Bigl(\frac{\a_c}{2}x\Bigr)-1\right\}\,,
\quad \a_c=\sqrt{\frac{c^2-1}{bc^2-a}}\,,\quad
\beta(c)=\frac{2(c^2-1)}{c\a_c}\,,
$$
and $$q_c(x):=\varphi_c'(x)=\frac{c^2-1}{c}\sech^2\bigl(\frac{\a_cx}{2}\bigr)$$
is a solution of
\begin{equation}
  \label{eq:qc}
  (bc^2-a)q_c''-(c^2-1)q_c+\frac{3c}{2}q_c^2=0\,.
\end{equation}
Stability of solitary waves to the $1$-dimensional Benney-Luke equation
was studied by \cite{Q03} for the strong surface tension case $a>b>0$
by using the variational argument (\cite{CL,GSS1}) which was originated by
\cite{Ben72,Bo75} and by \cite{MPQ13} for the weak surface tension case
$b>a>0$ by adopting the semigroup approach of \cite{PW94}.
\par
If $a>b>0$, then \eqref{eq:BL} has a stable ground state for $c$
satisfying $0<c^2<1$ (\cite{PQ99,Q05}). 
Note that for the water wave equation with strong surface tension,
orbital stability of solitary waves conditional on global solvability
has been proved by Mielke \cite{Mie02} and Buffoni \cite{Buffoni} by the variational argument. 
See also \cite{Marics} for the algebraic decay property of the ground state.
In view of \cite{RT1,RT2},
line solitary waves for the $2$-dimensional Benney-Luke equation
are expected to be unstable in this parameter regime.
\par
On the other hand if $0<a<b$ and $c:=\sqrt{1+\eps^2}$ is close to $1$
(the sonic speed), then $\varphi_c(x-ct)$ is expected to be
transversally stable because $q_c(x)$ is similar to a KdV $1$-soliton and
line solitons of the KP-II equation is transversally stable
(\cite{KP,Miz15,Miz18,MT11}).
\par

The dispersion relation for the linearized equation of \eqref{eq:BL} around $0$ is 
\begin{equation}
  \label{eq:drbl}
\omega^2=(\xi^2+\eta^2)\frac{1+a(\xi^2+\eta^2)}{1+b(\xi^2+\eta^2)}  
\end{equation}
for a plane wave solution $\phi(t,x,y)=e^{i(x\xi+y\eta-\omega t)}$.
If $b>a>0$, then $|\nabla\omega|\le 1$,
$\nabla\omega(\xi,\eta)\parallel(\xi,\eta)$ and line solitary waves
travel faster than the maximum group velocity of linear waves.  Using
this property and transverse linear stability of $1$-line solitons for
the KP-II equation (\cite{Burtsev85,Miz15}), transverse linear
stability of small line solitary waves of \eqref{eq:BL} was proved by
\cite{Miz-Sh17}.  The difference between the linear stability result
for solitary waves of the $1$-dimensional Benney-Luke equation
(\cite[Lemma~2.1 and Theorems~2.2 and 2.3]{MPQ13}) is that in the
$1$-dimensional case $\lambda=0$ is an isolated eigenvalue in
exponentially weighted spaces whereas $\lambda=0$ is not an isolated
eigenvalue of the linearized operator around line solitary waves
because line solitary waves do not decay in the transverse direction.
In \cite{Miz-Sh17}, we investigate the spectrum of the linearized
operator in a weighted space $L^2(\R^2;e^{2\a x}dxdy)$ with $\a>0$ and
find a curve of continuous spectrum
$\{\lambda \mid \pm i\lambda_{1,c}+\lambda_{2,c}\eta^2+O(\eta^3)\,,
\eta\in[-\eta_*,\eta_*]\}$, where $\lambda_{1,c}$, $\lambda_{2,c}$,
$\eta_*$ are positive constants.  We remark that the continuous
eigenmodes found in \cite{Miz-Sh17} grow exponentially as
$x\to-\infty$ and cannot be recognized as continuous eigenmodes in the
$L^2$-framework.  These resonant continuous eigenmodes have to do with
modulations of line solitary waves.  Indeed, we find in
\cite{Miz-Sh17} that linear evolution of those resonant continuous
eigenmodes can be approximately described by solutions of
$1$-dimensional dissipative linear wave equation on the time variable
$t$ and the transverse variable $y$ and that it illustrates phase shifts of
modulating line solitary waves.
\par
In this paper, we show that motion of the local amplitude $c(t,y)$ and the local phase shift $\gamma(t,y)$ 
of a modulating line solitary wave $\varphi_{c(t,y)}(x-\gamma(t,y))$ is described by $1$-dimensional a system of nonlinear dissipative
wave equations and prove nonlinear transverse stability of line solitary waves.
The nonlinear stability of line solitary waves in the entire domain has been proved only for the KP-II equation (\cite{Miz15,Miz18,Miz19})
and it is interesting to give its mathematical proof for non-integrable system such as the Benny-Luke equations.
Since the arguments in \cite{Miz15,Miz18,Miz19} are based on PDE methods and integrability of the KP-II equation is used only to prove linear stability of
line solitary waves, we are able to extend the arguments for the Benney-Luke equation.
\par

Our plan of the present paper is as follows.  In
Section~\ref{subsec:linear stability}, we recall the linear stability
property of line solitary waves in \cite{Miz-Sh17}
(Theorem~\ref{thm:linear-stability}) and introduce our main result
(Theorem~\ref{thm:main}) in Section~\ref{subsec:main}.  In
Section~\ref{sec:decay-est}, we introduce $L^2-\mF^{-1}L^\infty$ estimates
of solutions for the linearized modulation equations as well as
a substitute of d'Alembert's formula to prove the $L^\infty$-bound of the phase shifts of line solitary waves.
We also investigate the large time behavior of the solutions
to the linearized equation.
In Section~\ref{sec:decomposition}, we decompose a solution around line
solitary waves into a sum of the modulating line solitary wave, a
small freely propagating solution of \eqref{eq:BL}, exponentially
localized remainder part and an auxiliary function.  We will impose
the secular term condition to the exponentially localized part to make
use of linear stability property of line solitary waves.  We split
small solutions of \eqref{eq:BL} from solutions around line solitary
waves so that the remainder part is exponentially localized because as
in \cite{Miz18}, resonant continuous eigenmodes of the adjoint
linearized operator grow exponentially as $x\to\infty$ and the secular
term condition makes sense only for exponentially localized
perturbations. Since $\varphi_c(-\infty)$ differs as $c$ varies,
we need a correction term to keep the remainder terms in the energy class.
In Section~\ref{sec:modulation},
we compute time derivative of the secular term condition
and derive a system of PDEs that describe the motion of the local amplitude
$c(t,y)$ and the local phase shift $\gamma(t,y)$ whose linearized
equation is
$$\left\{\begin{aligned}
& \gamma_t\simeq \lambda_{2,c_0}\gamma_{yy}+c-c_0\,,
\\ & c_t\simeq \lambda_{1,c_0}^2\gamma_{yy}+\lambda_{2,c_0}c_{yy}\,.
\end{aligned}\right.$$
The modulation equation derived from the secular term condition
has a critical nonlinear term whose $L^1(\R_y)$-norm decays like $t^{-1}$
and a term coming from the freely propagating remainder part
whose $L^1(\R_y)$-norm is expected to grow as $t\to\infty$.
These terms are harmful to estimate modulation parameters $c(t,y)$ and
$\gamma(t,y)$.
As in \cite{Miz18,Miz19}, we use a change of variables to transform these terms into harmless forms.
In Section~\ref{sec:apriori}, we prove decay estimates of the local amplitude
$c(t,y)$ and the local orientation $\gamma_y(t,y)$ of the modulating
line solitary wave $\varphi_{c(t,y)}(x-c_0t-\gamma(t,y))$ by applying
linear estimates obtained in Section~\ref{sec:decay-est},
where $c_0$ is the amplitude of the unperturbed line solitary wave.
In Section~\ref{sec:energy}, we estimate a \textit{renormalized energy}
of perturbations which we find by removing infinite energy parts
from solutions around line solitary wave and making use of the
orthogonality condition imposed on the exponentially localized part 
in Section~\ref{subsec:orth}. 
In Section~\ref{sec:virial}, we prove virial identities for \eqref{eq:BL}
and prove decay estimates for localized energies of small solutions to
\eqref{eq:BL}. 
In Section~\ref{sec:U2}, we estimate an exponentially localized norm
of perturbations by using linear stability property of line solitary waves
(Theorem~\ref{thm:linear-stability}). 
If we linearize \eqref{eq:BL} around a modulating line solitary wave
and use the moving coordinate $z=x-c_0t-\gamma(t,y)$, we have
a space-time dependent advection term. Since \eqref{eq:BL}
is a $2$-dimensional wave equation with dispersion in the low frequency regime,
the smoothing effect of \eqref{eq:BL} is not strong enough to treat
the advection term as a remainder part.
To avoid the appearance of the advection term,
we prove that $\|\gamma(t,\cdot)\|_{L^\infty(\R_y)}$ remains small for polynomially localized perturbations
to line solitary waves by using an estimate similar to d'Alembert's formula following the idea of \cite{Miz19}.
In our paper, it remains open whether the phase shift $\gamma(t,y)$
can grow as $t$, $y\to\pm\infty$ for perturbations in the energy class.
See \cite{Marel-Merle05} for the growth of phase shifts of gKdV solitary waves.
In Section~\ref{sec:phase}, we show that the large time behavior
of $c(t,y)$ and $\gamma_y(t,y)$ can be expressed by a linear combination
of self-similar solutions of the Burgers' equation with spatial phase shift
$\pm\lambda_{1,c_0}t$ and prove transverse stability of line solitary
waves for polynomially localized perturbations.
\par
Finally, let us introduce several notations. 
We denote by $\sigma(T)$ the spectrum of the operator $T$.
For Banach spaces $V$ and $W$, let $B(V,W)$ be the space of all
linear continuous operators from $V$ to $W$ and
$\|T\|_{B(V,W)}=\sup_{\|u\|_V=1}\|Tu\|_W$ for $A\in B(V,W)$.
We abbreviate $B(V,V)$ as $B(V)$.
For $f\in \mathcal{S}(\R^n)$ and $m\in \mathcal{S}'(\R^n)$, let 
\begin{gather*}
(\mathcal{F}f)(\xi)=\hat{f}(\xi)
=(2\pi)^{-n/2}\int_{\R^n}f(x)e^{-ix\xi}\,dx\,,\\
(\mathcal{F}^{-1}f)(x)=\check{f}(x)=\hat{f}(-x)\,,
\end{gather*}
and $(m(D)f)(x)=(2\pi)^{-n/2}(\check{m}*f)(x)$.
We denote $\la f,g\ra$ by
$$\la f,g\ra=\sum_{j=1}^m \int_\R f_j(x)\overline{g_j(x)}\,dx$$
for $\C^m$-valued functions $f=(f_1,\cdots,f_m)$ and
$g=(g_1,\cdots,g_m)$.

Let $L^2_\a(\R^2)=L^2(\R^2;e^{2\a x}dxdy)$, $L^2_\a(\R)=L^2(\R;e^{2\a x}\,dx)$.
For $k\ge1$, let $H^k_\a(\R^2)$ and $H^k_\a(\R)$ be Hilbert spaces with the norms
\begin{gather*}
\|u\|_{H^k_\a(\R^2)}=\left(\|\pd_x^ku\|_{L^2_\a(\R^2)}^2
+\|\pd_y^ku\|_{L^2_\a(\R^2)}^2+\| u\|_{L^2_\a(\R^2)}^2\right)^{1/2}\,,
\\
\|u\|_{H^k_\a(\R)}=\left(\|\pd_x^ku\|_{L^2_\a(\R)}^2+\|u\|_{L^2_\a(\R)}^2\right)^{1/2}\,,
\end{gather*}
and let $\bX_k= H^{k+1}_\a(\R^2)\times H^k_\a(\R^2)$ with the norm
$$\|(u_1,u_2)\|_{\bX_k}
=\sqrt{\|\nabla u_1\|_{H^k_\a(\R^2)}^2+\|u_2\|_{H^k_\a(\R^2)}^2}\,,$$
and let $\bX=\bX_0$.
\par
The symbol $\la x\ra$ denotes $\sqrt{1+x^2}$ for $x\in\R$.
Let $\mathbf{1}_A$ be the characteristic function of the set $A$.
\par
We use $a\lesssim b$ and $a=O(b)$ to mean that there exists a
positive constant such that $a\le Cb$.
Let $a\wedge b=\min\{a,b\}$ and $a\vee b=\max(a,b)$. 
\bigskip

\section{Statement of results}
\label{sec:stability}
\subsection{Linear stability of line solitary waves}
\label{subsec:linear stability}
To begin with, we recall the linear stability of line solitary waves for \eqref{eq:BL} (\cite{Miz-Sh17}).
Let $\phi_1=\phi$, $\phi_2=\pd_t\phi$, $A=I-a\Delta$ and $B=I-b\Delta$.
Then the Benney-Luke equation \eqref{eq:BL} can be rewritten as a system
\begin{gather}
  \label{eq:BL1}
\pd_t\Phi=L\Phi+N(\Phi)\,,
\end{gather}
$$
\Phi=\begin{pmatrix}  \phi_1 \\ \phi_2\end{pmatrix}\,,
\quad L=
\begin{pmatrix} 0 & I \\ B^{-1}A\Delta & 0 \end{pmatrix}\,,
\quad
N(\Phi)=-B^{-1} 
\begin{pmatrix}
  0 \\ \phi_2\Delta\phi_1+2\nabla\phi_1\cdot\nabla\phi_2
\end{pmatrix}\,.$$
Since \eqref{eq:BL} is isotropic and translation invariant,
we may assume $\theta=\gamma=0$ in \eqref{eq:lsw} without loss of generality.
Let $r_c(x)=-cq_c(x)$ and $\Phi_c=(\varphi_c, r_c)^T$.
By \eqref{eq:qc}, 
\begin{equation}
  \label{eq:stationary}
  (c\pd_x+L)\Phi_c+N(\Phi_c)=0\,,
\end{equation}
and $\Phi_c(x-ct)$ is a planar traveling wave solution  of \eqref{eq:BL1}.
Linearizing \eqref{eq:BL1} around $\Phi_c$, we have
in the moving coordinate $z=x-ct$,
\begin{equation}
  \label{eq:linear}
\pd_t \Phi =\mL_c \Phi\,,
\end{equation}
where $\mL_c=c\pd_z+L+V_c$ and
$$V_c=-B^{-1}\begin{pmatrix}  0 & 0 \\ v_{1,c} & v_{2,c}\end{pmatrix}\,,
\quad v_{1,c}=2r_c'(z)\pd_z+r_c(z)\Delta\,,\quad
v_{2,c}=2q_c(z)\pd_z+q_c'(z)\,.$$
Let 
$A(\eta)=1+a\eta^2-a\pd_z^2$,  $B(\eta)=1+b\eta^2-b\pd_z^2$ and
\begin{gather*}
L(\eta)=
\begin{pmatrix}
0 & 1 \\ B(\eta)^{-1}A(\eta)(\pd_x^2-\eta^2) & 0  
\end{pmatrix}\,,\quad
V_c(\eta)=-B(\eta)^{-1}
\begin{pmatrix}
  0 & 0 \\ v_{1,c}(\eta) & v_{2,c}(\eta)
\end{pmatrix}\,,
\\
v_{1,c}(\eta)=2r_c'\pd_x+r_c(\pd_x^2-\eta^2)\,,\quad
v_{2,c}(\eta)=2q_c\pd_z+q_c'\,.
\end{gather*}
Then 
$\mL_c(\eta)=e^{-iy\eta}\mL_c(\,\cdot\,)
=c\pd_x+L(\eta)+V_c(\eta)$.
We expand $L(\eta)$ and $\mL_c(\eta)$  as
\begin{align*}
& L(\eta)=L_0+\eta^2L_1(\eta)\,,\quad
L_1(0)=B_0^{-1}(I-A_0-B_0^{-1}A_0)E_{21}\,, \\ &
\mL_c(\eta)=\mL_c(0)+\eta^2\mL_{1,c}(\eta)\,,
\\ &
\mL_{1,c}(0)=L_1(0)+B_0^{-1}r_cE_{21}+bB_0^{-2}\{v_{1,c}(0)E_{21}+v_{2,c}(0)E_2\}\,,
\end{align*}
where $L_0=L(0)$, $A_0=A(0)$, $B_0=B(0)$ and
$E_{ij}$ ($i$, $j=1$, $2$) is a $2\times2$ matrix whose entry in row $i$
and column $j$ equals to $1$ and the other entries are zero.
We will abbreviate $E_{ii}$ as $E_i$.
\par

By Theorem~2.1 in \cite{Miz-Sh17}, there exist an $\eta_0>0$,
$\lambda_c(\eta)\in C^\infty([-\eta_0,\eta_0])$ and
\begin{gather*}
\zeta_c(\cdot,\eta)\in
C^\infty([-\eta_0,\eta_0];H^k_\a(\R)\times H^{k-1}_\a(\R))\,,\\
\zeta_c^*(\cdot,\eta)\in
C^\infty([-\eta_0,\eta_0];H^k_{-\a}(\R)\times H^{k-1}_{-\a}(\R))  
\end{gather*}
such that for $\eta\in[-\eta_0,\eta_0]$ and $z\in\R$,
\begin{gather}
\notag \mL_c(\eta) \zeta_c(z,\eta)=\lambda_c(\eta)\zeta_c(z,\eta)\,,
\quad 
\mL_c(\eta)^* \zeta_c^*(z,\eta)=\lambda_c(-\eta)\zeta_c^*(z,\eta)\,,
\\  \label{eq:lambda-asymp}
\lambda_c(\eta)=i\lambda_{1,c}\eta -\lambda_{2,c}\eta^2+O(\eta^3)\,,
\\  \label{eq:zeta-asymp}
\zeta_c(\cdot,\eta)=\zeta_{1,c}+\{i\lambda_{1,c}\eta+O(\eta^2)\}\zeta_{2,c}
+z_c(\eta)\,,
\quad z_c(\eta)\perp \zeta_{1,c}^*, \zeta_{2,c}^*\,,
\\  \label{eq:zeta*-asymp}
\zeta_c^*(\cdot,\eta)=\zeta_{2,c}^*-\{i\lambda_{1,c}\eta+O(\eta^2)\}\zeta_{1,c}^*
+z_c^*(\eta)\,,
\quad z_c^*(\eta)\perp \zeta_{1,c}, \zeta_{2,c}\,,
\\  \label{eq:zeta-asymp-remainder}
\|z_c(\eta)\|_{H^k_\a(\R)\times H^{k-1}_\a(\R)}
+\|z_c^*(\eta)\|_{H^k_{-\a}(\R)\times H^{k-1}_{-\a}(\R)}=O(\eta^2)\,,
\\  \label{eq:zeta-parity}
\overline{\lambda_c(\eta)}=\lambda_c(-\eta)\,,\quad
\overline{\zeta_c(z,\eta)}=\zeta_c(z,-\eta)\,,\quad
\overline{\zeta_c^*(z,\eta)}=\zeta_c^*(z,-\eta)\,,
\end{gather}
where $\varphi_c^*(x)=\varphi_c(x)+2\beta(c)$ and 
\begin{gather}
\label{eq:zetas}
\zeta_{1,c}=\pd_x\Phi_c=
\begin{pmatrix}  q_c\\ r_c'\end{pmatrix}\,,
\quad
\zeta_{2,c}=-\pd_c\Phi_c=
-\begin{pmatrix} \pd_c\varphi_c \\ \pd_cr_c \end{pmatrix}\,,
\\ \label{eq:zeta*s}
\zeta_{1,c}^*=c
\begin{pmatrix}
-B_0\pd_cr_c-2q_c\pd_cq_c-q_c'\pd_c\varphi_c^*
\\ B_0\pd_c\varphi_c^*
\end{pmatrix}
\,,\quad
\zeta_{2,c}^*=\begin{pmatrix}  A_0q_c'\\ -B_0r_c\end{pmatrix}\,,
\\ \label{eq:lambda1}
\lambda_{1,c}=\sqrt{
\frac{\la \mL_{1,c}(0)\zeta_{1,c},\zeta_{2,c}^*\ra}
{-\la\zeta_{2,c},\zeta_{2,c}^*\ra}}>0\,,
\\ \label{eq:lambda2}
\lambda_{2,c}=
\frac{\la\mL_{1,c}(0)\zeta_{1,c},\zeta_{2,c}^*\ra \la \zeta_{2,c},\zeta_{1,c}^*\ra
-\sum_{j=1,2}\la\mL_{1,c}(0)\zeta_{j,c},\zeta_{j,c}^*\ra
\la \zeta_{3-j,c},\zeta_{3-j,c}^*\ra}
{2\la \zeta_{1,c},\zeta_{1,c}^*\ra \la \zeta_{2,c},\zeta_{2,c}^*\ra}>0\,.
\end{gather}
We remark that $\mL_c(0)$ is the same with the linearized operator of
$1$-dimensional Benney-Luke equation around solitary waves and that
$\mL_c(0)\zeta_{1,c}=0$, $\mL_c(0)\zeta_{2,c}=\zeta_{1,c}$.
\par
By \cite[(6.4)--(6.6), (6.9)--(6.11)]{Miz-Sh17},
\begin{gather}
  \label{eq:zeta1-zeta1*}
\la \zeta_{1,c},\zeta_{2,c}^*\ra =0\,,\quad
\la \zeta_{1,c},\zeta_{1,c}^*\ra = \la \zeta_{2,c},\zeta_{2,c}^*\ra
=\frac{1}{2}\frac{d}{dc}E(q_c,r_c)=:\beta_1(c) >0\,,
\\  \label{eq:zeta2-zeta1*}
\la \zeta_{2,c},\zeta_{1,c}^*\ra=:\beta_2(c) >0\,,
\\ \label{eq:L1czs}
\la \mL_{1,c}(0)\zeta_{1,c},\zeta_{1,c}^*\ra
=\la \mL_{1,c}(0)\zeta_{2,c},\zeta_{2,c}^*\ra<0\,,
\quad
\la \mL_{1,c}(0)\zeta_{1,c},\zeta_{2,c}^*\ra<0\,.
\end{gather}
\par

Now we introduce spectral projections associated with continuous eigenvalues
$\{\lambda(\eta)\}_{-\eta_0\le \eta\le \eta_0}$.
As in \cite[Section~8]{Miz-Sh17}, let 
$\kappa_c(\eta)=\frac12 \Im\la g(\cdot,\eta,c),g^*(\cdot,\eta,c)\ra$ and
\begin{gather}
g(z,\eta,c)=\left(1
+i\frac{\Re \la \zeta_c(\cdot,\eta),\zeta_c^*(\cdot,\eta)\ra}
{\Im \la \zeta_c(\cdot,\eta),\zeta_c^*(\cdot,\eta)\ra}
\right) \zeta_c(z,\eta)\,,
\quad g^*(z,\eta,c)=\zeta_c^*(z,\eta)\,,
\\
\label{eq:def-gk}
g_1(z,\eta,c)=\frac{1}{\beta_1(c)}\Re g(z,\eta,c)\,,
\quad
g_2(z,\eta,c)=\frac{1}{\kappa_c(\eta)}\Im g(z,\eta,c)\,,
\\
\label{eq:def-gk*}
g_1^*(z,\eta,c)=-\frac{\beta_1(c)}{\kappa_c(\eta)}\Im g^*(z,\eta,c)\,,
\quad g_2^*(z,\eta)=\Re g^*(z,\eta,c)\,.
\end{gather}
Note that  $g_k(z,\eta,c)$ and $g_k^*(z,\eta,c)$ are $\R^2$-valued functions 
that are even in $\eta$ and that
\begin{gather}
\mL_c(\eta)g_1(z,\eta,c)=Re\lambda_c(\eta)g_1(z,\eta,c)
-\frac{\kappa_c(\eta)}{\beta_1(c)}\Im\lambda_c(\eta)g_2(z,\eta,c)\,,
\\
\mL_c(\eta)g_2(z,\eta,c)=
\frac{\beta_1(c)}{\kappa_c(\eta)}\Im\lambda_c(\eta)g_1(z,\eta,c)
+Re\lambda_c(\eta)g_2(z,\eta,c)\,,
\\ \label{eq:gevp1}
\mL_c^*(\eta)g_1^*(z,\eta,c)=
\frac{\kappa_c(\eta)}{\beta_1(c)}\Im\lambda_c(\eta)g_1^*(z,\eta,c)
+Re\lambda_c(\eta)g_2^*(z,\eta,c)\,,
\\  \label{eq:gevp2}
\mL_c^*(\eta)g_2^*(z,\eta,c)=Re\lambda_c(\eta)g_1^*(z,\eta,c)
+\frac{\beta_1(c)}{\kappa_c(\eta)}\Im\lambda_c(\eta)g_2^*(z,\eta,c)\,.
\end{gather}

Let $\a\in(0,\a_c)$ and $\eta_0$ be a small positive number
such that for $k=1$, $2$,
$$
\sup_{\eta\in[-\eta_0,\eta_0]}
\left(
\|g_k(\cdot,\eta,c)\|_{L^2_\a(\R)}+\|g_k^*(\cdot,\eta,c)\|_{L^2_{-\a}(\R)}
\right)<\infty\,.$$ 
Then for $\eta\in[-\eta_0,\eta_0]$,
\begin{gather}
\la g_j(\cdot,\eta,c),g_k^*(\cdot,\eta,c)\ra=\delta_{jk}\,, \notag
\\ \notag
\|g_1(\cdot,\eta,c)-\beta_1(c)^{-1}\zeta_{1,c}\|_{L^2_\a(\R;\R^2)}=O(\eta^2)\,,
\\ \notag
\|g_2(\cdot,\eta,c)+\beta_1(c)^{-2}\beta_2(c)\zeta_{1,c}
-\beta_1(c)^{-1}\zeta_{2,c}\|_{L^2_\a(\R;\R^2)}=O(\eta^2)\,,
\\ \label{eq:gk*-approx}
\|g_1^*(\cdot,\eta,c)-\zeta_{1,c}^*\|_{L^2_{-\a}(\R;\R^2)}
+\|g_2^*(\cdot,\eta,c)-\zeta_{2,c}^*\|_{L^2_{-\a}(\R;\R^2)}=O(\eta^2)\,.
\end{gather}
Let $\mathcal{P}_c(\eta_0)$ and $\mathcal{Q}_c(\eta_0):\bX\to \bX$ 
be projections defined by
\begin{gather}
\label{def-P-1}
\mathcal{P}_c(\eta_0)U(x,y)=\sum_{k=1,\,2}
\int_{-\eta_0}^{\eta_0}c_k(\eta)g_k(x,\eta,c)e^{iy\eta}\,d\eta
\,,\quad \mathcal{Q}_c(\eta_0)=I-\mathcal{P}_c(\eta_0)\,,
 \\
\label{def-P-2}
 c_k(\eta)=\frac{1}{\sqrt{2\pi}}\int_\R
(\mF_yU)(x,\eta)\cdot g_k^*(x,\eta,c)\,dx\,.
\end{gather}
Then $\mathcal{P}_c(\eta_0)$ is a spectral projection for $\mL_c$
corresponding to a family of continuous eigenvalues
$\{\lambda_c(\eta)\}_{-\eta_0\le\eta\le \eta_0}$.
Let $\mathbf{Z}=\mathcal{Q}_c\bX$ and $\mL_c|_{\mathbf{Z}}$ be the restriction
of the operator $\mL_c$ to $\mathbf{Z}$.
Assuming spectral stability of $\mL_c|_{\mathbf{Z}}$, we have exponential
stability of $e^{t\mL_c}\mathcal{Q}_c$.
\begin{theorem}\emph{(\cite[Theorem~2.2]{Miz-Sh17})}
  \label{thm:linear-stability}
Let $0<a<b$, $c>1$ and $\a\in(0,\a_c)$. Consider the operator $\mL_c$
in the space $\bX$. Assume that there exist positive
constants $\beta$ and $\eta_0$ such that
\begin{equation}
  \label{ass:S}
\sigma(\mL|_{\mathbf{Z}})  \subset \{\lambda\mid \Re\lambda\le -\beta\}\,.
\tag{S}
\end{equation}
Then for any $\beta'<\beta$, there exists a positive constant $C$ such that
\begin{equation}
  \label{eq:decay}
\|e^{t\mL_c}\mathcal{Q}_c(\eta_0)\|_{B(\bX)}
\le Ce^{-\beta' t}\quad\text{for any  $t\ge0$.}  
\end{equation}
\end{theorem}
\begin{remark}
\label{rem:small}
  If $c>1$ is sufficiently close to $1$, then the assumption
  \eqref{ass:S} is valid and the spectrum of $\mL_c$ near $0$ is
  similar to that of the linearized KP-II operator around a line
  soliton solution. See \cite[Theorem~2.4]{Miz-Sh17}.
\end{remark}

\subsection{Main Result}
\label{subsec:main}
Now let us introduce our main result.
\begin{theorem}
  \label{thm:main}
Let $0<a<b$ and $c_0>1$. Assume \eqref{ass:S} for $c=c_0$.
Suppose that  $\Phi(t,x,y)$ is a solution of \eqref{eq:BL1}
satisfying
\begin{equation}
  \label{eq:init-cond}
\Phi(0,x,y)=\Phi_{c_0}(x)+U_0(x,y)\,,\quad U_0(x,y)=(u_{01}(x,y),u_{02}(x,y))^T\,.  
\end{equation}
Then there exist positive constants $C$ and $\eps_0$ such that if
\begin{equation}
  \label{eq:init-size}
\eps:=\|(1+x^2+y^2)\nabla u_{01}\|_{H^1(\R^2)}
+\|(1+x^2+y^2) u_{02}\|_{H^1(\R^2)}<\eps_0\,,
\end{equation}
then there exist $C^1$-functions $c(t,y)$ and $\gamma(t,y)$ such that
for every $t\ge0$ and $k\ge0$,
\begin{align}
\label{OS}
 & \|\Phi(t,x,y)-\Phi_{c(t,y)}(x-c_0t-\gamma(t,y))\|_\bE\le C\eps\,,
 \\  & \label{phase1}
\sum_{j=0,1}\la t\ra^{(2j+1)/4}\left\{\left\|\pd_y^j\left(c(t,\cdot)-c_0\right)\right\|_{H^k(\R)}+\left\|\pd_y^{j+1}\gamma(t,\cdot)\right\|_{H^k(\R)}\right\}
\le C\eps\,,
\\ & \|\gamma(t,\cdot)\|_{L^\infty(\R)}\le C\eps\,,
\end{align}
and for any $R>0$,
\begin{equation}
\label{AS}
\lim_{t\to\infty}
\left\|\Phi(t,x+c_0t+\gamma(t,y),y)-\Phi_{c_0}(x)\right\|_{\bE((x>-R)\times\R_y)}
=0\,.
\end{equation}
Moreover, there exists a $\gamma_\infty\in\R$ such that 
for any $\delta>0$,
\begin{equation}
  \label{eq:phase-lim}
  \begin{cases}
\lim_{t\to\infty}\left\|\gamma(t,\cdot)-\gamma_\infty
\right\|_{L^\infty(|y|\le (\lambda_{1,c_0}-\delta)t)}=0\,,
\\    
\lim_{t\to\infty}\left\|\gamma(t,\cdot)
\right\|_{L^\infty(|y|\ge (\lambda_{1,c_0}+\delta)t)}=0\,.
  \end{cases}
\end{equation}
\end{theorem}

In the case where $\gamma_\infty\ne0$ in \eqref{eq:phase-lim}, 
the distance between the solution $u$ and the set of line solitary waves
in the energy space grows like $t^{1/2}$ or faster.
\begin{corollary}
  \label{cor:instability2}
Let $c_0>1$. Suppose that \eqref{ass:S} holds for $c=c_0$.
Then for any $\eps_0>0$, there exists a solution of \eqref{eq:BL1} satisfying
\eqref{eq:init-cond}, \eqref{eq:init-size} and
$$\liminf_{t\to\infty}t^{-1/2}\inf_{\widetilde{\Phi} \in\mathcal{K}}
\|\Phi(t,\cdot)-\widetilde{\Phi}\|_\bE>0\,,$$ where
$\mathcal{K}=\{\Phi_c(x\cos\theta+y\sin\theta-ct+\gamma)\mid
 \pm c>1\,,\quad \gamma\in\R\,,\quad \theta\in[0,2\pi)\}$.
\end{corollary}
\bigskip

\section{Decay estimates for linearized modulation equations}
\label{sec:decay-est}
Modulation of the local amplitude and the local phase shift of line solitary
waves can be described by a system of Burgers' equations.
In this subsection, we introduce decay estimates for linearized
modulation equations following \cite{Miz15,Miz19}.
\par
Let $\nu$ be a real number, 
$\omega(\eta)=\bigl\{1-(\nu/\lambda_{1,c_0})^2\eta^2\bigr\}^{1/2}$ and
\begin{align*}
& \lambda_\pm(\eta)=-\lambda_{2,c_0}\eta^2\pm i\lambda_{1,c_0}\eta\omega(\eta)
\,,\quad \pi_\pm(\eta)=
(\nu\eta\mp i\lambda_{1,c_0}\omega(\eta), \lambda_{1,c_0}^2\eta)^T\,,
\\ &
\mathcal{A}_*(\eta)=
\begin{pmatrix}
-(\lambda_{2,c_0}+\nu)\eta^2 & 1 \\ 
-\lambda_{1,c_0}^2\eta^2 & -(\lambda_{2,c_0}-\nu)\eta^2
\end{pmatrix}\,,
\quad
\Pi(\eta)=\frac{1}{\lambda_{1,c_0}\eta}\left(\pi_+(\eta),\pi_-(\eta)\right)\,.
\end{align*}
Then $\mathcal{A}_*(\eta)\Pi(\eta)
=\diag(\lambda_+(\eta),\lambda_-(\eta))\Pi(\eta)$ and 
\begin{equation}
  \label{eq:etA}
e^{t\mathcal{A}_*(\eta)}=e^{-\lambda_{2,c_0}t\eta^2}\left\{
\cos t\lambda_{1,c_0}\eta\omega(\eta) I
+\frac{\sin t\lambda_{1,c_0}\eta\omega(\eta)}
{\lambda_{1,c_0}\eta\omega(\eta)}
\left(\mathcal{A}_*(\eta)+\lambda_{2,c_0}I\right)\right\}
\end{equation}

Let $\eta_0$ be a positive number satisfying
\begin{equation}
  \label{ass:eta0}
|\nu|\eta_0<\lambda_{1,c_0}\,,
\end{equation}
and let 
$\chi_1(\eta)$ and $\chi_2(\eta)$ be nonnegative smooth functions such that
$\chi_1(\eta)+\chi_2(\eta)=1$, 
$\chi_1(\eta)=1$ if $|\eta|\le \frac12\eta_0$ and $\chi_1(\eta)=0$ if
$|\eta|\ge \frac{3}{4}\eta_0$.
Then 
\begin{equation}
  \label{eq:k-est,high}
\|\chi_2(D_y)e^{t\mathcal{A}_*(D_y)}\|_{B(L^2(\R))}\lesssim e^{-\eta_0^2t/2}
\quad\text{for $t\ge0$.}
\end{equation}
Next, we will estimate the low frequency part of $e^{t\mathcal{A}_*(\eta)}$.
Let  
\begin{align*}
& K_1(t,y)=\frac{1}{\sqrt{2\pi}}\mathcal{F}^{-1}
\left(\chi_1(\eta)e^{-\lambda_{2,c_0}t\eta^2}\cos t\lambda_{1,c_0}\eta\omega(\eta)
\right)\,,
\\ &
K_2(t,y)=\frac{1}{\sqrt{2\pi}}\mathcal{F}^{-1} \left(
e^{-\lambda_{2,c_0}t\eta^2}\frac{\eta\chi_1(\eta)}{\omega(\eta)}
\sin t\lambda_{1,c_0}\eta\omega(\eta)\right)\,,
\\ &
K_3(t,y)=\frac{1}{\sqrt{2\pi}}\mathcal{F}^{-1}\left(
e^{-\lambda_{2,c_0}t\eta^2} \frac{\chi_1(\eta)}{\eta\omega(\eta)}
\sin t\lambda_{1,c_0}\eta\omega(\eta) \right)\,.  
\end{align*}
Then
\begin{equation}
  \label{eq:fundamental-low}
\chi_1(D_y)e^{t\mathcal{A}_*(D_y)}\delta=\begin{pmatrix}
K_1(t,y)-\frac{\nu}{\lambda_{1,c_0}}K_2(t,y)  &
\frac{1}{\lambda_{1,c_0}}K_3(t,y) \\ -\lambda_{1,c_0}K_2(t,y)
&  K_1(t,y)+\frac{\nu}{\lambda_{1,c_0}}K_2(t,y)
\end{pmatrix}\,.
\end{equation}
We can prove the following estimates for $K_1$, $K_2$ and $K_3$
in the same way as \cite[Lemma~2.2]{Miz19}.
\begin{lemma}
  \label{lem:fundamental-sol}
Suppose \eqref{ass:eta0} and $k\ge0$. Then for every $t\ge 0$,
  \begin{gather}
    \label{eq:k1-est}
\sup_{t>0}\|K_1(t,\cdot)\|_{L^1(\R)}<\infty\,,\quad \|K_1(t,\cdot)\|_{L^2(\R)}\lesssim \la t \ra^{-1/4}\,,
\\  \label{eq:k2-est1}
\|\pd_y^{j+1}K_1(t,\cdot)\|_{L^1(\R)} +\|\pd_y^jK_2(t,\cdot)\|_{L^1(\R)}+
\|\pd_y^{j+2}K_3(t,\cdot)\|_{L^1(\R)} 
\lesssim \la t \ra^{-(j+1)/2}\,,
\\
\label{eq:k2-est2}
\quad \|\pd_y^{j+1}K_1(t,\cdot)\|_{L^2(\R)}+\|\pd_y^jK_2(t,\cdot)\|_{L^2(\R)}
+\|\pd_y^{j+2}K_3(t,\cdot)\|_{L^2(\R)}\lesssim \la t \ra^{-(2j+3)/4}\,,
\\
\label{eq:k3-est'}
\sup_{t>0}\|\pd_yK_3(t,\cdot)\|_{L^1(\R)}<\infty\,,
\quad \|\pd_yK_3(t,\cdot)\|_{L^2(\R)}\lesssim \la t \ra^{-1/4}\,,
\\  \label{eq:k3-est}
\sup_{t>0}\|K_3(t,\cdot)*f\|_{L^\infty(\R)}\lesssim \|f\|_{L^1(\R)}\,.
\end{gather}
\end{lemma}

Let $Y$ and $Z$ be closed subspaces of $L^2(\R)$ defined by
$$Y=\mF^{-1}_\eta Z\quad\text{and}\quad
Z=\{U \in L^2(\R)\mid\supp U \subset[-\eta_0,\eta_0]\}\,,$$
and let $Y_1=\mF^{-1}_\eta Z_1$ and
$Z_1=\{U \in Z \mid\|U\|_{Z_1}:=\|U\|_{L^\infty}<\infty\}$.
By the definition,
\begin{equation}
  \label{eq:H-infty-Y}
\|U\|_{H^s}\le (1+\eta_0^2)^{s/2}\|f\|_{L^2}
\quad\text{for any $s\ge0$ and $U\in Y$.}  
\end{equation}
Especially, we have $\|U\|_{L^\infty}\lesssim \|U\|_{L^2}$ for any $U\in Y$.
Let $\wP_1$ be a projection defined by
$\wP_1 U=\mF_\eta^{-1}\mathbf{1}_{[-\eta_0,\eta_0]}\mF_yU$.
Then $\|\wP_1U\|_{Y_1}\le (2\pi)^{-1/2}\|U\|_{L^1}$ for any $U\in L^1(\R)$.
In particular, for any $U_1$, $U_2\in Y$,
\begin{equation}
  \label{eq:Y1-L1}
  \|\wP_1(U_1U_2)\|_{Y_1}\le (2\pi)^{-1/2}\|U_1U_2\|_{L^1}
  \le (2\pi)^{-1/2}\|U_1\|_Y\|U_2\|_Y\,.
\end{equation}

Let $\chi(\eta)$ be a smooth function
such that $\chi(\eta)=1$ if $\eta\in[-\frac{\eta_0}{4},\frac{\eta_0}{4}]$
and $\chi(\eta)=0$ if $\eta\not\in[-\frac{\eta_0}{2},\frac{\eta_0}{2}]$.
We will use the following estimates to investigate large time
behavior of modulation parameters.
\begin{lemma}
\label{lem:fund-sol}
Let $k\ge0$. Then for every $t\ge0$,
\begin{gather}
  \label{eq:decay1}
\sum_{i=1,2}\|E_i\pd_y^ke^{t\mathcal{A}_*}E_i\|_{B(Y;L^\infty)}
+\|E_1\pd_y^{k+1}e^{t\mathcal{A}_*}E_2\|_{B(Y;L^\infty)}
\lesssim \la t\ra^{-(2k+1)/4}\,,
\\ \label{eq:decay2}
\sum_{i=1,2}\|E_i\pd_y^ke^{t\mathcal{A}_*}E_i\|_{B(Y_1;L^\infty)}
+\|E_1\pd_y^{k+1}e^{t\mathcal{A}_*}E_2\|_{B(Y_1;L^\infty)}
\lesssim \la t\ra^{-(k+1)/2}\,,
\\  \label{eq:decay3}
\|E_2\pd_y^ke^{t\mathcal{A}_*}E_1\|_{B(Y;L^\infty)}
\lesssim \la t\ra^{-(2k+3)/4}\,,\;
\|E_2\pd_y^ke^{t\mathcal{A}_*}E_1\|_{B(Y_1;L^\infty)}
\lesssim \la t\ra^{-(k+2)/2}\,,
\\  \label{eq:etA-1}
 \|\chi(D_y)e^{t\mathcal{A}_*}E_2\|_{B(L^1;L^\infty)}=O(1)\,,\quad
\|(I-\chi(D_y))e^{t\mathcal{A}_*}\|_{B(Y;L^\infty)}=O(e^{-c_1 t})\,,
\end{gather}
where $c_1$ is a positive constant.
\end{lemma}
Lemma~\ref{lem:fund-sol} follows immediately from
Lemma~\ref{lem:fundamental-sol}, \eqref{eq:k-est,high} and
\eqref{eq:fundamental-low}.
\par
Now we will show decay estimates for linearly perturbed equations
of $\pd_tu=\mathcal{A}_*u$. 
Suppose that $\delta_1$, $\delta_2$ and $\kappa$ are positive constants
and that $d_{ij}(\eta)$ and $b_{ij}(t,\eta)$ are
continuous functions satisfying  for $\eta\in[-\eta_0,\eta_0]$ and $t\ge0$,
\begin{equation}
  \label{eq:H}
  \begin{split}
& |d_{11}(\eta)+(\lambda_{2,c_0}+\nu)\eta^2|\le \delta_1|\eta|^3\,,
\quad |d_{12}(\eta)-1|\le \delta_1\eta^2\,,\\
& |d_{21}(\eta)+\lambda_{1,c_0}^2| \le \delta_1\eta^2\,,
\quad |d_{22}(\eta)+(\lambda_{2,c_0}-\nu)\eta^2|\le \delta_1|\eta|^3\,,\\
& |b_{ij}(t,\eta)|\le \delta_2 e^{-\kappa t}\,,
\quad |b_{ij}(t,\eta)-b_{ij}(t,0)|\le \delta_2\eta^2e^{-\kappa t}
\quad\text{for $i$, $j=1$, $2$.}
  \end{split}  \tag{H}
\end{equation}
Let  $\mathcal{A}(t,D_y)=\mathcal{A}_0(D_y)+\mathcal{A}_1(t,D_y)$,
$$\mathcal{A}_0(D_y)=\begin{pmatrix}d_{11}(D_y) & d_{12}(D_y)
\\ d_{21}(D_y)\pd_y^2 & d_{22}(D_y)\end{pmatrix}\,,\quad
\mathcal{A}_1(t,D_y)=\begin{pmatrix} b_{11}(t,D_y)  & b_{12}(t,D_y)
\\ b_{21}(t,D_y)\pd_y & b_{22}(t,D_y)\end{pmatrix}\,,$$
and let $U(t,s)$ a solution operator of
\begin{equation}
\label{eq:cx-linear}
\frac{\pd u}{\pd t}
\begin{pmatrix}u_1 \\ u_2\end{pmatrix}
=\mathcal{A}(t,D_y)\begin{pmatrix}u_1 \\ u_2\end{pmatrix}\,,\quad 
u(s,\cdot)=f\,.
\end{equation}
Then we have the following.
\begin{lemma}
\label{lem:decay-BB}
Let $k\ge0$.
If $\delta_1$ is sufficiently small, then for every 
$t\ge s\ge0$ and $f=(f_1,f_2)^T\in Y\times Y$,
\begin{align}
\label{eq:BB1}
& \|\pd_y^kU(t,s)f\|_{\dot{H}^1\times L^2}\le
C(1+t-s)^{-k/2}(\|\pd_yf_1\|_Y+\|f_2\|_Y)\,,
\\ & \label{eq:BB2}
\|\pd_y^kU(t,s)f\|_{\dot{H}^1\times L^2}\le C(1+t-s)^{-(2k+1)/4}
(\|\pd_yf_1\|_{Y_1}+\|f_2\|_{Y_1})\,,
\\ & \label{eq:BB1'}
\|\pd_y^kU(t,s)\diag(1,\pd_y)f\|_{\dot{H}^1\times L^2}
\le C(1+t-s)^{-(k+1)/2}\|f\|_Y\,,
\\ & \label{eq:BB2'}
\|\pd_y^kU(t,s)\diag(1,\pd_y)f\|_{\dot{H}^1\times L^2}\le
 C(1+t-s)^{-(2k+3)/4} \|f\|_{Y_1}\,,
\\ & \label{eq:BB3}
\|U(t,s)f_2\mathbf{e_2}\|_{L^\infty(\R)}\le C
\left(\|f_2\|_{Y_1}+\|\chi(D_y)f_2\|_{L^1}\right)\,,
\\ & \label{eq:BB4}
  \|U(t,s)\diag(1,\pd_y)f\|_{L^\infty(\R)}\le
C\la t-s\ra^{-1/4}\|f\|_Y\,,
\\ & \label{eq:BB3'}
\left\|\{\mu(t,s)U(t,s)-e^{(t-s)\mathcal{A}_*}\}f_2\mathbf{e_2}\right\|_{L^\infty(\R)}
\le C\la t-s\ra^{-1/4}\|f_2\|_{Y_1}\,,
\\ & \label{eq:BB4'}
  \|\{\mu(t,s)U(t,s)-e^{(t-s)\mathcal{A}_*}\}\diag(1,\pd_y)f\|_{L^\infty(\R)}\le
C\la t-s\ra^{-1/2}\|f\|_Y\,,
\end{align}
where $\mu(t,s)=\exp\left(-\int_s^t b_{22}(s,0)\,ds\right)$,
$C=C(\eta_0)$ and $\limsup_{\eta_0\downarrow0}C(\eta_0)<\infty$.
\end{lemma}
\begin{remark}
Since $\chi(D_y)\wP_1=\chi(D_y)$ is bounded on $L^1(\R)$,
we have $\chi(D_y)f_2\in L^1(\R)$
if $f_2=\wP_1\tilde{f}_2$ and $\tilde{f}_2\in L^1(\R)$.
\end{remark}
\begin{proof}[Proof of Lemma~\ref{lem:decay-BB}]
Let $u=(u_1,u_2)^T$ be a solution of \eqref{eq:cx-linear} and
$v=(\pd_yu_1,u_2)$. Then\eqref{eq:cx-linear} can be read as 
\begin{equation}
\label{eq:cx-linear2}
 \pd_t v=
\begin{pmatrix}d_{11}(D_y)+b_{11}(t,D_y) & (d_{12}(D_y)+b_{12}(t,D_y))\pd_y
\\ d_{21}(D_y)\pd_y+b_{21}(t,D_y) & d_{22}(D_y)+b_{22}(t,D_y)\end{pmatrix}v\,.
\end{equation}
Applying a standard energy method to \eqref{eq:cx-linear2}
as in  \cite[Lemma~4.2]{Miz15}, we have \eqref{eq:BB1} and \eqref{eq:BB2}.
We have \eqref{eq:BB1'} and \eqref{eq:BB2'}
immediately from \eqref{eq:BB1} and \eqref{eq:BB2}.
\par
Next, we will prove \eqref{eq:BB3}.
Let $w=(w_1,w_2)^T$ and  $w(t,y)=\mu(t,s)u(t,y)$. Since $Y\subset L^\infty$
and $0<\inf_{t\ge s\ge0}\mu(t,s)\le \sup_{t\ge s\ge0}\mu(t,s)<\infty$,
it suffices to show
$\|w_1(t)\|_{L^\infty}\lesssim \|f_2\|_{Y_1}+\|\chi(D_y)f_2\|_{L^1}$.
By \eqref{eq:cx-linear},
\begin{align*}
  \pd_tw=\mathcal{A}_0(D_y)w+B(t,D_y)w\,,\quad
B(t,D_y)=\mathcal{A}_1(t,D_y)-b_{22}(t,0)E\,,
\end{align*}
and it follows from \eqref{eq:H} and \eqref{eq:BB1} that
\begin{align*}
  \|\pd_tw_1(t)\|_Y \lesssim & \|\pd_y^2u_1(t)\|_Y+\|u_2(t)\|_Y
+\delta_2e^{-\kappa t}\|w_1(t)\|_Y
\lesssim  \|f\|_Y+\delta_2e^{-\kappa t}\|w_1(t)\|_Y\,.
\end{align*}
Thus by Gronwall's inequality,
\begin{equation}
  \label{eq:w1-est}
\|w_1(\tau)\|_Y \lesssim \la \tau-s\ra \|f\|_Y
\quad\text{for $\tau\ge s$.}
\end{equation}
Since
$|E_1B(t,\eta)|+|E_2\eta^{-1}B(t,\eta)|=O(\delta_2e^{-\kappa t})$,
it follows from Lemma~\ref{lem:fund-sol}, \eqref{eq:H}, \eqref{eq:BB2} and
\eqref{eq:w1-est} that
\begin{align*}
\|w_1(t)\|_{L^\infty} \lesssim &
\|E_1 e^{(t-s)\mathcal{A}_*}u(s)\|_{L^\infty}
+\left\|E_1\int_s^t e^{(t-\tau)\mathcal{A}_*}
\left(\mathcal{A}_0-\mathcal{A}_*+B(\tau)\right)w(\tau)\,d\tau\right\|_{L^\infty}
\\ \lesssim  &
\|f_2\|_{Y_1}+\|\chi(D_y)f_2\|_{L^1}
+ \delta_1\int_s^t \la t-\tau\ra^{-3/4}\|\pd_yw(\tau)\|_{\dot{H}^1\times L^2}\,d\tau
\\ &
+\delta_2\int_s^t \la t-\tau\ra^{-1/4} e^{-\kappa \tau}\|w(\tau)\|_Y\,d\tau
\\ \lesssim & \|f_2\|_{Y_1}+\|\chi(D_y)f_2\|_{L^1}\,.
\end{align*}
\par
Next, we will prove \eqref{eq:BB4}.
Using Lemma~\ref{lem:fund-sol} and \eqref{eq:BB1'}, we have
for $w=(w_1,w_2)^T=\mu(t,s)U(t,s)\diag(1,\pd_y)f$,
\begin{align*}
& \|w_1(t)\|_{L^\infty}\lesssim 
\| e^{(t-s)\mathcal{A}_*}w(0)\|_{L^\infty}
+
\left\|\int_s^t e^{(t-\tau)\mathcal{A}_*}
\left(\mathcal{A}_0-\mathcal{A}_*+B(\tau)\right)w(\tau)\,d\tau\right\|_{L^\infty}
\\  \lesssim & \la t-s\ra^{-1/4}\|f\|_Y
+ \delta_1 \int_s^t \la t-\tau\ra^{-3/4}
\|\pd_yw(\tau)\|_{\dot{H}^1\times L^2}\,d\tau
  \\ & \qquad
 +\delta_2\int_s^t \la t-\tau\ra^{-1/4}e^{-\kappa \tau}\|w(\tau)\|_Y\,d\tau
\\ \lesssim & \la t-s\ra^{-1/4}\|f\|_Y\,.
\end{align*}
We can prove \eqref{eq:BB3'} and \eqref{eq:BB4'} in the same way.
Thus we complete the proof.
\end{proof}
\begin{lemma}
  \label{lem:fun-asymp} Let $f=(f_1,f_2)^T$.
For every $t\ge0$,
\begin{equation}
\label{eq:fund-asymp1}
\left\|e^{t\mathcal{A}_*}f-H_{\lambda_{2,c_0}t}*W_t*f_2\mathbf{e_1}
\right\|_{L^\infty}
\lesssim \la t \ra^{-1/2}\|f\|_{Y_1}\,,
\end{equation}
\begin{equation}
\label{eq:fund-asymp2}
\begin{split}
& \left\|\diag(\pd_y,1)e^{t\mathcal{A}_*}f
-\frac1{2\lambda_{1,c_0}}\begin{pmatrix}
1 & -1 \\  \lambda_{1,c_0} & \lambda_{1,c_0} \end{pmatrix}
\begin{pmatrix}
 H_{\lambda_{2,c_0}t}(\cdot+\lambda_{1,c_0}t)\\ H_{\lambda_{2,c_0}t}(\cdot-\lambda_{1,c_0}t)
\end{pmatrix}
*f_2\right\|_{L^\infty}
\\ \lesssim &
\la t \ra^{-1}\|f\|_{Y_1}\,,  
\end{split}
\end{equation}
\begin{equation}
\label{eq:fund-asymp3}
\left\|e^{t\mathcal{A}_*}\begin{pmatrix}f_1 \\ \pd_yf_2 \end{pmatrix}
-\frac{1}{2\lambda_{1,c_0}}\sum_{\pm}
 H_{\lambda_{2,c_0}t}(\cdot\pm \lambda_{1,c_0}t)*(\lambda_{1,c_0}f_1\pm f_2)\mathbf{e_1}
\right\|_{L^\infty}
\lesssim 
\la t \ra^{-1}\|f\|_{Y_1}\,,  
\end{equation}
where
$H_t(y)=(4\pi t)^{-1/2}\exp(-y^2/4t)$ and
$W_t(y)=(2\lambda_{1,c_0})^{-1} \mathbf{1}_{[-\lambda_{1,c_0}t,\lambda_{1,c_0}t]}(y)$.
\end{lemma}
To investigate the large time behavior of $\gamma(t,y)$, we need the following.
\begin{lemma}
  \label{lem:phase-lim}
Suppose that $f\in L^1(\R_+\times \R)$. Then for any $\delta>0$,
\begin{gather}
\label{eq:lim-in}
\lim_{t\to\infty}\sup_{|y|\le (\lambda_{1,c_0}-\delta)t}
\left|\int_0^t H_{\lambda_{2,c_0}(t-s)}*W_{t-s}*f(s,\cdot)(y)\,ds
-\gamma_*\right|=0\,,
\\ \label{eq:lim-out}
\lim_{t\to\infty}\sup_{|y|\ge (\lambda_{1,c_0}+\delta)t}
\left|\int_0^t H_{\lambda_{2,c_0}(t-s)}*W_{t-s}*f(s,\cdot)(y)\,ds\right|=0\,,
\end{gather}
where $\gamma_*=(2\lambda_{1,c_0})^{-1}\int_0^\infty\int_\R f(s,y)\,dyds$.
\end{lemma}
Lemmas~\ref{lem:fun-asymp} and \ref{lem:phase-lim} can be shown in exactly
the same way as \cite[(2.20)--(2.22) and Lemma~2.4]{Miz19}.
\bigskip

\section{Decomposition of the perturbed line solitary waves}
\label{sec:decomposition}
\subsection{Ansatz for solutions around line solitary waves}
\label{subsec:ansatz}
Let us decompose a solution around $\varphi_c(x-c_0t)$
into a sum of a modulating line solitary wave
and a dispersive part plus a small wave which is caused by amplitude
changes of the line solitary wave:
\begin{equation}
  \label{eq:decomp}
\left\{
  \begin{aligned}
& \Phi(t,x+c_0t,y)=\Phi_{c(t,y)}(z)+U(t,x,y)-\Psi_{c(t,y)}(z_1)\,,
\\ & z=x-\gamma(t,y)\,,\quad z_1=z+\frac{c_0-1}{2}t+h\,.
  \end{aligned}
\right.
\end{equation}
Here $h$ is a large positive constant (see Lemma~\ref{lem:Implicit}).
The modulation parameters $c(t_0,y_0)$ and $\gamma(t_0,y_0)$ denote
the speed and the phase shift of the modulating line solitary wave
$q_{c(t,y)}(x-c_0t-\gamma(t,y))$ along the line $y=y_0$ at the time $t=t_0$,
$U$ is a remainder part which is expected to behave like an
oscillating tail and $\Psi_c$ is an auxiliary function 
such that 
\begin{gather*}
\Psi_c=(\tpsi_c, 0)^T\,,\quad 
\tpsi_c(z_1)=\int_{z_1}^\infty \psi_c(x)\,dx\,,
\quad \psi_c(x)=2\{\beta(c_0)-\beta(c)\}\psi(x)
\end{gather*}
with $\psi\in C_0^\infty(-1,1)$ and $\int_\R \psi(x)\,dx=1$.
Note that
\begin{equation}
  \label{eq:0mean}
\lim_{x\to-\infty}\tpsi_c(x)=
\lim_{x\to-\infty}
\left\{\varphi_c(x)-\varphi_{c_0}(x)\right\}
=-\int_\R\left(q_c-q_{c_0}\right)\,dx\,.
\end{equation}
We need the adjustment by $\Psi_c$ in order to obtain
the energy identity in Section~\ref{sec:energy}.
\par

Substituting \eqref{eq:decomp} into \eqref{eq:BL1}, we have
\begin{equation}
  \label{eq:U}
\pd_tU = (c_0\pd_x+L)U+\ell+N(\Phi_{c(t,y)}-\Psi_{c(t,y)}+U)
-N(\Phi_{c(t,y)}-\Psi_{c(t,y)})\,,
\end{equation}
where $\ell=\ell_1+\ell_2$, $\ell_2=\ell_{21}+\ell_{22}+\ell_{23}$ and
\begin{align*}
\ell_1=&-c_t\pd_c\Phi_{c(t,y)}(z)
+\left\{(c_0+\gamma_t)\pd_x+L\right\}\Phi_{c(t,y)}(z)
+N(\Phi_{c(t,y)}(z))\,,
\end{align*}
\begin{equation}
  \label{eq:defl2}
\left\{ 
\begin{aligned}
& \ell_{21}=c_t\pd_c\Psi_{c(t,y)}(z_1)
-\left(\gamma_t+\frac{c_0+1}{2}\right)\pd_x\Psi_{c(t,y)}(z_1)\,,\quad
\ell_{22}=-L\Psi_{c(t,y)}(z_1)\,,
\\ &
\ell_{23}=N(\Phi_{c(t,y)}(z)-\Psi_{c(t,y)}(z_1))-N(\Phi_{c(t,y)}(z))
=-N'(\Phi_c)\Psi_c\,.
\end{aligned}
\right.
\end{equation}
Note that $N(\Psi_c)=\mathbf{0}$.
Let $\tc(t,y)=c(t,y)-c_0$ and
$N_0(\Phi)=-B_0^{-1}(\phi_2\pd_x^2\phi_1+2\pd_x\phi_1\pd_x\phi_2)\mathbf{e_2}$
with $\Phi=\phi_1\mathbf{e_1}+\phi_2\mathbf{e_2}$,
$\mathbf{e_1}=(1,0)^T$ and $\mathbf{e_2}=(0,1)^T$.
In view of \eqref{eq:stationary} and the fact that $L=L_0-\pd_y^2L_1$,
we have $\ell_1=\ell_{11}+\ell_{12}+\ell_{13}$,
\begin{equation}
  \label{eq:defl1}
\left\{
\begin{aligned}
& \ell_{11}=-c_t\pd_c\Phi_{c(t,y)}(z)+(\gamma_t-\tc)\pd_z\Phi_{c(t,y)}(z)\,,
\quad \ell_{12}=-\pd_y^2L_1\Phi_{c(t,y)}(z_1)\,,
\\ & \ell_{13}=N(\Phi_{c(t,y)}(z))-N_0(\Phi_{c(t,y)}(z))\,.
\end{aligned}
\right.
\end{equation}
\par

Now we introduce a symplectic orthogonality condition to
fix the decomposition \eqref{eq:decomp}.
Since the adjoint resonant modes are exponentially increasing
as $x\to\infty$, we further decompose $U$ into a small solution of
\eqref{eq:BL1} and an exponentially localized part following the idea of
\cite{Miz09} and \cite{MPQ13,MT11,Miz18} in order to use 
exponential linear stability in \cite{Miz-Sh17}.
\par
Let $\Phi(t,x,y)$ be a solution of \eqref{eq:BL1} with $\Phi(0,x,y)=\Phi_{c_0}(x)+U_0(x,y)$
and let $U_1$ be a solution of 
\begin{equation}
  \label{eq:U1}
\pd_tU_1=(c_0\pd_x+L)U_1+N(U_1)\,, \quad U_1(0,x,y)=U_0(x,y)\,,
\end{equation}
and 
\begin{equation}
  \label{eq:defU2}
U_2(t,x,y)=U(t,x,y)-U_1(t,x,y)\,,\quad \wU_2(t,z,y)=U_2(t,x,y)\,. 
\end{equation}
By \eqref{eq:U}, \eqref{eq:U1} and \eqref{eq:defU2},
\begin{equation}
  \label{eq:U2}
\pd_tU_2=\mL_{c_0}U_2+\ell+N_1+N_2+N_3\,,\quad U_2(0)=0\,,
\end{equation}
where $N_{1,0}=N(\Phi_{c(t,y)}(z)-\Psi_{c(t,y)}(z_1)+U)
-N(\Phi_{c(t,y)}(z)-\Psi_{c(t,y)}(z_1))-N(U)$ and
\begin{align*}
& N_1=N_{1,0}-V_{c_0}U_2-N_3\,,\quad N_2=N(U)-N(U_1)\,,
\\ &
N_3=N'(\Phi_{c(t,y)}(z)-\Psi_{c(t,y)}(z_1))U_1\,.
\end{align*}
Lemma~\ref{lem:Phi-U1} below implies that 
$U_2(t,x,y)$, as well as $\wU_2(t,z,y)$, are exponentially localized
as $x\to\infty$ provided the phase shift $\gamma(t,y)$ is uniformly bounded.
\begin{lemma}
  \label{lem:Phi-U1}
  Let $\a\in (0,\a_{c_0})$ and $U_0\in \bE$. Suppose that $\Phi(t)$ is a
  solution of   \eqref{eq:BL1} with $\Phi(0)=\Phi_{c_0}+U_0$ 
 and that $U_1(t)$ is a solution of \eqref{eq:U1}. Then
\begin{equation}
  \label{eq:W-conti}
W(t,x,y):=\Phi(t,x+c_0t,y)-\Phi_{c_0}(x)-U_1(t,x,y)\in C([0,\infty);\bX_1)\,.
\end{equation}
\end{lemma}

\begin{proof}[Proof of Lemma~\ref{lem:Phi-U1}]
First, we remark that $c_0\pd_x+L$ generates $C^0$-semigroup on $\bE$
and that $U_1(t)$, $W(t)\in C([0,\infty);\bE)$ follows from a standard
argument.
\par
Let
\begin{gather*}
  \mathcal{E}(\phi_1,\phi_2)
=\frac{1}{2}\left\{\phi_2^2+b|\nabla \phi_2|^2+|\nabla \phi_1|^2
+a(\Delta\phi_1)^2\right\}\,,
\\
\mathcal{F}_{quad}(\phi_1,\phi_2)
=\phi_2B^{-1}A\nabla \phi_1+a\Delta \phi_1\nabla\phi_2\,.
\end{gather*}
By \eqref{eq:BL1}, \eqref{eq:stationary}, \eqref{eq:U1} and the fact that
$N$ is quadratic nonlinearity,
\begin{equation}
  \label{eq:W}
  \pd_tW=(c_0\pd_x+L)W+G\,,\quad G=V_{c_0}(U_1+W)+N(U_1+W)-N(U_1)\,,
\end{equation}
and
\begin{equation}
  \label{eq:linearE-M}
\frac{d}{dt}\mathcal{E}(W(t))=
c_0\pd_x\mathcal{E}(W(t))+\nabla\cdot\mathcal{F}_{quad}(W(t))
+w_2g+b\nabla w_2\cdot\nabla g\,,
\end{equation}
where $W=(w_1,w_2)$ and $g=G\cdot\mathbf{e_2}$.
\par
Let $p_n(x)=e^{2\a n}(1+\tanh\a(x-n))$ and
$$E_n(W)=\int_{\R^2}p_n(x)\mathcal{E}(W)(t,x,y)\,dxdy\,.$$
By \eqref{eq:linearE-M},
\begin{align*}
  \frac{d}{dt}E_n(W(t))=& 
-\int_{\R^2} p_n'(x)\left\{c_0\mathcal{E}(W)(t,x,y)+\mathbf{e_1}\cdot\mF_{quad}(W)(t,x,y)\right\}\,dxdy
\\ &+
\int_{\R^2} p_n(x)\{w_2g+b\nabla w_2\cdot\nabla g\}\,dxdy\,.
\end{align*}
We remark that $0<p_n'(x)\le 2\a p_n(x)\le 4\a e^{2\a x}$.
By Claim~\ref{cl:ker-B},
$$\left|\int_{\R^2} p_n'(x)\mathbf{e_1}\cdot\mF_{quad}(W)(t,x,y)\,dxdy\right|
\lesssim \int_{\R^2}p_n'(x)\mathcal{W}(t,x,y)\,dxdy\,,$$
\begin{align*}
& \left|  \int_{\R^2} p_n(x)\{w_2g+b\nabla w_2\cdot\nabla g\}\,dxdy\right|
\\ \lesssim &
\|U_1(t)\|_{\bE}E_n(W)^{1/2}+(\|U_1+W\|_{\bE}+\|U_1\|_{\bE})E_n(W)
\\ \lesssim & 
E(U_0)+(1+E(U_0)+\|W(t)\|_{\bE})E_n(W(t))\,.
\end{align*}
In the last line, we use the energy identity
\begin{equation}
  \label{eq:energyU1}
 E(U_1(t))=E(U_0)\,.  
\end{equation}
Combining the above, we have
$$\frac{d}{dt}E_n(W(t)) \lesssim E(U_0)+(1+E(U_0)+\|W(t)\|_{\bE})E_n(W(t))\,.$$
Using Gronwall's inequality, we see that for any $T>0$, there exists a $C>0$ such that
\begin{equation*}
  E_n(t) \le CE(U_0)\quad\text{for every $t\in[0,T]$ and $n\in\N$.}
\end{equation*}
By passing to the limit $n\to\infty$, we have $\sup_{t\in[0,T]}\|W(t)\|_{\bX_1} \lesssim \|U_0\|_{\bE}$ since $0<p_n(x)\uparrow 2e^{2\a x}$.
Moreover, we can prove that $\|W(t)\|_{\bX_1}$ is continuous.
Since $W(t)\in C([0,\infty;\bE)$ and $W(t)$ is  weakly continuous in $\bX_1$,
we have \eqref{eq:W-conti}. Thus we complete the proof.
\end{proof}

\subsection{The  orthogonality condition}
\label{subsec:orth}
We impose a secular term condition for $\wU_2(t,z,y)$ to fix 
the decomposition \eqref{eq:decomp} with \eqref{eq:defU2}. 
\begin{equation}
  \label{eq:orth}
 \int_{\R^2}
\wU_2(t,z,y)\cdot g_k^*(z,\eta,c(t,y))e^{-iy\eta}\,dzdy=0
\quad\text{for $\eta\in(-\eta_0,\eta_0)$ and $k=1$, $2$,}
\end{equation}
where $\eta_0$ is a sufficiently small positive number.
\par

Now we introduce functionals to prove the existence of
the decomposition \eqref{eq:decomp} that satisfies \eqref{eq:defU2} and
\eqref{eq:orth}. 
For $U\in \bX$ and $\gamma$, $\tc\in Y$ and $h\ge 0$, let 
$c(y)=c_0+\tc(y)$ and 
\begin{align*}
F_k[U,\tc,\gamma,h](\eta):=\mathbf{1}_{[-\eta_0,\eta_0]}&(\eta)
\lim_{M\to\infty}\int_{-M}^M\int_\R
\bigl\{U(x,y)+\Phi_{c_0}(x)- \Phi_{c(y)}(x-\gamma(y))
\\ & +\Psi_{c(y)}(x-\gamma(y)+h)\bigr\}\cdot
g_k^*(x-\gamma(y),\eta,c(y))e^{-iy\eta}\,dxdy\,.
\end{align*}
To begin with, we will show that $F=(F_1,F_2)$ maps
$L^2_\a(\R^2)\times Y\times Y\times \R$ into $Z\times Z$.
\begin{lemma}
\label{lem:F_k}
Let $\a\in(0,\a_{c_0})$,  $U\in L^2_\a(\R^2)$, $\tc$, $\gamma\in Y$ and $h\ge 0$.
Then there exists a $\delta>0$ such that if 
$\|\tc\|_Y+\|\gamma\|_Y\le \delta$, then
$F_k[u,\tc,\gamma,h]\in Z$ for $k=1$, $2$.
\end{lemma}
\begin{proof}
Let $U\in C_0^\infty(\R^2;\R^2)$ and
\begin{align*}
& \Theta(x,y)=\Phi_{c(y)}(x-\gamma(y))-\Phi_{c_0}(x)
-\Psi_{c(y)}(x-\gamma(y)+h)\,,\\
& \Theta_0(x,y)=\{\pd_c\Phi_{c_0}(x)-\pd_c\Psi_{c_0}(x+h)\}\tc(y)
-\pd_x\Phi_{c_0}(x)\gamma(y)\,,\\
& \Theta_1(x,y)=\Theta(x,y)-\Theta_0(x,y)\,,\quad
\Theta^*(x,y)=g_k^*(x-\gamma(y),\eta,c(y))-g_k^*(x,\eta,c_0)\,.
\end{align*}
Then
\begin{align*}
\int_{\R^2}\left\{U(x,y)-\Theta(x,y)\right\}\cdot
g_k^*(x-\gamma(y),\eta,c(y))e^{-iy\eta}\,dxdy=\sum_{j=1}^4 I_j(\eta)\,,
\end{align*}
where
\begin{align*}
I_1=& \int_{\R^2}U(x,y)\cdot g_k^*(x,\eta,c_0)e^{-iy\eta}\,dxdy\,,\\
I_2=& -\int_{\R^2}\Theta_0(x,y)\cdot g_k^*(x,\eta,c_0)e^{-iy\eta}\,dxdy\,,\\
I_3=& -\int_{\R^2}\Theta_1(x,y)\cdot g_k^*(x,\eta,c_0)e^{-iy\eta}\,dxdy\,,\\
I_4=& \int_{\R^2}\{U(x,y)-\Theta(x,y)\}\Theta^*(x,y)e^{-iy\eta}\,dxdy\,.\\
\end{align*}
By \eqref{eq:gk*-approx},
\begin{equation}
  \label{eq:gk-bound}
\sup_{|c-c_0|\le \delta\,,\, |\eta-\eta_0|\le\delta}
\left\|\pd_c^j\pd_x^kg_k^*(\cdot,\eta,c)\right\|_{L^2_{-\a}(\R)}<\infty
\enskip\text{ for $j$, $k\ge0$,}
\end{equation}
and it follows from the Plancherel theorem and \eqref{eq:gk-bound} that
\begin{align*}
 \int_{-\eta_0}^{\eta_0}|I_1(\eta)|^2d\eta \lesssim \|U\|_{L^2_\a(\R^2)}^2\,,
\enskip
 \int_{-\eta_0}^{\eta_0}|I_2(\eta)|^2d\eta \lesssim 
\|\tc\|_Y^2+\|\gamma\|_Y^2\,.
\end{align*}
Since $\sup_y(|\tc(y)|+|\gamma(y)|)\lesssim \|\tc\|_Y+\|\gamma\|_Y$, we have
\begin{gather*}
\|\Theta\|_{L^2_\a(\R^2)}+\|\Theta_0\|_{L^2_\a(\R^2)}
\le C(\|\tc\|_Y+\|\gamma\|_Y)\,,\\
\|e^{\a x}\Theta_1(x,y)\|_{L^1(\R^2)}\le C(\|\tc\|_Y+\|\gamma\|_Y)^2\,,\\
\|\Theta^*(x,y)\|_{L^2_{-\a}(\R^2)}\le C(\|\tc\|_Y+\|\gamma\|_Y)\,,
\end{gather*}
where $C$ is a positive constant depending only on $\delta$.
\par
Combining the above, we obtain
$$\sup_{-\eta_0\le \eta\le \eta_0}(|I_3(\eta)|+|I_4(\eta)|)
\lesssim \|U\|_{L^2_\a(\R^2)}(\|\tc\|_Y+\|\gamma\|_Y)
+(\|\tc\|_Y+\|\gamma\|_Y)^2\,.$$
Since $C_0^\infty(\R^2)$ is dense in $L^2_\a(\R^2)$, it follows that for any
$u\in L^2_\a(\R^2)$,
$$\mathbf{1}_{[-\eta_0,\eta_0]}(I_1+I_2)\in Z\,,\quad
\mathbf{1}_{[-\eta_0,\eta_0]}(I_3+I_4)\in Z_1\subset Z\,.$$
Thus we complete the proof.
\end{proof}

Next, we will prove the  existence of parameters $c$ and $\gamma$ that satisfy
\eqref{eq:decomp} and \eqref{eq:orth}.
\begin{lemma}
  \label{lem:Implicit}
Let $\a\in(0,\a_{c_0})$.
There exist positive constants $\delta_0$, $\delta_1$, $h_0$ and $C$ such that
if $\|U\|_{L^2_\a}<\delta_0$ and $h\ge h_0$, then there exists a unique
$(\tc,\gamma)$ with $c=c_0+\tc$ satisfying
\begin{gather}
\label{eq:imp1}
  \|\tc\|_{Y}+\|\gamma\|_{Y} <\delta_1\,,\\
\label{eq:imp2}
F_1[u,\tc,\gamma,h]=F_2[u,\tc,\gamma,h]=0\,.
\end{gather}
Moreover, the mapping $\{U\in L^2_\a(\R^2)\mid \|U\|_{L^2_\a(\R^2)}<\delta_0\}
\ni U\mapsto (\tc,\gamma)=:\Omega(U)\in Y\times Y$ is $C^1$.
\end{lemma}
\begin{proof}
We have $(F_1,F_2)\in C^1(L^2_\a(\R^2)\times Y\times Y\times\R;Z\times Z)$
and  for $\tc$, $\tilde{\gamma}\in Y$, 
$$D_{(c,\gamma)}(F_1,F_2)(0,0,0,h)
\begin{pmatrix}\tc \\ \tilde{\gamma}\end{pmatrix}
=\sqrt{2\pi}\begin{pmatrix}f_{11} & f_{12} \\ f_{21} & f_{22}\end{pmatrix}
\begin{pmatrix}
\mF_y\tc \\ \mF_y\tilde{\gamma}\end{pmatrix}\,,
$$
where 
\begin{align*}
& f_{k1}=-\int_\R\left\{\pd_c\Phi_{c_0}(x)-\pd_c\Psi_{c_0}(x+h)\right\}
\cdot g_k^*(x,\eta,c_0)\,dx\,,\\
& f_{k2}=\int_\R \pd_x\Phi_{c_0}(x)\cdot g_k^*(x,\eta,c_0)\,dx\,,
\end{align*}
and by \eqref{eq:zetas}, \eqref{eq:zeta*s}, \eqref{eq:zeta1-zeta1*},
\eqref{eq:zeta2-zeta1*} and \eqref{eq:gk*-approx},
\begin{align*}
& f_{11}=\beta_2(c_0)+O(\eta_0^2)+O(e^{-h\a})\,,
\quad f_{12}=\beta_1(c_0)+O(\eta_0^2)\,,
\\ &
f_{21}=\beta_1(c_0)+O(\eta_0^2)+O(e^{-h\a})\,,
\quad f_{22}=O(\eta_0^2)\,.
\end{align*}
Suppose $\eta_0$ and $e^{-h\a}$ are sufficiently small.
Then $D_{(c,\gamma)}(F_1,F_2)(0,0,0,h)\in B(Y\times Y,Z\times Z)$
has a bounded inverse and by the implicit function theorem,
there exists a unique $(\tc,\gamma)\in Y\times Y$ 
satisfying \eqref{eq:imp1} and \eqref{eq:imp2} for any $U$ satisfying
$\|U\|_{L^2_\a(\R^2)}<\delta_0$ and the mapping $(\tc,\gamma)=\Omega(U)$
is $C^1$.
\end{proof}
\begin{remark}
  \label{rem:decomp}
  If $U_0\in \bE$ and $\Phi(0,x,y)=\Phi_{c_0}(x)+U_0(x,y)$,
  then a solution of \eqref{eq:BL1} is in the class
\begin{equation}
  \label{eq:lwp}
\widetilde{U}(t,x,y):=\Phi(t,x+c_0t,y)-\Phi_{c_0}(x)
\in C(\R;\bE)\cap C^1(\R;H^1(\R^2)\times L^2(\R^2))\,.
\end{equation}
Combining the above with \eqref{eq:W-conti}, we have
\begin{equation}
  \label{eq:lwpX}
  W(t)\in C([0,\infty);\bX_1)\cap C^1((0,\infty);\bX)\,,
\end{equation}
and it follows from Lemma~\ref{lem:Implicit} that
$(\tc(t),\gamma(t))=\Omega(W(t))\in C^1$ as long as
$\|W(t)\|_{L^2_\a}$ remains small. Since $W(0)=0$, there exists a $T>0$ such
that the decomposition \eqref{eq:decomp} satisfying \eqref{eq:defU2} and
\eqref{eq:orth} persists for $t\in[0,T]$ and 
\begin{gather}
  \label{eq:mod-init}
\tc(0)=\gamma(0)=0\,,\quad \wU_2(0)=0\,\,,
\\ (\tc(t),\gamma(t))\in  C([0,T];Y\times Y)\cap C^1((0,T);Y\times Y)\,.
\end{gather}
\end{remark}

By a standard argument, we have the following continuation principle for
the decomposition \eqref{eq:decomp} satisfying \eqref{eq:defU2}
and \eqref{eq:orth}.
\begin{proposition}
\label{prop:continuation}
Let $\a$, $\delta_0$ and $h$ be the same as in Lemma~\ref{lem:Implicit}
and let $\Phi(t)$ and $U_1(t)$ be as in Lemma~\ref{lem:Phi-U1}.
Then there exists a constant $\delta_2>0$ such that if \eqref{eq:decomp},
\eqref{eq:defU2} and \eqref{eq:orth} hold for $t\in[0,T)$ and
\begin{gather}
\label{eq:C1cg}
(\tc,\gamma)\in C([0,T);Y\times Y)\cap C^1((0,T);Y\times Y)\,,\\
\label{eq:bd-vc}
\sup_{t\in[0,T)}\|U(t)\|_{L^2_\a}< \frac{\delta_0}{2}\,,
\quad \sup_{t\in[0,T)}\|\tc(t)\|_Y<\delta_2\,,\quad
\sup_{t\in[0,T)}\|\gamma(t)\|_Y<\infty\,,
\end{gather}
then either $T=\infty$ or $T$ is not the maximal time of
the decomposition \eqref{eq:decomp} satisfying \eqref{eq:defU2},
\eqref{eq:orth},  \eqref{eq:C1cg} and \eqref{eq:bd-vc}.
\end{proposition}
\bigskip

\section{Modulation equations}
\label{sec:modulation}
In this section, we will derive a system of PDEs which describe the motion
of  $c(t,y)$ and $\gamma(t,y)$.

Let $\wU_1(t,z,y)=U_1(t,x,y)$, $\wU(t,z,y)=U(t,x,y)$ and
$\tau$ be a shift operator defined by $(\tau_\gamma f)(x,y)=f(x+\gamma,y)$.
Then \eqref{eq:U1} and \eqref{eq:U2} are transformed into
\begin{gather}
  \label{eq:wU1}
  \pd_t\wU_1=\{(c_0+\gamma_t)\pd_z+L\}\wU_1+N(U_1)
+\tau_{\gamma(t,y)}[L,\tau_{-\gamma(t,y)}]\wU_1\,,
\\
  \label{eq:wU2}
\pd_t\wU_2=\mL_{c(t,y)}\wU_2+\ell+\sum_{j=1}^4 \widetilde{N}_j\,,
\quad \wU_2(0)=0\,,
\end{gather}
where
\begin{align*}
& \widetilde{N}_1=N_{1,0}-V_{c(t,y)}\wU\,,\quad
\widetilde{N}_2=N_2 \,,\quad \widetilde{N}_3=V_{c(t,y)}\wU_1
\\ & 
\widetilde{N}_4=(\gamma_t-\tc)\pd_z\wU_2
+\tau_{\gamma(t,y)}[L,\tau_{-\gamma(t,y)}]\wU_2\,.
\end{align*}

Suppose that the decomposition \eqref{eq:decomp}, \eqref{eq:defU2}
satisfying \eqref{eq:orth} persists for $t\in[0,T]$.
Differentiating \eqref{eq:orth} with respect to $t$, we have
in $L^2(-\eta_0,\eta_0)$,
\begin{equation}
  \label{eq:modeq}
\begin{split}
& \frac{d}{dt}\int_{\R^2}\wU_2(t,z,y)\cdot g_k^*(z,\eta,c(t,y))e^{-iy\eta}\,dzdy
\\=& \int_{\R^2} \ell\cdot g_k^*(z,\eta,c(t,y))e^{-iy\eta}\,dzdy
+\sum_{j=0}^5II_k^j(t,\eta)\,,
\end{split}  
\end{equation}
where
\begin{align}
& II_k^0(t,\eta)=\int_{\R^2} \wU_2\cdot\mL_{c(t,y)}^*
\left(g_k^*(z,\eta,c(t,y))e^{-iy\eta}\right)\,dzdy\,,
\\ &
II_k^j(t,\eta)=\int_{\R^2} \widetilde{N}_j
\cdot g_k^*(z,\eta,c(t,y))e^{-iy\eta}\,dzdy
\quad\text{for $j=1$, $2$, $3$,}
\\ &
II_k^4(t,\eta)=\int_{\R^2} \wU_2(t,z,y)\cdot 
\{c_t\pd_c-(\gamma_t-\tc)\pd_z\}g_k^*(z,\eta,c(t,y))e^{-iy\eta}\,dzdy\,,
\\ & 
II_k^5(t,\eta)=\int_{\R^2} \wU_2(t,z,y)\cdot 
[\tau_{\gamma(t,y)},L^*]\tau_{-\gamma(t,y)}
g_k^*(z,\eta,c(t,y))e^{-iy\eta}\,dzdy\,.
\end{align}
The modulation equations of $c(t,y)$ and $\gamma(t,y)$ can be obtained by
taking the inverse Fourier transform of \eqref{eq:modeq}.
\subsection{The dominant part of modulation equations}
In view of \eqref{eq:defl1}, we have
\begin{align}
\label{eq:l13exp}
\ell_{13} =&
-B^{-1}\{2\pd_y\varphi_{c(t,y)}(z)\pd_yr_{c(t,y)}(z)
+r_{c(t,y)}(z)\pd_y^2\varphi_{c(t,y)})(z)\}\mathbf{e_2}
\\ & \notag
-bB_0^{-1}B^{-1}\pd_y^2\{2q_{c(t,y)}(z)r_{c(t,y)}'(z)+r_{c(t,y)}(z)q_{c(t,y)}'(z)\}
\mathbf{e_2}\,,
\end{align}
and $\ell_1=\ell_{11}+\tilde{\ell}_{12}+\pd_y^2\tilde{\ell}_{1r}$,
\begin{align} 
 \tilde{\ell}_{12}=& -\pd_y^2L_1(0)\Phi_{c(t,y)}(z) \notag
\\ & \notag
 -B_0^{-1}\{2\pd_y\varphi_{c(t,y)}(z)\pd_yr_{c(t,y)}(z)
+r_{c(t,y)}(z)\pd_y^2\varphi_{c(t,y)})(z)\}\mathbf{e_2}
\\ & \notag
-bB_0^{-2}\pd_y^2\{2q_{c(t,y)}(z)r_{c(t,y)}'(z)
+r_{c(t,y)}(z)q_{c(t,y)}'(z)\}\mathbf{e_2} \,,
\\ & \label{eq:l1r}
\tilde{\ell}_{1r}=\pd_y^2L_2(D_y)\Phi_{c(t,y)}(z)
+bB^{-1}\left(\tilde{\ell}_{12}+\pd_y^2L_1(0)\Phi_{c(t,y)}(z)\right)\,,
\end{align}
where $L_2(\eta)=\eta^{-2}(L_1(\eta)-L_1(0))$ for $\eta\ne0$ and
$L_2(0)=\frac12\frac{d^2}{d\eta^2}L_1(\eta)|_{\eta=0}$.
Since $g_k^*(z,0,c)=\zeta_{k,c}^*(z)$, the leading terms of
\begin{equation}
\label{eq:modeql1}
\frac{1}{\sqrt{2\pi}}
\mF_\eta^{-1}\int_{\R^2} \ell_1\cdot g_k^*(z,\eta,c(t,y))e^{-iy\eta}\,dzdy
\end{equation}
are $G_k(t,y):=\int_\R (\ell_{11}+\tilde{\ell}_{12})\cdot
\zeta_{k,c(t,y)}^*(z)\,dz$.
\begin{lemma}
  \label{lem:G}
  \begin{equation}
\label{eq:G}
\begin{split}
\begin{pmatrix}G_1 \\ G_2\end{pmatrix}=&
\wP_1\left\{B_1(c)\begin{pmatrix}\gamma_t-\tc \\  c_t\end{pmatrix}
+M(c)\begin{pmatrix}\gamma_{yy}\\  c_{yy}\end{pmatrix}\right\}
\\ & +\wP_1
\begin{pmatrix}
2m_{13}(c)c_y\gamma_y +m_{14}(c)(c_y)^2+m_{15}(c)(\gamma_y)^2
\\ 2m_{23}(c)c_y\gamma_y +m_{24}(c)(c_y)^2 
\end{pmatrix}\,,
\end{split}
  \end{equation}
where $B_1(c)=\beta_1(c)I+\beta_2(c)E_{12}$,
$$M(c)=\sum_{j,k=1,2}m_{kj}(c)E_{kj}\,,\quad
m_{kj}(c)=
\begin{cases}
\la \mL_{1,c}(0)\zeta_{j,c},\zeta_{k,c}^*\ra \quad & \text{if $j=1$, $2$,}
\\ m_{kj}(c)=\la \Phi_{j,c},\zeta_{k,c}^*\ra & \quad \text{if $3\le j\le 5$,}
\end{cases}$$
\begin{align*}
& \Phi_{3,c}=\pd_c\{\mL_{1,c}(0)\zeta_{1,c}\}
+B_0^{-1}r_c'\pd_c\varphi_c\mathbf{e_2}\,,\enskip
\Phi_{4,c}= \pd_c\{\mL_{1,c}(0)\zeta_{2,c}\}
-B_0^{-1}\pd_cr_c\pd_c\varphi_c\mathbf{e_2}\,,
\\ &
\Phi_{5,c}=- \pd_z\{\mL_{1,c}(0)\zeta_{1,c}\}-B_0^{-1}r_c' q_c\mathbf{e_2}\,.
\end{align*}
\end{lemma}
\begin{proof}[Proof of Lemma~\ref{lem:G}]
By \eqref{eq:zeta1-zeta1*} and \eqref{eq:zeta2-zeta1*},
\begin{align*}
\begin{pmatrix}
\int_\R \ell_{11}\cdot \zeta_{1,c(t,y)}^*(z)\,dz
\\
\int_\R \ell_{11}\cdot \zeta_{2,c(t,y)}^*(z)\,dz   
\end{pmatrix}
=B_1(c)\begin{pmatrix}\gamma_t-\tc \\ c_t \end{pmatrix}\,.
\end{align*}
By a straightforward computation, we have
$\tilde{\ell}_{12}=\tilde{\ell}_{12L}+\tilde{\ell}_{12N}$,
\begin{align}
\label{eq:l12L}
\tilde{\ell}_{12L}=&  
\gamma_{yy}\mL_{1,c}(0)\zeta_{1,c}+c_{yy}\mL_{1,c}(0)\zeta_{2,c}\,,
\\ \label{eq:l12N}
\tilde{\ell}_{12N}= &
2c_y\gamma_y\Phi_{3,c}+(c_y)^2\Phi_{4,c}+(\gamma_y)^2\Phi_{5,c}\,,
\end{align}
\begin{align*}
\mL_{1,c}\pd_z\Phi_c=& B_0^{-1}\{(2a-bc^2)q_c''-q_c+r_cq_c\}\mathbf{e_2}\,,
\\
\mL_{1,c}\pd_c\Phi_c=& B_0^{-1}\left[\pd_c\{(2a-bc^2)q_c'\}
-\pd_c\varphi_c+r_c\pd_c\varphi_c\right]\mathbf{e_2}\,,
\end{align*}
and
\begin{align*}
& \int_\R \tilde{\ell}_{12L}\cdot \zeta_{k,c(t,y)}^*(z)\,dz
= m_{k1}(c)\gamma_{yy}+m_{k2}(c)c_{yy}\,,
\\ &
\int_\R \tilde{\ell}_{12N}\zeta_{k,c(t,y)}^*(z)\,dz
=2m_{k3}(c)c_y\gamma_y+m_{k4}(c)(c_y)^2+m_{k5}(c)(\gamma_y)^2\,.
\end{align*}
We have $m_{25}(c)= c\la (bc^2-2a)q_c'''+q_c'+3cq_cq_c', q_c\ra=0$ by parity.
By \eqref{eq:L1czs},
\begin{equation}
  \label{eq:m11}
m_{11}(c)=m_{22}(c)<0\,,\quad m_{21}(c)<0\,.
\end{equation}
Thus we complete the proof.
\end{proof}
\par

\par
The equation $(G_1, G_2)\simeq(0,0)$ is a dissipative wave equation.
Indeed, let
$a_{1j}(c)=\{\beta_2(c)m_{2j}(c)-\beta_1(c)m_{1j}(c)\}/\beta_1(c)^2$,
$a_{2j}(c)=-m_{2j}(c)/\beta_1(c)$ and
$$A(c,\eta)=
 \begin{pmatrix}0 & 1 \\ 0 & 0 \end{pmatrix}
- \begin{pmatrix}a_{11}(c) & a_{12}(c) \\ a_{21}(c) & a_{22}(c)
\end{pmatrix}\eta^2\mathbf{1}_{[-\eta_0,\eta_0]}(\eta)\,.$$
Then by \eqref{eq:lambda1} and \eqref{eq:lambda2},
\begin{align*}
\det A(c,\eta)=& a_{21}(c)\eta^2+
\{a_{11}(c)a_{22}(c)-a_{12}(c)a_{21}(c)\}\eta^4
= \lambda_{1,c}^2\eta^2+O(\eta^4)\,,  
\\ \operatorname{tr} A(c,\eta)=& \beta_1(c)^{-2}
\left\{\beta_1(c)(m_{11}(c)+m_{22}(c))-\beta_2(c)m_{21}(c)\right\}\eta^2
= -2\lambda_{2,c}\eta^2<0\,.
\end{align*}
Solutions of the linearized equation
\begin{equation*}
  \frac{d}{dt}
  \begin{pmatrix}   \gamma \\ \tc   \end{pmatrix}
=A(c_0,D_y)  \begin{pmatrix}   \gamma \\ \tc   \end{pmatrix}
\end{equation*}
decays like 
$\|\pd_y^{k+1}\gamma\|_Y+\|\pd_y^k\tc\|_Y \lesssim \la t\ra^{-(2k+1)/4}$
for $k\ge0$ (see Lemma~\ref{lem:fund-sol}).
Therefore we expect that the modulation equations for
$c(t,y)$ and $\gamma(t,y)$ can be reduced to
\begin{equation}
  \label{eq:approx-modeq}
\begin{split}
\begin{pmatrix} \gamma_t \\ c_t \end{pmatrix}
\simeq& 
B_1(c)^{-1}\left\{(E_{12}+M(c)\pd_y^2)\begin{pmatrix}\gamma \\ \tc\end{pmatrix}
+
\begin{pmatrix}
m_{15}(c)(\gamma_y)^2 \\ 2m_{23}(c)c_y\gamma_y
\end{pmatrix}\right\}
\\ \simeq &
A(c,D_y)\begin{pmatrix} \gamma \\ \tc\end{pmatrix}
+
\begin{pmatrix} a_{15}(c)(\gamma_y)^2  \\ 2a_{23}(c)c_y\gamma_y
\end{pmatrix}                                                       
\end{split}  
\end{equation}
To estimate $\|\tc(t)\|_Y$ by using Lemma~\ref{lem:fund-sol},
we need to translate the nonlinear term $2m_{23}(c)c_y\gamma_y$ into
a divergent form because $\|m_{23}(c)c_y\gamma_y\|_{L^1}=O(t^{-1})$.
Let 
\begin{gather}
\ta_{21}(c)=a_{21}(c_0)
\exp\left(\int_{c_0}^c\frac{2a_{23}(s)}{a_{21}(s)}\,ds\right)\,,
\quad
\rho(c)=\int_{c_0}^c \frac{\ta_{21}(s)}{a_{21}(s)}\,ds\,,
\\
\label{eq:rhoc}
b=\wP_1\{\rho(c)-\rho(c_0)\}\,.
\end{gather}
We remark that $\rho'(c_0)= 1$ and $b\simeq \tc$
(see Claim~\ref{cl:b-capprox} in Appendix~\ref{ap:r}).

By \eqref{eq:approx-modeq} and \eqref{eq:rhoc},
\begin{align*}
\gamma_t\simeq &
\tc +a_{11}(c)\gamma_{yy}+a_{12}(c)c_{yy}+a_{15}(c)(\gamma_y)^2 
\\ =&
b+a_{11}(c_0)\gamma_{yy}+a_{12}(c_0)b_{yy}
+\wP_1\left(p_1b^2+a_{15}(c)(\gamma_y)^2\right)+\wP_1 R^{G,1}_1\,,
\end{align*}
where $p_1=-\frac12\rho''(c_0)$ and
$$R^{G,1}_1=(a_{11}(c)-a_{11}(c_0))\gamma_{yy}+a_{12}(c)c_{yy}-a_{12}(c_0)b_{yy}
+\tc-b-p_1b^2\,.$$
Let $\ta_{22}(c)=\rho'(c)a_{22}(c)$.
Since $\ta_{21}(c)=\rho'(c)a_{21}(c)$ and
$\ta_{21}'(c)=2\rho'(c)a_{23}(c)$,
\begin{align*}
b_t= \wP_1 \rho'(c)c_t
\simeq &  \wP_1(\ta_{21}(c)\gamma_{yy}
+\ta_{21}'(c)c_y\gamma_y+\ta_{22}(c)c_{yy})
\\=& a_{21}(c_0)\gamma_{yy}+a_{22}(c_0)b_{yy}
+\wP_1(p_2(c)\gamma_y)_y
+\wP_1 R^{G,1}_2\,,
\end{align*}
where $p_2(c)=\ta_{21}(c)-\ta_{21}(c_0)=O(\tc)$ and
$R^{G,1}_2= \ta_{22}(c)c_{yy}-\ta_{22}(c_0)b_{yy}$.
\par

Now we will translate the principal part of $(G_1, G_2)$ in terms of
$b$ and $\gamma$. Let $B_2(c)=\diag(1,\rho'(c))$,
$B_3(c)=B_2(c)B_1(c)^{-1}$, $B_4(c)=\wP_1B_2(c)\wP_1(\wP_1B_1(c)\wP_1)^{-1}\wP_1$,
\begin{align}
\label{eq:defdB}
& \delta B(c):= \wP_1B_3(c)-B_4(c)
\\=& \notag
\wP_1\left(B_3(c)-B_3(c_0)\right)(I-\wP_1)
+\wP_1\left(B_2(c)-B_2(c_0)\right)(I-\wP_1)B_1(c)^{-1}\wP_1
\\ \notag & +\wP_1B_2(c)\wP_1
\left(B_1(c)^{-1}-(\wP_1B_1(c)\wP_1)^{-1}\right)\wP_1\,.
\end{align}
By \eqref{eq:G} and \eqref{eq:rhoc},
\begin{equation}
\label{eq:modeq-pf1}
\begin{split}
 B_4(c) \begin{pmatrix} G_1 \\ G_2 \end{pmatrix}
=&
\begin{pmatrix}\gamma_t \\ b_t \end{pmatrix}
-\wP_1B_2(c)A(c,D_y)\begin{pmatrix}\gamma \\ \tc \end{pmatrix}
-\wP_1B_2(c)
\begin{pmatrix}a_{15}(c)(\gamma_y)^2 \\ 2a_{23}(c)c_y\gamma_y\end{pmatrix}
-R^{G,2}
\\ =&
\begin{pmatrix}\gamma_t \\ b_t \end{pmatrix}
-A(c_0,D_y)\begin{pmatrix}\gamma \\ b \end{pmatrix}-\mN_1-R^G\,,
\end{split}
\end{equation}
where
\begin{equation}
  \label{eq:defmN1}
\mN_1=\wP_1(n_1,\pd_yn_2)^T\,,\quad
n_1=p_1b^2+a_{15}(c)(\gamma_y)^2\,, \quad n_2=p_2(c)\gamma_y\,,
\end{equation}
with $p_2(c)=O(\tc)$ and $R^G=R^{G,1}+R^{G,2}$, $R^{G,1}=(R^{G,1}_1,R^{G,1}_2)^T$,
\begin{align*}
R^{G,2}=& \delta B(c)\left\{M(c)
\begin{pmatrix} \gamma_{yy} \\ c_{yy}\end{pmatrix}
+\begin{pmatrix}a_{15}(c)(\gamma_y)^2 \\
2\rho'(c)a_{23}(c)c_y\gamma_y\end{pmatrix}
\right\}
\\ &\qquad -B_4(c)
\begin{pmatrix}
2m_{13}(c)c_y\gamma_y+m_{14}(c)(c_y)^2 \\ m_{24}(c)(c_y)^2
\end{pmatrix}\,.
\end{align*}

\subsection{Expressions of remainder terms}
To describe the remainder terms of \eqref{eq:modeql1},
we introduce operators $S_k$ $(k=1,2)$.
For $Q_c=\zeta_{1,c}$, $\zeta_{2,c}$  and so on, let
\begin{align*}
& S_k^1[Q_c](f)(t,y)=\frac{1}{2\pi} \int_{-\eta_0}^{\eta_0}\int_{\R^2}
f(y_1)Q_{c_0}(z)\cdot g_{k1}^*(z,\eta,c_0)e^{i(y-y_1)\eta}\,dy_1dzd\eta\,,
\\ &
S_k^2[Q_c](f)(t,y)=\frac{1}{2\pi}\int_{-\eta_0}^{\eta_0}\int_{\R^2}
f(y_1)\tc(t,y_1)g_{k2}^*(z,\eta,c(t,y_1))e^{i(y-y_1)\eta}\,dy_1dzd\eta\,,
\end{align*}
where $\delta Q_c(z)=\{Q_c(z)-Q_{c_0}(z)\}/(c-c_0)$,
$g_{k1}^*(z,\eta,c)=\eta^{-2}\{g_k^*(z,\eta,c)-g_k^*(z,0,c)\}$ and
\begin{align*}
& g_{k2}^*(z,\eta,c)=\delta Q_c(z)\cdot g_{k1}^*(z,\eta,c_0)+
Q_c(z)\cdot\frac{g_{k1}^*(z,\eta,c)-g_{k1}^*(z,\eta,c_0)}{c-c_0}\,.
\end{align*}
For $j=1$ and $2$, let
\begin{equation*}
\wS_j=\begin{pmatrix}S^j_1[\zeta_{1,c}]  & S^j_1[\zeta_{2,c}]
\\ S^j_2[\zeta_{1,c}]  & S^j_2[\zeta_{2,c}]\end{pmatrix}\,,
\quad
\wS_j'=
\begin{pmatrix}S^j_1[\mL_{1,c}(0)\zeta_{1,c}]  & S^j_1[\mL_{1,c}(0)\zeta_{2,c}]
\\ S^j_2[\mL_{1,c}(0)\zeta_{1,c}]  & S^j_2[\mL_{1,c}(0)\zeta_{2,c}]\end{pmatrix}\,.
\end{equation*}
Then by \eqref{eq:defl1}, \eqref{eq:l12L} and \eqref{eq:l12N},
\begin{equation}
  \label{eq:mod-l1lin}
\begin{split}
&  \frac{1}{\sqrt{2\pi}}\wP_1\mF_\eta^{-1}
\begin{pmatrix}
\la \ell_{11}+\tilde{\ell}_{12L}, g_{11}^*(z,\eta,c(t,y))e^{iy\eta}\ra 
\\ 
\la \ell_{11}+\tilde{\ell}_{12L}, g_{21}^*(z,\eta,c(t,y))e^{iy\eta}\ra
\end{pmatrix}
\\ =& \wS_1
\begin{pmatrix}\gamma_t-b \\ b_t\end{pmatrix}
+\wS_1'
\begin{pmatrix}\gamma_{yy} \\ b_{yy}\end{pmatrix}+R^{\ell_1,1}\,,
\end{split}
\end{equation}
\begin{align} \notag
R^{\ell_1,1}=&
\wS_1
\begin{pmatrix}0 \\ c_t-b_t\end{pmatrix}
+
\wS_2\begin{pmatrix}\gamma_t-\tc \\ c_t\end{pmatrix}
+\wS_1'\begin{pmatrix} 0 \\ c_{yy}-b_{yy}\end{pmatrix}
+\wS_2'\begin{pmatrix}\gamma_{yy} \\ c_{yy}\end{pmatrix}\,,
\\ \label{eq:mod-l1nl}
R^{\ell_1,2}_k:=& \frac{1}{\sqrt{2\pi}}\wP_1\mF_\eta^{-1}
\la \tilde{\ell}_{12N},g_{k1}^*(z,\eta,c(t,y))e^{iy\eta}\ra
\\=& \sum_{j=1,2}\left(2S^j_k[\Phi_{3,c}](c_y\gamma_y)
+S^j_k[\Phi_{4,c}](c_y^2)+S^j_k[\Phi_{5,c}](\gamma_y^2)\right)\,.
\notag
\end{align}
Let 
$$R^{\ell_1,r}_k=\wS_1\left((b-\tc)\mathbf{e_1}\right)
+\frac{1}{\sqrt{2\pi}}\wP_1\mF_\eta^{-1}
\la \pd_y^2\tilde{\ell}_{1r}, g_{k}^*(z,\eta,c(t,y))e^{iy\eta}\ra\,.
$$
We set 
$R^{\ell_1,j}=(R^{\ell_1,j}_1,R^{\ell_1,j}_2)^T$ for $j=1$, $2$, $r$ 
and $R^{\ell_1}=\pd_y^2(R^{\ell_1,1}+R^{\ell_1,2})+R^{\ell_1,r}$.
\par
We rewrite $\ell_2$ as  $\ell_2=c_t\pd_c\Psi_{c_0}-\tc\widetilde{\Psi}_{c_0}
+\tilde{\ell}_{2N}+\pd_y^2\tilde{\ell}_{2r}$,
\begin{align*}
& \widetilde{\Psi}_{c_0}
=\left\{\frac{c_0+1}{2}\pd_x+L_0+N_0'(\Phi_{c_0})\right\}\pd_c\Psi_{c_0}\,,
\\ & 
\tilde{\ell}_{2N}=c_t(\pd_c\Psi_c-\pd_c\Psi_{c_0})
-\left(\frac{c_0+1}{2}\pd_x+L_0\right)(\Psi_c-\tc\pd_c\Psi_{c_0})
-\gamma_t\pd_x\Psi_c+N_{2N}\,,
\\ &
N_{2N}=\tc N_0'(\Phi_{c_0})\pd_c\Psi_{c_0}-N'(\Phi_c)\Psi_c\,,
\quad
\tilde{\ell}_{2r}=L_1\Psi_c\,.
\end{align*}
Let 
$S^j_k[p]=S^j_{k1}[p]-\pd_y^2S^j_{k2}[p]$ for $j=3$, $4$,
$p(z)\in C_0^\infty(\R;\R^2)$ and
\begin{align*}
& g_{k3}^*(z,\eta,c)=\frac{g_k^*(z,\eta,c)-g_k^*(z,\eta,c_0)}{c-c_0}\,,
\quad
g_{k4}^*(z,\eta,c)=\frac{g_{k3}^*(z,\eta,c)-g_{k3}^*(z,0,c)}{\eta^2}\,.
\end{align*}
\begin{align}
  \label{eq:defS3k1}
S^3_{k1}[p](f)=&
 \left(\int_\R p(z_1)\cdot\zeta_{k,c_0}^*(z)\,dz\right)\wP_1f\,,
\\  \label{eq:defS3k2}
S^3_{k2}[p](f)=&\frac{1}{2\pi}\int_{-\eta_0}^{\eta_0}\int_{\R^2}
f(y_1)p(z_1)\cdot g_{k1}^*(z,\eta,c_0)e^{i(y-y_1)\eta}\,dy_1dzd\eta\,,
\\ \label{eq:defS4k1}
S^4_{k1}[p](f)=& \wP_1\left(f(y)\tc(t,y)
\int_\R p(z_1)\cdot g_{k3}^*(z,0,c(t,y))\,dz\right)\,,
\\ \label{eq:defS4k2}
S^4_{k2}[p](f)=&\frac{1}{2\pi}\int_{-\eta_0}^{\eta_0}\int_{\R^2}
f(y_1)\tc(t,y_1)p(z_1)\cdot g_{k4}^*(z,\eta,c(t,y_1))e^{i(y-y_1)\eta}\,dy_1dzd\eta\,,
\\ & 
\wS_3=\sum_{k=1,2}S^3_k[\pd_c\Psi_{c_0}]E_{k2}\,,\quad \wS_{3j}=\sum_{k=1,2}S^3_{kj}[\pd_c\Psi_{c_0}]E_{k2}\,, \notag
\\ & A_1(t,D_y)=\sum_{k=1,2}S^3_k[\widetilde{\Psi}_{c_0}]E_{k2}\,,\quad
A_2(t,D_y)=\sum_{k=1,2}S^3_{k2}[\widetilde{\Psi}_{c_0}]E_{k2}\,.\notag
\end{align}
Then $\wS_3=\wS_{31}-\pd_y^2\wS_{32}$, $A_1(t,D_y)=A_1(t,0)-\pd_y^2A_2(t,D_y)$ and
\begin{equation}
\label{eq:mod-l2}
\frac{1}{\sqrt{2\pi}}\wP_1\mF_\eta^{-1}
\begin{pmatrix}
\la \ell_2,g_1^*(z,\eta,c(t,y))e^{iy\eta}\ra  
\\ \la \ell_2,g_2^*(z,\eta,c(t,y))e^{iy\eta}\ra  
\end{pmatrix}
=\wS_3\begin{pmatrix}  \gamma_t-b \\ b_t\end{pmatrix}
-A_1(t,D_y)
\begin{pmatrix}  \gamma \\ b\end{pmatrix}+R^{\ell_2}\,,
\end{equation}
where 
$R^{\ell_2}=R^{\ell_2,1}-\pd_y^2R^{\ell_2,2}+\pd_yR^{\ell_2,r}$,
$R^{\ell_2,j}=(R^{\ell_2,j}_1,R^{\ell_2,j}_2)^T$ for $j=1$, $2$ and $r$,
\begin{align*}
R^{\ell_2,1}_k=& -\wS_{31}(b_t-c_t)\mathbf{e_2}+A_1(t,0)(b-\tc)\mathbf{e_2}
\\ & 
+S^4_{k1}[\pd_c\Psi_{c_0}](c_t)
-S^4_{k1}[\widetilde{\Psi}_{c_0}](\tc)
+\wP_1\left(\int_\R \tilde{\ell}_{2N}\cdot \zeta_{k,c(t,y)}^*(z)\,dz\right)
\\ & +\frac{1}{\sqrt{2\pi}}\wP_1\mF_\eta^{-1}\left\la\pd_y\tilde{\ell}_{2r},
(\gamma_y\pd_z-c_y\pd_c)g_k^*\left(z,\eta,c(t,y)\right)e^{iy\eta}\right\ra\,,
\\
R^{\ell_2,2}_k=&  -\wS_{32}(b_t-c_t)\mathbf{e_2}+A_2(t,D_y)(b-\tc)\mathbf{e_2}
\\ &
+S^4_{k2}[\pd_c\Psi_{c_0}](c_t)-S^4_{k2}[\widetilde{\Psi}_{c_0}](\tc)
+\frac{1}{\sqrt{2\pi}}\wP_1\mF_\eta^{-1}
\la\tilde{\ell}_{2N}, g_{k1}^*(z,\eta,c(t,y))e^{iy\eta}\ra\,,
\\
R^{\ell_2,r}_k=& \frac{1}{\sqrt{2\pi}}\wP_1\mF_\eta^{-1}
\la \pd_y\tilde{\ell}_{2r}, g_k^*(z,\eta,c(t,y))e^{iy\eta}\ra\,.
\end{align*}
We set
$R^j_k(t,y)=\frac{1}{2\pi}\int_{-\eta_0}^{\eta_0}II^j_k(t,\eta)e^{iy\eta}\,d\eta$ 
and $R^j=(R^j_1,R^j_2)^T$ for $0\le j\le 5$.
\subsection{Modulation equations}
Let $p_\a(x)=1+\tanh\a x$,\newline
$\|U\|_{\bW(t)}=\left\|p_\a(z_1)^{1/2}\mathcal{E}(U)^{1/2}\right\|_{L^2(\R^2)}$ and
$\bM_\infty(T)=\sup_{0\le t\le T, y\in\R}|\gamma(t,y)|$,
\begin{align*}
& \bM_{c,\gamma}(T)=
\sum_{k=0,\,1}\sup_{t\in[0,T]}\bigl\{\la t\ra^{(2k+1)1/4}
(\|\pd_y^k\tc(t)\|_Y+\|\pd_y^{k+1}\gamma(t)\|_Y)
\\ & \phantom{\bM_{c,x}(T)=\sum}
+\la t\ra(\|\pd_y^2\tc(t)\|_Y+\|\pd_y^3\gamma(t)\|_Y)\bigr\}\,,
\\
& \bM_1(T)=\sup_{t\in[0,T]}\bigl\{\|U_1(t)\|_{\bE}+
\la t\ra^2\|U_1(t)\|_{\bW(t)}+\la t\ra\|(1+z_+)U_1(t)\|_{\bW(t)}\bigr\}
\\ & \qquad\qquad +\|\mathcal{E}(U_1)^{1/2}\|_{L^2(0,T;\bW(t))}\,,
\\ & \bM_2(T)=\sup_{0\le t\le T}\la t\ra^{3/4}\|U_2(t)\|_{\bX_1}\,,\quad
 \bM_U(T)=\sup_{t\in[0,T]}\|U(t)\|_{\bE}\,.
\end{align*}
\par

Suppose that $\|U_0\|_{\bE}$, $\bM_{c,\gamma}(T)$, $\bM_1(T)$,
$\bM_2(T)$ and $\bM_\infty(T)$ and $\bM_2(T)$ are sufficiently small.
Then by Proposition~\ref{prop:continuation},
the decomposition \eqref{eq:decomp}
satisfying \eqref{eq:defU2} and \eqref{eq:orth} persists for $t\in[0,T]$
and it follows from \eqref{eq:modeq}, \eqref{eq:G}, \eqref{eq:modeq-pf1},
\eqref{eq:mod-l1lin}, \eqref{eq:mod-l1nl} and \eqref{eq:mod-l2}
that for $t\in[0,T]$, 
\begin{align} \notag
B_5 \begin{pmatrix}\gamma_t-b \\ b_t \end{pmatrix}
=& \left\{A(c_0,D_y)-E_{12}+B_4(c_0)A_1(t,D_y)+\pd_y^4B_4(c_0)\wS_1'\right\}
  \begin{pmatrix} \gamma \\ b  \end{pmatrix}
\\ & +\mN_1
-B_4(c)\sum_{j\in\{\ell_1,\ell_2,0,1,\cdots,5\}}R^j+R^G+R^6+R^7\,,
\label{eq:modeq-nl}
\end{align}
where $B_5=I-B_4(c_0)(\pd_y^2\wS_1-\wS_3)$,
\begin{gather*}
R^6=(B_4(c)-B_4(c_0))\pd_y^2r_6\,,
\enskip R^7=(B_4(c)-B_4(c_0))r_7\,,\\
r_6=\wS_1\left((\gamma_t-b)\mathbf{e_1}+b_t\mathbf{e_2}\right)
+\wS_1'\left(\gamma_{yy}\mathbf{e_1}+b_{yy}\mathbf{e_2}\right)\,,
\enskip
r_7=A_1(t,D_y)(b\mathbf{e_2})-\wS_3(b_t\mathbf{e_2})\,.
\end{gather*}
\par
Next, we will show that $B_5^{-1}$ can be decomposed
into a sum of a time-dependent matrix multiplied by $\wP_1$ 
and an operator which belongs to $\pd_y^2(B(Y)\cap B(Y_1))$.
Let
$$\obu{B}_5=(I+B_4(c_0)\wS_{31})^{-1}\,,\quad \oc{B}_5=\obu{B}_5B_4(c_0)(\wS_1+\wS_{32})B_5^{-1}\,.$$
By the definition of $\wS_1$ and \eqref{eq:defS3k1},
$\left[\pd_y,\obu{B}_5\right]=\left[\pd_y,\oc{B}_5\right]=O$ and 
$B_5^{-1}=\obu{B}_5-\pd_y^2\oc{B}_5$.
\begin{claim}
\label{cl:obuB5}
There exist positive constants $\eta_1$, $h_0$, $\delta$ and $C$ such that
if $\eta_0\in(0,\eta_1]$, $h\ge h_0$ and $\bM_{c,\gamma}(T)\le \delta$,
then for $t\in[0,T]$,
\begin{gather*}
\bigl\|\obu{B}_5\bigr\|_{B(Y)\cap B(Y_1)}+
\bigl\|\chi(D_y)\obu{B}_5\bigr\|_{B(L^1)}+
\bigl\|\oc{B}_5\bigr\|_{B(Y)\cap B(Y_1)} \le C\,,
\\
\bigl\|\obu{B}_5-I\bigr\|_{B(Y_1)}
+\bigl\|\chi(D_y)(\obu{B}_5-I)\bigr\|_{B(L^1)}
\le C e^{-\a((c_0-1)t/2+h)}\,,
\\
\bigl\|\oc{B}_5-B_4(c_0)\wS_1(I-B_4(c_0)\pd_y^2\wS_1)^{-1}\bigr\|_{B(Y)\cap B(Y_1)}
\le C e^{-\a((c_0-1)t/2+h)}\,.
\end{gather*}
\end{claim}
Using Claim~\ref{cl:S3} in Appendix~\ref{ap:s},
we can prove Claim~\ref{cl:obuB5} in the same way as \cite[Claim~5.2]{Miz19}.
\par
Finally, we consider $B_5^{-1}B_4(c)R^3$.  In order to prove
$\|\tc(t,\cdot)\|_Y+\|\gamma_y(t,\cdot)\|_Y\lesssim \la t\ra^{-1/4}$
by using Lemma~\ref{lem:decay-BB}, a crude estimate \eqref{eq:R3-est1}
is insufficient.  To distinguish a problematic part of
$B_5^{-1}B_4(c)R^3$ in Section~\ref{sec:apriori},  we decompose
$B_5^{-1}B_4(c)R^3$ as follows.
Let 
\begin{gather*}
\mN_{U_1}=-E_2\obu{B}_5B_4(c_0)R^3\,,\quad  
III_1=-E_1B_5^{-1}B_4(c)R^3\,,\\
III_2=E_2\obu{B}_5(B_4(c)-B_4(c_0))R^3\,,\quad
III_3=E_2\oc{B}_5B_4(c)R^3\,.
\end{gather*}
Then $B_5^{-1}B_4(c)R^3=-\mN_{U_1}+III_1+III_2-\pd_y^2III_3$.
By Claims~\ref{cl:obuB5}, \ref{cl:B4} and \eqref{eq:R3-est1},
\begin{gather*}
\|III_1\|_Y+\|III_3\|_Y\lesssim \bM_1(T)\la t\ra^{-2} \,,
\\ \|III_2\|_{Y_1}+\|\chi(D_y)III_2\|_{L^1}\lesssim
\bM_{c,\gamma}(T)\bM_1(T)\la t\ra^{-9/4}\,,
\end{gather*}
and $III_i$ $(i=1,\,2,\,3)$ are good terms.
Using a change of variable, we will eliminate the bad part from $\mN_{U_1}$
in Section~\ref{sec:apriori} following the lines of \cite{Miz18,Miz19}.
\par
It follows from Claims~\ref{cl:Rg}--\ref{cl:R7} and the above that
that nonlinear terms of \eqref{eq:modeq-nl} satisfy
$$ B_5^{-1}(\mN_1+R^G+R^6+R^7)-B_5^{-1}B_4\sum_{j\in\{\ell_1,\ell_2,0,1,\cdots,5\}}R^j
= \mN_1+\mN_{U_1}+\mN_2+\mN_3+\pd_y\mN_4\,,$$
\begin{align}
 \label{eq:mN2}
& \mN_2=E_2\mN_2\,,\quad
\|\mN_2\|_{Y_1}+\|\chi(D_y)E_2\mN_2\|_{L^1(\R)}
\\ \notag \lesssim  &
\left(\bM_{c,\gamma}(T)+\bM_1(T)+\bM_2(T)\right)^2\la t\ra^{-5/4}\,,
\end{align}
\begin{align}
& \label{eq:mN3}
\mN_3=E_1\mN_3\,,\quad \|\mN_3\|_Y\lesssim \left(\bM_1(T)+\bM_{c,\gamma}(T)^2
+\bM_2(T)^2\right)\la t\ra^{-5/4}\,,
\\ & \label{eq:mN4}
\|\mN_4\|_Y \lesssim (\bM_{c,\gamma}(T)+\bM_1(T)+\bM_2(T)^2)\la t\ra^{-1}\,.
\end{align}
\begin{proposition}
  \label{prop:modulation}
There exists a positive number $\delta_3$ such that if
$\bM_{c,\gamma}(T)+\bM_1(T)\linebreak+\bM_2(T)+\eta_0+e^{-\a h}<\delta_3$,
then
\begin{align}
  \label{eq:modeq2}
& \begin{pmatrix}\gamma_t \\ b_t\end{pmatrix}
=\mathcal{A}(t,D_y)\begin{pmatrix} \gamma \\ b \end{pmatrix}
+\mN_1+\mN_2+\mN_3+\pd_y\mN_4+\mN_{U_1}\,,
\\ & \gamma(0,\cdot)=0\,,\quad \tc(0,\cdot)=0\,,
\end{align}
where $\mathcal{A}(t,D_y)
=E_{12}+B_5^{-1}\left\{A(c_0,D_y)-E_{12}+B_4(c_0)A_1(t,D_y)
+\pd_y^4B_4(c_0)\wS_1'\right\}$\linebreak
and $\mN_j$ $(2\le j\le 4)$ satisfy \eqref{eq:mN2}--\eqref{eq:mN4}.
\end{proposition}
\bigskip

\section{\`A priori estimates for modulation parameters}
\label{sec:apriori}
In this section, we will estimate  $\bM_{c,\gamma}(T)$ and $\bM_\infty(T)$.
\begin{lemma}
  \label{lem:Mcx-bound}
There exist positive constants $\delta_4$ and $C$ such that if
$\bM_{c,\gamma}(T)+\bM_1(T)+\bM_2(T)+\eta_0+e^{-\a h}\le \delta_4$, then
\begin{equation*}
\bM_{c,\gamma}(T)\le C(\bM_1(T)+\bM_2(T)^2)\,,\quad 
\bM_\infty(T) \le C(\bM_{c,\gamma}(T)+\bM_1(T)+\bM_2(T)^2)\,.
\end{equation*}
\end{lemma} 
\begin{proof}
Let $\mathcal{A}_*(D_y)=A(c_0,D_y)$, $\mathcal{A}(t,D_y)=\mathcal{A}_0(D_y)+\mathcal{A}_1(t,D_y)$ and
$$\mathcal{A}_0(D_y)=E_{12}+\left(I-B_4(c_0)\pd_y^2\wS_1\right)^{-1}\left(A(c_0,D_y)-E_{12} +\pd_y^4B_4(c_0)\wS_1'\right)\,.$$
Then $\mathcal{A}(t,D_y)$ satisfies \eqref{eq:H} and
we can apply Lemma~\ref{lem:decay-BB} to the solution operator $U(t,s)$
of \eqref{eq:cx-linear}.
\par
The $Y_1$-norm of $E_2\mN_{U_1}$ does not necessarily decay
as $t\to\infty$. We will express a bad part of $\mN_{U_1}$
as a time derivative of $$k(t,y):=\la \wU_1(t,\cdot,y),\zeta_{2,c_0}^*\ra\,,$$
and eliminate that part from \eqref{eq:modeq2} by the change of variable
\begin{gather*}
\bb(t,y)=(\gamma(t,y),\tb(t,y))^T\,,\quad  \tb(t,y)=b(t,y)-\tk(t,y)\,,
\end{gather*}
where $\quad\tk(t,y)=(\beta_1(c_0)+S^3_{21})^{-1}\wP_1k(t,y)$.
Since $E_2\obu{B}_5=(1+\beta_1(c_0)^{-1}S^3_{21})^{-1}E_2$ and
$E_2B_4(c_0)=\beta_1(c_0)^{-1}E_2\wP_1$, it follows from Claim~\ref{cl:R3} that
\begin{align*}
\mN_{U_1}=& -(\beta_1(c_0)+S^3_{21})^{-1}R^3_2\mathbf{e_2}
\\=& \left\{\pd_t\tk+(\beta_1(c_0)+S^3_{21})^{-1}\left(
(\pd_tS^3_{21})\tk-R^3_{21}+\pd_y^2R^3_{22}\right)\right\}\mathbf{e_2}\,,
\end{align*}
and
\begin{align}
  \label{eq:modeq3}
& \bb_t=\mathcal{A}(t,D_y)\bb+\mN_1+\mN_2+\mN_3+\pd_y\mN_4
+\obu{\mN}_{U_1}+\pd_y^2\oc{\mN}_{U_1}\,, 
\\ \label{eq:modeq3-init}
& \bb(0,\cdot)=-\wP_1\tk(0,\cdot)\mathbf{e_2}\,,
\end{align}
where $\mathcal{A}_2(t,\eta)=\eta^{-2}\{\mathcal{A}_1(t,\eta)
-\mathcal{A}_1(t,0)\}$ and
\begin{align*}
\obu{\mN}_{U_1}=& (\beta_1(c_0)+S^3_{21})^{-1}(\pd_tS^3_{21}\tk-R^3_{21})\mathbf{e_2}
-\tk\mathbf{e_1}-\mathcal{A}_1(t,0)\tk\mathbf{e_2}\,,
\\
\oc{\mN}_{U_1}=& (\beta_1(c_0)+S^3_{21})^{-1}R^3_{22}\mathbf{e_2}
-\left\{\pd_y^{-2}\left(\mathcal{A}_0(D_y)-E_{12}\right)
-\mathcal{A}_2(t,D_y)\right\}\tk\mathbf{e_2}\,.
\end{align*}
\par
Now we will estimate the right hand side of
$$
\bb(t)=U(t,0)\bb(0)+\int_0^tU(t,s)\left(\mN_1+\mN_2+\mN_3+\pd_y\mN_4
+\obu{\mN}_{U_1}+\pd_y^2\oc{\mN}_{U_1}\right)\,ds$$
by using Lemma~\ref{lem:decay-BB}.
Since $\|\tk(0)\|_{Y_1}+\|\chi(D_y)\tk(0)\|_{L^1}
\lesssim \|\la y\ra U_0\|_{L^2(\R^2)}$, it follows 
that for $k\ge0$ and $t\ge0$,
\begin{equation*}
\la t\ra^{(2k+1)/4} \left\|\pd_y^kU(t,0)\bb(0)\right\|_{\dot{H}^1\times L^2}
+\left\|U(t,0)\bb(0)\right\|_{L^\infty}
\lesssim \|\la y\ra U_0\|_{L^2}\,.
\end{equation*}
By \eqref{eq:defmN1}, we have
$\|\pd_y^kn_1(s)\|_Y+\|\pd_y^kn_2(s)\|_Y\lesssim
\bM_{c,\gamma}(T)^2 \la s\ra^{-(2k+3)/4}$ and for $t\in[0,T]$,
\begin{align*}
\left\|\int_0^tU(t,s)\mN_1\,ds\right\|_{\dot{H}^1\times L^2}
\lesssim & \sum_{j=1,2}\int_0^t \la t-s\ra^{-1/2}\|n_j(s)\|_Y\,ds
\lesssim  \bM_{c,\gamma}(T)^2\la t\ra^{-1/4}\,,
\\
 \left\|\int_0^t\pd_yU(t,s)\mN_1\,ds\right\|_{\dot{H}^1\times L^2}
\lesssim & \sum_{j=1,2\,,k=0,1}\int_{kt/2}^{(k+1)t/2} \la t-s\ra^{k/2-1}
\|\pd_y^kn_j(s)\|_Y\,ds
\\ \lesssim  &  \bM_{c,\gamma}(T)^2\la t\ra^{-3/4}\,,
\\
\left\|\int_0^t\pd_y^2U(t,s)\mN_1\,ds\right\|_{\dot{H}^1\times L^2}
\lesssim & \sum_{j=1,2}\int_0^t \la t-s\ra^{-1}\|\pd_yn_j(s)\|_Y\,ds
\lesssim  \bM_{c,\gamma}(T)^2\la t\ra^{-1}\,,
\\
\left\|\int_0^tU(t,s)\mN_1\,ds\right\|_{L^\infty} 
\lesssim & \sum_{j=1,2}\int_0^t \la t-s\ra^{-1/4}\|n_j(s)\|_Y\,ds
\lesssim  \bM_{c,\gamma}(T)^2\,.
\end{align*}
\par
It follows from \eqref{eq:mN2}--\eqref{eq:mN4} that
for $k=0$, $1$, $2$ and $t\in[0,T]$,
\begin{align*}
&  \la t\ra^{(2k+1)/4}
\left\|\pd_y^k\int_0^t U(t,s)\mN_2(s)\,ds\right\|_{\dot{H}^1\times L^2}
+\left\|\int_0^t U(t,s)\mN_2(s)\,ds\right\|_{L^\infty}
\\ \lesssim &
(\bM_{c,\gamma}(T)+\bM_1(T))(\bM_{c,\gamma}(T)+\bM_1(T)+\bM_2(T))\,,
\end{align*}
\begin{multline*}
\left\|\pd_y^k\int_0^t U(t,s)\mN_3(s)\,ds\right\|_{\dot{H}^1\times L^2}
\lesssim 
\int_0^t \la t-s\ra^{-(k+1)/2}\|\mN_3(s)\|_Y\,ds
\\ \lesssim 
(\bM_1(T)+\bM_{c,\gamma}(T)^2+\bM_{c,\gamma}(T)\bM_2(T))\la t\ra^{-(k+1)/2\wedge 5/4}\,,
\end{multline*}
\begin{align*}
\left\|\pd_y^k\int_0^t U(t,s)\pd_y\mN_4(s)\,ds\right\|_{\dot{H}^1\times L^2}
\lesssim &
\eta_0^{1/4}\int_0^t \la t-s\ra^{-\frac12(k+\frac34)}\|\mN_4(s)\|_Y\,ds
\\ \lesssim &
\eta_0^{1/4}(\bM_{c,\gamma}(T)+\eta_0\bM_1(T))\la t\ra^{-(2k+1)/4\wedge 1}\,,
\end{align*}
\begin{equation}
  \label{eq:Linfy=0-1}
\begin{split}
 \left\|\int_0^t U(t,s)\left(\mN_3(s)+\pd_y\mN_4(s)\right)\,ds
\right\|_{L^\infty}
\lesssim & \sum_{k=3,4}
\int_0^t \la t-s\ra^{-1/4}\|\mN_k(s)\|_Y\,ds
\\ \lesssim & (\bM_1(T)+\bM_{c,\gamma}(T))\la t\ra^{-1/4} \log(2+t)\,.
\end{split}  
\end{equation}
\par
Since $|S^3_{21}|+|\pd_tS^3_{21}|\lesssim e^{-\a((c_0-1)t/2+h)}$ and
$$\|\mathcal{A}_1(t,0)\|_{B(Y_1)}+
\|\chi(D_y)\mathcal{A}_1(t,0)\|_{B(L^1)}\|\mathcal{A}_2(t,D_y)\|_{B(Y_1)}
\lesssim e^{-\a((c_0-1)t/2+h)}\,,$$
it follows from Claims~\ref{cl:R3}, \ref{cl:k-decay} and \ref{cl:k-growth}
that for $t\in[0,T]$ and $k$ satisfying $0\le k\le 2$,
\begin{gather}
\label{eq:obu-est1}
\bigl\|E_1\obu{N}_{U_1}\bigr\|_Y+\bigl\|\oc{N}_{U_1}\bigr\|_Y
\lesssim \bM_1(T)\la t\ra^{-2}\,,
\\ \label{eq:obu-est2}
\bigl\|E_2\obu{N}_{U_1}\bigr\|_{Y_1}+\bigl\|E_2\chi(D_y)\obu{N}_{U_1}\bigr\|_{L^1}
\lesssim \bM_1(T)(e^{-\a h}+\bM_{c,\gamma}(T)+\bM_1(T))\la t\ra^{-9/4}\,.
\end{gather}
and that
\begin{align*}
\la t\ra^{(2k+1)/4}
\biggl\|\pd_y^k\int_0^t U(t,s)\obu{\mN}_{U_1}(s)\,ds\biggr\|_{\dot{H}^1\times L^2}
+\biggl\|\int_0^t U(t,s)E_2\obu{\mN}_{U_1}(s)\,ds\biggr\|_{L^\infty}
\lesssim & \bM_1(T)\,,
\end{align*}
\begin{align}
\label{eq:Linfy=0-2}
& \biggl\|\int_0^t U(t,s)E_1\obu{\mN}_{U_1}(s)\,ds\biggr\|_{L^\infty}
\lesssim 
\int_0^t \la t-s\ra^{-1/4} \bigl\|E_1\obu{\mN}_{U_1}\bigr\|_Y\,ds
\lesssim \bM_1(T)\la t\ra^{-1/4}\,,
\end{align}
\begin{align}
\label{eq:Linfy=0-3}
 \biggl\|\pd_y^{k+2}  \int_0^tU(t,s)\oc{\mN}_{U_1}\,ds\biggr\|_Y
\lesssim &
\int_0^t\la t-s\ra^{-(k+1)/2}\bigl\|\oc{\mN}_{U_1}\bigr\|_Y\,ds
\\ \lesssim  & \bM_1(T)\la t\ra^{-(k+1)/2}\,. \notag
\end{align}
\par
By the definition and Claim~\ref{cl:b-capprox}, we have for $k=0$, $1$, $2$ and $t\in[0,T]$,
\begin{align*}
\|\pd_y^k\tc(t)\|\lesssim &  \|\pd_y^k\tb\|_Y+\|k(t)\|_Y
\lesssim \|\pd_y^k\tb\|_Y+ \bM_1(T)\la t\ra^{-2}\,.
\end{align*}
Thus we have Lemma~\ref{lem:Mcx-bound}.
\end{proof}
\bigskip

\section{The energy identity}
\label{sec:energy}
In this section, we will derive an energy estimate for solutions around
modulating line solitary waves.
\begin{lemma}
  \label{lem:MU-bound}
Let $\a\in(0,\a_{c_0})$ and $\delta_4$ be as in Lemma~\ref{lem:Mcx-bound}.
Suppose that $\bM_{c,\gamma}(T)+\bM_1(T)+\bM_2(T)+\eta_0+e^{-\a h}\le \delta_4$.
Then there exists a positive constant $C$ such that 
$$\bM_U(T)\le C\left(\|U_0\|_{\bE}+\bM_{c,\gamma}(T)+\bM_1(T)+\bM_2(T)\right)\,.$$
\end{lemma}
To prove Lemma~\ref{lem:MU-bound}, we will derive an energy identity
of \eqref{eq:U}. Let $A_y=1-a\pd_y^2$ and
\begin{align*}
& \la \Phi,\Psi\ra_\bE=\la (-A\Delta E_1+BE_2)\Phi,\Phi\ra\,,
\quad
\la \Phi,\Psi\ra_{\bE_x}
=\la (-A_0\pd_x^2E_1+B_0E_2)\Phi, \Psi\ra\,,
\\ &
\la \Phi,\Psi\ra_{\bE_y}=\la \Phi,\Psi\ra_{\bE}-\la \Phi,\Psi\ra_{\bE_x}
=\la \pd_y^2(-A_y+2a\pd_x^2)E_1\Phi-b\pd_y^2E_2\Phi, \Psi\ra\,,
\end{align*}
and let $\|U\|_{\bE_x}=\la U,U\ra_{\bE_x}^{1/2}$.
The energy of solutions around line solitary waves is infinite since line solitary waves do not
decay in the $y$-direction. Moreover, the velocity potential $\varphi_c(x)$
tends to a negative constant as $x\to-\infty$.
Subtracting infinite energy part from $E(\Phi)$ and having
the orthogonality condition \eqref{eq:orth} for $\eta=0$ and $k=2$ in mind,
we find the quantity
\begin{align*}
 Q(t)=& \frac12\|U(t)\|_\bE^2-\left\la U(t),\Psi_{c(t,y)}(z_1)\right\ra_{\bE_x}
\\ & +\left\la U(t),\Phi_{c(t,y)}(z)-\Phi_{c_0}(z)-\Psi_{c(t,y)}(z)\right\ra_{\bE_y}
+\frac12\left\| \Psi_{c(t,y)}\right\|_{\bE_x}^2 \\ &
+\frac12\left\la \Phi_{c(t,y)}(z)-\Phi_{c_0}(z)-\Psi_{c(t,y)}(z_1),
\Phi_{c(t,y)}(z)-\Phi_{c_0}(z)-\Psi_{c(t,y)}(z_1)\right\ra_{\bE_y}\,.
\end{align*}
\begin{lemma}
\label{lem:energy-id}
  \begin{multline*}
Q(t)= Q(0)+\int_0^t\left(
\la U,\ell_{11}\ra_{\bE_x}+\la U,\ell_{13}\ra_\bE-\la \ell_{11},\Psi_c\ra_{\bE_x}
\right)\,ds
\\
+\int_0^t\left(\sum_{j=1,2,\,k=2,3}\la \ell_{jk},\Phi_c-\Phi_{c_0}\ra_{\bE_y}+
\left\la (\gamma_t-\tc)\pd_x\Phi_{c_0}-c\pd_x\Psi_c,
\Phi_c-\Phi_{c_0}-\Psi_c\right\ra_{\bE_y}\right)\,ds
\\
+\int_0^t\left(\la U, \tc\pd_x\Phi_{c_0}\ra_{\bE_y}+\la N'(\Psi_c)U-N(U),\Phi_c\ra_{\bE_x}+\la N'(\Psi_c)U-N(U),\Phi_{c_0}\ra_{\bE_y}\right)\,ds
\\  
+\int_0^t\left(
\la U, (\gamma_t-\tc)\pd_x\Phi_{c_0}\ra_{\bE_y}
-\la U, \pd_y^2L_1\Phi_{c_0}\ra_\bE
+\la\pd_y^2L_1U_,\Phi_c-\Phi_{c_0}\ra_{\bE_x}\right)\,ds
\\ 
+\int_0^t\left(\la N'(\Phi_c)U,\Phi_c-\Phi_{c_0}\ra_{\bE_y}
+\la U, N_0(\Phi_c)-N_0(\Phi_{c_0})\ra_{\bE_y}\right)\,ds\,.
  \end{multline*}
\end{lemma}
\begin{proof}
Let $\Phi=(\phi_1, \phi_2)^T$, $\Psi=(\psi_1,\psi_2)^T$ and
$J=B^{-1}(E_{12}-E_{21})$.
Since $L\Phi=JE'(\Phi)$ and $J$ is skew-symmetric,
\begin{equation}
\label{eq:E'1}
\la \Phi, L\Phi\ra_\bE=\la E'(\Phi), L\Phi\ra=0\,,
\quad 
\la L\Phi, \Psi\ra_\bE+\la \Phi, L\Psi\ra_\bE=0\,.
\end{equation}
Similarly,
\begin{equation}
  \label{eq:E0'1}
\la L_0\Phi, \Phi\ra_{\bE_x}=0\,,
\quad 
\la L_0\Phi, \Psi\ra_{\bE_x}+\la \Phi, L_0\Psi\ra_{\bE_x}=0\,.
\end{equation}
By integration by parts,
\begin{align}
  \label{eq:E'2}
& \la \pd_x\Phi,\Phi\ra_\bE=0\,,\enskip \la \Phi, N(\Phi)\ra_\bE=0\,,
\\ & \label{eq:UN'U+N}
\la U,N'(\Phi_c)U\ra_\bE+\la N(U),\Phi_c\ra_\bE=0\,,
\quad
\la U,N'(\Psi_c)U\ra_\bE=-\la N(U),\Psi_c\ra_\bE=0\,,
\\ & \label{eq:UN'abba}
\la U, N'(\Phi_c)\Psi_c\ra_\bE+\la N'(\Psi_c)U,\Phi_c\ra_\bE=0\,.
\end{align}
We remark that 
$N'(\Phi_c)\Psi_c=-B^{-1}(r_c\Delta \tpsi_c+2\nabla r_c\cdot\nabla \tpsi_c)
\mathbf{e_2}\in \bE$ and $E_1N'(\Psi_c)=O$.

Since $N$ is quadratic,
$$N(\Phi_c-\Psi_c+U)-N(\Phi_c-\Psi_c)-N(U)=N'(\Phi_c-\Psi_c)U\,,$$
and it follows from \eqref{eq:U}, \eqref{eq:E'1}, \eqref{eq:E'2} and
\eqref{eq:UN'U+N} that
\begin{equation}
\label{eq:energy-pf1}
\frac{d}{dt}E(U)= \la U,\ell+N'(\Phi_c-\Psi_c)U\ra_\bE
=\la U,\ell\ra_\bE-\la N(U),\Phi_c\ra_\bE\,.
\end{equation}
By \eqref{eq:U}, \eqref{eq:E'1},  \eqref{eq:UN'abba} and the fact that
$N\cdot\mathbf{e_1}=0$ and $\Psi_c\cdot\mathbf{e_2}=0$,
\begin{equation}
\label{eq:energy-pf2}
\begin{split}
  \frac{d}{dt}\la U,\Psi_c\ra_\bE
=& \la U, \ell_{21}\ra_\bE+\la LU+\ell,\Psi_c\ra_\bE
\\ =& \la U, \ell_2\ra_\bE+\la \ell,\Psi_c\ra_\bE-\la N'(\Psi_c)U,\Phi_c\ra_\bE\,.
\end{split}
\end{equation}
Combining the above, we have
\begin{equation}
    \label{eq:energy-pf3}
    \begin{split}
\frac{d}{dt}\{E(U)-\la U,\Psi_c\ra_\bE\}
=&\la U,\ell_1\ra_\bE-\la \ell,\Psi_c\ra_\bE-\la N(U),\Phi_c\ra_\bE
+\la N'(\Psi_c)U,\Phi_c\ra_\bE\,.
    \end{split}
\end{equation}
Since $(d\pd_x+L_0)\Phi_c+N_0(\Phi_d)=0$ for $d=c_0$ and $d=c(t,y)$,
it follows from \eqref{eq:U} and \eqref{eq:defl1} that
\begin{align*}
& \frac{d}{dt}\la U,\Phi_c-\Phi_{c_0}\ra_{\bE_y}=
\la \pd_tU,\Phi_c-\Phi_{c_0}\ra_{\bE_y}+\la U, 
c_t\pd_c\Phi_c-\gamma_t\pd_x(\Phi_c-\Phi_{c_0})\ra_{\bE_y}
\\= & - \left\la U, \ell_{11}\right\ra_{\bE_y}
+\left\la U,\gamma_t\pd_x\Phi_{c_0}
+L_0(\Phi_c-\Phi_{c_0})
+N_0(\Phi_c)-N_0(\Phi_{c_0})\right\ra_{\bE_y}
\\ & \qquad +\la LU+\ell+N'(\Phi_c-\Psi_c)U+N(U),\Phi_c-\Phi_{c_0}\ra_{\bE_y}\,.
\end{align*}
By \eqref{eq:E'1} and \eqref{eq:E0'1},
\begin{align*}
& \la LU,\Phi_c-\Phi_{c_0}\ra_{\bE_y}+
\la U, L_0(\Phi_c-\Phi_{c_0})\ra_{\bE_y}
\\ = & -\la U,\ell_{12}\ra_\bE-\la U, \pd_y^2L_1\Phi_{c_0}\ra_\bE
+\la \pd_y^2L_1U,\Phi_c-\Phi_{c_0}\ra_{\bE_x}\,,
\end{align*}
and
\begin{multline*}
 \frac{d}{dt}\la U,\Phi_c-\Phi_{c_0}\ra_{\bE_y}=
-\la U,\ell_{11}\ra_{\bE_y}-\la U,\ell_{12}\ra_\bE
+\la \ell, \Phi_c-\Phi_{c_0}\ra_{\bE_y}
\\ +\la N(U)-N'(\Psi_c)U,\Phi_c-\Phi_{c_0}\ra_{\bE_y}
+\la U, \gamma_t\pd_x\Phi_{c_0}\ra_{\bE_y}-\la U, \pd_y^2L_1\Phi_{c_0}\ra_\bE
\\ +\la\pd_y^2L_1U_,\Phi_c-\Phi_{c_0}\ra_{\bE_x}
+\la N'(\Phi_c)U,\Phi_c-\Phi_{c_0}\ra_{\bE_y}
+\la U, N_0(\Phi_c)-N_0(\Phi_{c_0})\ra_{\bE_y}\,.
\end{multline*}
Since $\Psi_c\cdot\mathbf{e_2}=0$,
$N\cdot \mathbf{e_1}=N_0\cdot \mathbf{e_1}=0$
and $E_1LE_1=O$,
\begin{gather}
\label{eq:ljkge2}
\ell_{jk}\cdot\mathbf{e_1}=0\,,
\\ \notag
\la \ell_{jk}, \Psi_c\ra_{E_x}=
\la \ell_{jk}, \Psi_c\ra_{E_y}=0 \quad\text{for $j=1$, $2$ and $k=2$, $3$,}
\\ \notag
\frac12\frac{d}{dt}\la \Psi_{c(t,y)},\Psi_{c(t,y)}\ra_{\bE_x}
=\la \ell_2,\Psi_c\ra_{\bE_x}\,.
\end{gather}
Moreover,
\begin{align*}
& \frac12\frac{d}{dt}\la \Phi_c-\Phi_{c_0}-\Psi_c,\Phi_c-\Phi_{c_0}-\Psi_c\ra_{\bE_y}
+\la \ell_{11}+\ell_{21},\Phi_c-\Phi_{c_0}-\Psi_c\ra_{\bE_y}
\\ =&
\left\la (\gamma_t-\tc)\pd_x\Phi_{c_0}-c\pd_x\Psi_c,
\Phi_c-\Phi_{c_0}-\Psi_c\right\ra_{\bE_y}\,.
\end{align*}
\par

Combining the above, we have
\begin{multline*}
 \frac{d}{dt}Q(t)=
\la U,\ell_{11}\ra_{\bE_x}+\la U,\ell_{13}\ra_\bE-\la \ell_{11},\Psi_c\ra_{\bE_x}
 +\sum_{j=1,2,\,k=2,3}\la \ell_{jk},\Phi_c-\Phi_{c_0}\ra_{\bE_y}
\\  +\left\la (\gamma_t-\tc)\pd_x\Phi_{c_0}-c\pd_x\Psi_c,
\Phi_c-\Phi_{c_0}-\Psi_c\right\ra_{\bE_y}
\\ +\la N'(\Psi_c)U- N(U),\Phi_c\ra_{\bE_x}
+\la N'(\Psi_c)U-N(U),\Phi_{c_0}\ra_{\bE_y}
\\ 
+\la U, \gamma_t\pd_x\Phi_{c_0}\ra_{\bE_y}-\la U, \pd_y^2L_1\Phi_{c_0}\ra_\bE
+\la\pd_y^2L_1U_,\Phi_c-\Phi_{c_0}\ra_{\bE_x}
\\ 
+\la N'(\Phi_c)U,\Phi_c-\Phi_{c_0}\ra_{\bE_y}
+\la U, N_0(\Phi_c)-N_0(\Phi_{c_0})\ra_{\bE_y}\,.
\end{multline*}
Thus we prove Lemma~\ref{lem:energy-id}.
\end{proof}

By Proposition~\ref{prop:modulation},
Claim~\ref{cl:b-capprox} and \eqref{eq:R3-est1}, we have the following.
\begin{claim}
\label{cl:cx_t-bound}
Let $\delta_2$ be as in proposition~\ref{prop:modulation}. Suppose
$\bM_{c,\gamma}(T)+\bM_1(T)+\bM_2(T)+\eta_0+e^{-\a L}<\delta_2$ for a $T\ge0$.
Then for $t\in[0,T]$,
$$\|c_t\|_Y+\|\gamma_t-\tc\|_Y\lesssim 
(\bM_{c,\gamma}(T)+\bM_1(T)+\bM_2(T)^2)\la t\ra^{-3/4}\,.$$
\end{claim}
Now we are in a position to prove Lemma~\ref{lem:MU-bound}.
\begin{proof}[Proof of Lemma~\ref{lem:MU-bound}]
First, we remark $\|\Psi_c\|_{\bE_x}\lesssim \|\tc\|_Y$ and 
that for $k\ge1$,
\begin{equation*}
\pd_c^k\left(\varphi_c(z)-\tpsi_c(z_1)\right)=\left\{
  \begin{aligned}
     O(e^{-2\a|z|})\quad &\text{if $z\ge0$ or $z_1\le-1$,}
\\   O(1)\quad &\text{otherwise.}
  \end{aligned}\right.
\end{equation*}
Since 
$\pd_y\{\varphi_{c(t,y)}(z)-\varphi_{c_0}(z)\}
=   c_y\pd_c\varphi_{c(t,y)} -\gamma_y(q_{c(t,y)}(z)-q_{c_0}(z))$,
\begin{multline*}
\pd_y^2\{\varphi_{c(t,y)}(z)-\varphi_{c_0}(z)\}
=   c_{yy}\pd_c\varphi_{c(t,y)}(z)-c_y^2\pd_c^2\varphi_{c(t,y)}(z)
-2c_y\gamma_y\pd_cq_{c(t,y)}(z)
\\  -\gamma_{yy}(q_{c(t,y)}(z)-q_{c_0}(z))
+(\gamma_y)^2(q_{c(t,y)}'(z)-q_{c_0}'(z))\,,
\end{multline*}
and $r_c(z)$, $q_c(z)=\varphi_c'(z)$ and their derivatives are exponentially
localized in $z$, we have
\begin{align*}
\|\Phi_{c(t,y)}(z)-\Phi_{c_0}(z)-\Psi_{c(t,y)}(z_1)\|_\bE
\lesssim & \|c_y(t)\|_Y\la t\ra^{1/2}+\|\tc(t)\|_Y+\|\gamma_y(t)\|_Y
\\ \lesssim & \bM_{c,\gamma}(T)\la t\ra^{-1/4}\,,
\end{align*}
\begin{gather*}
|\la U, \Phi_c-\Phi_{c_0}-\Psi_c\ra_{\bE_y}|
\lesssim   \|U\|_\bE\|\Phi_c-\Phi_{c_0}-\Psi_c\|_\bE
\lesssim  \bM_{c,\gamma}\bM_U(T)\la t\ra^{-1/4}\,,
\\
\left|\left\la \Phi_c-\Phi_{c_0}-\Psi_c,
\Phi_c-\Phi_{c_0}-\Psi_c\right\ra_{\bE_y}\right|
\lesssim \bM_{c,\gamma}(T)^2\la t\ra^{-1/2}\,.
\end{gather*}
Hence it follows that 
$$ \left|Q(t)-\frac{1}{2}\|U(t)\|_\bE^2\right| \lesssim 
\bM_{c,\gamma}(T)\bM_U(T)\la t\ra^{-1/4}
+\bM_{c,\gamma}(T)^2\la t\ra^{-1/2}\,.$$
\par

Since $\|e^{\a|z|}\diag(\pd_x,1)\ell_{11}\|_{H^1}\lesssim \|c_t\|_Y+\|\gamma_t-\tc\|_Y$, it follows from Claim~\ref{cl:cx_t-bound} that
\begin{align*}
| \la U,\ell_{11}\ra_{\bE_x}|+| \la \ell_{11},\Psi_c\ra_{\bE_x}|
 \lesssim & (\|U_1\|_{\bW(t)}+\|U_2\|_{\bX_1}+\|\Psi_c\|_{\bX_1})
(\|c_t\|_Y+\|\gamma_t-\tc\|_Y)
\\ \lesssim & \left(\bM_{c,\gamma}(T)+\bM_1(T)+\bM_2(T)\right)^2\la t\ra^{-3/2}\,.
\end{align*}
Simlilarly, we have
$$|\la U, (\gamma_t-\tc)\pd_x\Phi_{c_0}\ra_{\bE_y}|
\lesssim \left(\bM_1(T)+\bM_2(T)\right)
\left(\bM_{c,\gamma}(T)+\bM_1(T)+\bM_2(T)^2\right)\la t\ra^{-3/2}\,.$$
By \eqref{eq:defl1}, \eqref{eq:l13exp} and the definition of $\Phi_c$,
\begin{align*}
& \|\sech\a z \ell_{12}\|_{H^2(\R^2)}+\|\cosh(\a z) \ell_{13}\|_{H^2(\R^2)}
+\|\cosh(\a z)E_2\pd_y^2(\Phi_c-\Phi_{c_0})\|_{L^2}
\\ \lesssim &
\|c_{yy}\|_Y+\|\gamma_{yy}\|_Y+\|(c_y)^2\|_Y+\|c_y\gamma_y\|_Y+\|(\gamma_y)^2\|_Y
\\ \lesssim  & \bM_{c,\gamma}(T)\la t\ra^{-3/4}\,.
\end{align*}
Combining the above with \eqref{eq:ljkge2}, we have
\begin{gather*}
|\la U,\ell_{13}\ra_\bE|\lesssim
(\bM_1(T)+\bM_2(T))\bM_{c,\gamma}(T)\la t\ra^{-3/2}\,,
\\  
\sum_{j=1,2}|\la \ell_{j2}+\ell_{j3},\Phi_c-\Phi_{c_0}\ra_{\bE_y}|
 \lesssim \bM_{c,\gamma}(T)^2\la t\ra^{-3/2}\,.
\end{gather*}
\par
We have
$|\la c\pd_x\Psi_c,\Psi_c\ra_{\bE_y}|\lesssim 
\bM_{c,\gamma}(T)^2\la t\ra^{-3/2}$ because
\begin{align*}
|\la \pd_x^j\pd_y^k(c\tpsi_c),\pd_x^{j-1}\pd_y^k\tpsi_c\ra|\lesssim &
\|c_y\|_Y^2+\|c_y\|_Y\|\gamma_y\|_Y\|\tc\|_{L^\infty}+\|\gamma_y\|_Y^2\|\tc\|_{L^\infty}^2
\\ \lesssim & \bM_{c,\gamma}(T)^2\la t\ra^{-3/2}
\quad\text{for $j$, $k\ge1$.}
\end{align*}
\par
Since $\mathbf{e_2}\cdot \Phi_c=r_c=O(e^{-2\a_c|z|})$,
$N\cdot\mathbf{e_1}=0$ and $E_1N'(\Psi_c)=O$,
\begin{align*}
& \left|\left\la N'(\Psi_c)U-N(U),\Phi_c\right\ra_{\bE_x}\right|+\left|\left\la N'(\Psi_c)U-N(U),\Phi_{c_0}\right\ra_{\bE_y}\right|
\\ \lesssim & (\|\Psi_c\|_{\bX_1}+\|U_1\|_{\bW(t)}+\|U_2\|_{\bX_1})
(\|U_1\|_{\bW(t)}+\|U_2\|_{\bX_1})
\\ \lesssim &  (\bM_{c,\gamma}(T)+\bM_1(T)+\bM_2(T))^2\la t\ra^{-3/2}\,.
\end{align*}
\par
In view of $\pd_y^2\Phi_{c_0}(z)= -\gamma_{yy}\pd_x\Phi{c_0}(z)+(\gamma_y)^2\pd_x^2\Phi_{c_0}(z)$, we have\newline
$\|\sech(\a z) \pd_y^2L_1\Phi_{c_0}\|_\bE\lesssim \|\gamma_{yy}\|_Y+\|(\gamma_y)^2\|_Y $ and
\begin{align*}
|\la U,\pd_y^2L_1\Phi_{c_0}\ra_\bE| \lesssim & 
(\|U_1\|_{\bW(t)}+\|U_2\|_{\bX_1})(\|\gamma_{yy}\|_Y+\|(\gamma_y)^2\|_Y)
\\ \lesssim & (\bM_{c,\gamma}(T)+\bM_1(T)+\bM_2(T))^2\la t\ra^{-3/2}\,.
\end{align*}
We can estimate the rest of the terms in a similar manner.
\end{proof}

\section{Virial identities}
\label{sec:virial}
In this section, we will prove virial identities for small solutions
to \eqref{eq:BL}. They give bounds on the transport of energy and
the rate of decay of the energy density in the region $x>c_1t$
with $c_1>1$.

If the initial data is small in the energy space and polynomially localized,
we can prove time decay estimates by using virial identities.
\begin{lemma}
  \label{lem:virial-0}
Let $\Phi(t)$ be a solution of \eqref{eq:BL1}.
For any $c_1>1$, there exist positive constants $\a_0$ and $\delta$
such that if $\a\in(0,\a_0)$ and $\|\Phi(0)\|_{\bE}<\delta$,
then
\begin{equation}
\label{eq:U1-decay1}
\int_0^\infty\int_{\R^2} p_\a(x-c_1t)\mathcal{E}(\Phi(t,x,y))\,dxdydt
\lesssim 
\int_{\R^2} (1+x_+)p_\a(x)\mathcal{E}(\Phi(0,x,y))\,dxdy\,,
\end{equation}
\begin{equation}
\label{eq:U1-decay2}
\begin{split}
& \left\|(1+x_+)^{\rho_1}p_\a(x)^{1/2}\mathcal{E}(\Phi(t,x+c_1t,y))^{1/2}
\right\|_{L^2(\R^2)}
\\ & \lesssim
\la t \ra^{-(\rho_2-\rho_1)}
\left\|(1+x_+)^{\rho_2}p_\a(x)^{1/2}\mathcal{E}(\Phi(0,x,y))^{1/2} 
\right\|_{L^2(\R^2)}\,,  
\end{split}
\end{equation}
where $\rho_1$ and $\rho_2$ are constants satisfying $\rho_2>\rho_1\ge0$.
\end{lemma}

To start with, we observe the energy conservation law of the Benney-Luke
equation. Let $\mathcal{E}$ and $\mathcal{F}_{quad}$ be as
in Section~\ref{subsec:ansatz} and let
\begin{gather*}
\mathcal{F}(\phi_1,\phi_2)
=\mathcal{F}_{quad}(\phi_1,\phi_2)+\mathcal{F}_{cubic}(\phi_1,\phi_2)\,,
\\
\mathcal{F}_{cubic}(\phi_1,\phi_2)
=-b\phi_2\nabla B^{-1}(\phi_2\Delta \phi_1+2\nabla\phi_1\cdot\nabla\phi_2)
-\phi_2^2\nabla\phi_1\,. 
\end{gather*}
By a straightforward computation,
\begin{gather*}
\pd_t \mathcal{E}(\Phi)=
\pd_t\phi (B\pd_t^2\phi-A\Delta\phi)+\nabla\cdot\mathcal{F}_1\,,\\
\mathcal{F}_1=
b(\pd_t\phi)(\nabla\pd_t^2\phi)+(\pd_t\phi)(\nabla\phi)
+a\{(\Delta\phi)(\nabla\pd_t\phi)-(\pd_t\phi)(\nabla\Delta\phi)\}\,.
\end{gather*}
In view of \eqref{eq:BL1},
\begin{align*}
& \pd_t\phi (B\pd_t^2\phi-A\Delta\phi)
=-\nabla\cdot\{(\pd_t\phi)^2\nabla\phi\}\,,
\\ &
b\pd_t^2\phi-a\Delta\phi+\phi=B^{-1}A\phi-bB^{-1}
\{(\pd_t\phi)(\Delta\phi)+\pd_t|\nabla\phi|^2\}\,.
\end{align*}
Combining the above, we have
\begin{equation}
  \label{eq:energy-flux}
\pd_t\mathcal{E}(\Phi)=\nabla\cdot\mathcal{F}(\Phi)\,.  
\end{equation}

To prove Lemma~\ref{lem:virial-0}, we need the following.
\begin{claim}
  \label{cl:em}
Suppose that $\Phi=(\phi_1,\phi_2)^T\in \bE$.
Then
$$\int_{\R^2} \left(\mathcal{E}(\Phi)+\mathbf{e_1}\cdot
\mathcal{F}_{quad}(\Phi)\right)\,dxdy\ge0\,.$$
\end{claim}
\begin{proof}
Let  $U(\xi,\eta)=\left((\xi^2+\eta^2)^{1/2}
\mF_y\phi_1(\xi,\eta),\mF_y\phi_2(\xi,\eta)\right)^T$ and
\begin{gather*}
d(\xi,\eta)=\frac{1+a(\xi^2+\eta^2)}{1+b(\xi^2+\eta^2)}
\frac{\xi}{\sqrt{\xi^2+\eta^2}}+a\xi\sqrt{\xi^2+\eta^2}\,,
\\  
D(\xi,\eta)=
\begin{pmatrix}
  1+a(\xi^2+\eta^2) & -id(\xi,\eta) \\ id(\xi,\eta) & 1+b(\xi^2+\eta^2)
\end{pmatrix}\,.
\end{gather*}
By the Plancherel theorem, 
$$
\int_{\R^2} \left(\mathcal{E}(\Phi)+\mathbf{e_1}\cdot
\mathcal{F}_{quad}(\Phi)\right)\,dxdy
=\frac12\int_{\R^2} U(\xi,\eta)^TD(\xi,\eta)\overline{U(\xi,\eta)}\,
d\xi d\eta\,.$$
Since $b>a>0$, the eigenvalues $\kappa_1(\xi,\eta)$ and $\kappa_2(\xi,\eta)$
of the Hermitian matrix $D(\xi,\eta)$ satisfies
$$\kappa_j(\xi,\eta)\ge 1+a(\xi^2+\eta^2)-|d(\xi,\eta)|\ge0\,.$$
Thus $D(\xi,\eta)$ is non-negative and Claim~\ref{cl:em} follows.
\end{proof}

Now we are in a position to prove Lemma~\ref{lem:virial-0}.
\begin{proof}[Proof of Lemma~\ref{lem:virial-0}]
Let
$$\mathcal{I}(t)=
\int_{\R^2}p_\a(x-c_1t)\mathcal{E}(\Phi)(t,x,y)\,dxdy\,.$$
By \eqref{eq:energy-flux},
\begin{align*}
  \frac{d}{dt}\mathcal{I}(t)
=& -\int_{\R^2}p_\a'(x-c_1t)\left\{c_1\mathcal{E}(\Phi)+\mathbf{e}_1\cdot
\mathcal{F}(\Phi)\right\}\,dxdy\,.
\end{align*}
Now let $\tilde{p}_\a=\sech\a(x-c_1t)$. Then $p_\a'(x-c_1t)=\a\tilde{p}_\a^2$
and it follows from Claims~\ref{cl:em} and \ref{cl:em-commu} and
\eqref{eq:pd-1} that
\begin{align*}
&  -\int_{\R^2}p_\a'(x-c_1t)\left(\mathcal{E}(\Phi)+\mathbf{e}_1\cdot
\mathcal{F}_{quad}(\Phi)\right)\,dxdy
\\=& 
 -\a\int_{\R^2}\left(\mathcal{E}(\tilde{p}_\a\Phi)+\mathbf{e}_1\cdot
\mathcal{F}_{quad}(\tilde{p}_\a\Phi)\right)\,dxdy
+O\left(\a\int_{\R^2}p_\a'(x-c_1t)\mathcal{E}(\Phi)\,dxdy \right)
\\ \lesssim & 
\a\int_{\R^2}p_\a'(x-c_1t)\mathcal{E}(\Phi)\,dxdy\,.
\end{align*}
Here we use the fact that $\a\|u\|_{L^2_\a(\R^2)}\le \|\pd_xu\|_{L^2(\R^2)}$.
Moreover,
\begin{gather*}
\left|\int_{\R^2} p_\a'(x-c_1t)\mathcal{F}_{cube}(\Phi)\,dxdy\right|
\lesssim
\|\Phi(t)\|_{\bE} \int_{\R^2} p_\a'(x-c_1t)\mathcal{E}(\Phi)\,dxdy\,.
\end{gather*}
Since $\|\Phi(t)\|_{\bE}=\|\Phi(0)\|_{\bE}$ by the energy conservation law,
we have for $t\ge0$,
\begin{equation}
  \label{eq:virial-pa}
  \begin{split}
& \int_{\R^2}p_\a(x-c_1t)\mathcal{E}(\Phi(t))\,dxdy
+\mu\int_0^\infty\int_{\R^2}p_\a(x-c_1t)\mathcal{E}(\Phi(t))\,dxdydt
\\ \lesssim & 
\int_{\R^2}p_\a(x)\mathcal{E}(\Phi(0))\,dxdy
  \end{split}
\end{equation}
with $\mu=(c_1-1)/2$
provided $\a_0$ and $\delta_5$ are sufficiently small.
\par
Using \eqref{eq:virial-pa}, we can prove \eqref{eq:U1-decay1} and
\eqref{eq:U1-decay2} in the same way as \cite[Lemma~7.2]{Miz19}.
Thus we complete the proof.
\end{proof}

Lemma~\ref{lem:virial-0} and \eqref{eq:energyU1} give an upper bound of
$\bM_1(T)$. 
\begin{lemma}
  \label{lem:virial-a}
There exist positive constants $\a_0$, $\delta_5$ and $C$ such that
if $\a\in(0,\a_0)$ and $\|U_0\|_\bE+\bM_\infty(T)\le \delta_5$, then
$\bM_1(T)\le C\left\|\la x\ra^2\mathcal{E}(U_0)^{1/2}\right\|_{L^2(\R^2)}$.
\end{lemma}
\bigskip

\section{The decay estimate for the exponentially localized 
perturbations}
\label{sec:U2}
In this section, we will estimate $\bM_2(T)$ using the exponential linear
stability of $e^{t\mL_{c_0}}\mathcal{Q}_{c_0}$
(Theorem~\ref{thm:linear-stability}).
\begin{lemma}
  \label{lem:exp-bound}
Let $\a\in(0,\a_{c_0})$ and $\eta_0$ be a sufficiently small positive number.
Then there exist positive constants $\delta_6$ and $C$
such that if  $\|\la x\ra^2\mathcal{E}(U_0)^{1/2}\|_{L^2}+\bM_{c,\gamma}(T)
+\bM_\infty(T)+\bM_2(T)+\bM_U(T)\le \delta_6$,
\begin{equation}
  \label{eq:M4-bound}
\bM_2(T)\le C\left(\|\la x\ra^2\mathcal{E}(U_0)^{1/2}\|_{L^2}
+\bM_{c,\gamma}(T)\right)\,.
\end{equation}
\begin{proof}
Applying Theorem~\ref{thm:linear-stability} to \eqref{eq:U2}, we have
\begin{align*}
  \|\mathcal{Q}_{c_0}U_2\|_{\bX_1}
\lesssim & \int_0^t e^{-\beta'(t-s)} \left(
\|\ell\|_{\bX_1}+\|N_1(s)\|_{\bX_1}+\|N_2(s)\|_{\bX_1}+\|N_3(s)\|_{\bX_1}
\right)\,ds\,.
\end{align*}
By \eqref{eq:defl2}, \eqref{eq:defl1} and
Claims~\ref{cl:cx_t-bound} and \ref{cl:B-1},
\begin{align*}
& \|\ell_1\|_{\bX_1}\lesssim 
(\bM_{c,\gamma}(T)+\bM_1(T)+\bM_2(T)^2)\la t\ra^{-3/4}\,,
\\ &  
\|\ell_2\|_{\bX_1}\lesssim 
(\bM_{c,\gamma}(T)+\bM_1(T)+\bM_2(T)^2)e^{-\a\{(c_0-1)t/2+h\}}\,.
\end{align*}
Since $N_1=\{N'(\Phi_{c(t,y)}(z)-\Phi_{c_0}(x))-N'(\Psi_{c(t,y)}(z_1))\}U_2$,
\begin{align*}
 \|N_1\|_{\bX_1} \lesssim & (\bM_{c,\gamma}(T)+\bM_\infty(T))
\bM_2(T)\la t\ra^{-3/4}\,.
\end{align*}
By \eqref{eq:energyU1},
\begin{align*}
\|N_2\|_{\bX_1}\lesssim  &(\|U_1\|_\bE+\|U_2\|_\bE)\|U_2\|_{\bX_1}
 \lesssim (\|U_0\|_\bE+\bM_U(T))\bM_2(T)\la t\ra^{-3/4}\,.
\end{align*}
Since $|\nabla^j\Phi_{c(t,y)}(z)|+ |\nabla^j\Psi_{c(t,y)}(z_1)|\lesssim e^{-2\a z}
\wedge 1$ for $j\ge0$,
$$\|N_3\|_{\bX_1}\lesssim  \|U_1\|_{\bW(t)}\lesssim \bM_1(T)\la t\ra^{-2}\,.
$$
Combining the above, we have for $t\in[0,T]$,
\begin{align*}
\|\mathcal{Q}_{c_0}(\eta_0)U_2\|_{\bX_1}
 \lesssim & (\bM_{c,\gamma}(T)+\bM_1(T))\la t\ra^{-3/4}
\\ & +  \bM_2(T)(\|U_0\|_\bE+\bM_2(T)+\bM_{c,\gamma}(T)+\bM_\infty(T)+\bM_U(T))
\la t\ra^{-3/4}\,.
\end{align*}
Since $\|\mathcal{P}_{c_0}(\eta_0)U_2\|_{\bX_1}\lesssim
(\bM_{c,\gamma}(T)+\bM_\infty(T))\|U_2\|_{\bX_1}$ by \eqref{eq:orth}, we have
$$\|U_2(t)\|_{\bX_1}\lesssim \|\mathcal{Q}_{c_0}(\eta_0)U_2\|_{\bX_1}\,.$$
Combining the above with Lemma~\ref{lem:virial-a}, we have 
\eqref{eq:M4-bound} provided $\delta_6$
is sufficiently small. Thus we complete the proof.
\end{proof}
\end{lemma}
\bigskip

\section{Large time behavior of the phase shift of line solitary waves}
\label{sec:phase}
In this section, we will prove Theorem~\ref{thm:main} and
Corollary~\ref{cor:instability2}.  To begin with, we remark that
$\bM_{c,\gamma}(T)$, $\bM_\infty(T)$, $\bM_1(T)$, $\bM_2(T)$ and
$\bM_U(T)$ remain small for every $T\in[0,\infty]$ provided the
initial perturbation $U_0$ is sufficiently small.  Combining
Proposition~\ref{prop:continuation} and Lemmas~\ref{lem:Mcx-bound},
\ref{lem:MU-bound}, \ref{lem:virial-a} and \ref{lem:exp-bound}, we
have the following.
\begin{proposition}
\label{prop:poly}  
There exist positive constants $\eps_0$ and $C$ such that if 
$\eps:=\|(1+x^2+y^2)\mathcal{E}(U_0)^{1/2}\|_{L^2(\R^2)}<\eps_0$,
then
$\bM_{c,\gamma}(\infty)+\bM_1(\infty)+\bM_2(\infty)+\bM_U(\infty)\le C\eps$.
\end{proposition}
We see that \eqref{OS}--\eqref{AS} follows immediately from
Proposition~\ref{prop:poly}.
To prove \eqref{eq:phase-lim}, we need the first order asymptotics
of $\gamma_y$ and $\tc$ as $t\to\infty$.
Let
$\nu=(a_{11}(c_0)-a_{22}(c_0))/2$ and let
$\omega(\eta)$, $\lambda_\pm(\eta)$ and $\Pi(\eta)$
be as in Section~\ref{sec:decay-est} and
$\widetilde{\Pi}(\eta)=\diag(i\eta,1)\Pi(\eta)$. 
Let
$$\mu(t)=\exp\left(\int_t^\infty b_{22}(s,0)\,ds\right)\,,\quad
\bd(t,\cdot)=\mu(t)e^{-\lambda_{1,c_0}t\sigma_3\pd_y}II(D_y)^{-1}\bb(t,\cdot)\,,
$$
where $b_{22}(t,\eta)$ is the $(2,2)$ entry of $\mathcal{A}_1(t,\eta)$
and $\sigma_3=\diag(1,-1)$.
Then \eqref{eq:modeq3} is translated into
\begin{align}
  \label{eq:modeq-bd}
& \bd_t=\left\{\lambda_{2,c_0}\pd_y^2I
+\lambda_{1,c_0}\pd_y\tilde{\omega}(D_y)\sigma_3\right\}\bd
+\pd_y(\widetilde{\mN}+\widetilde{\mN}'')+\widetilde{\mN}'\,,
\\ \label{eq:bd-init}
& \bb(0,\cdot)=-\wP_1\widetilde{\Pi}(D_y)^{-1}\tk(0,\cdot)\mathbf{e_2}\,,
\end{align}
where $\tilde{\omega}(\eta)=\omega(\eta)-1$,
$\widetilde{\mN}=e^{-\lambda_{1,c_0}t\sigma_3\pd_y}\widetilde{\Pi}(D_y)^{-1}
(n_1,n_2)^T$ and
 \begin{align*}
 \widetilde{\mN}'=& e^{-\lambda_{1,c_0}t\sigma_3\pd_y}\widetilde{\Pi}(D_y)^{-1} 
\chi(D_y)\bigl(\mN_2+E_2\obu{\mN}_{U_1}\bigr)\,,
\\ 
\widetilde{\mN}''=& e^{-\lambda_{1,c_0}t\sigma_3\pd_y}\widetilde{\Pi}(D_y)^{-1}
\biggl\{\pd_y^{-1}\left(I-\chi(D_y)\right)\bigl(\mN_2+E_2\obu{\mN}_{U_1}\bigr)
+\mN_3+E_1\obu{\mN}_{U_1}
\\ & 
+\diag(\pd_y,1)\left(\mN_4+\pd_y\oc{\mN}_{U_1}
+\pd_y^{-1}(\mathcal{A}_0(D_y)-\mathcal{A}_*(D_y))\bb\right)
\\ & +E_1\mathcal{A}_1(t,D_y)\bb
+\pd_y^{-1}\left(E_2\mathcal{A}_1(t,D_y)-b_{22}(t,0)E_2\right)\bb
\biggr\}\,.
\end{align*}
Note that $\diag(\pd_y,1)\mN_2=\mN_2$ and $\diag(\pd_y,1)\mN_3=\pd_y\mN_3$
since $E_2\mN_2=\mN_2$ and $E_1\mN_3=\mN_3$.
We have for $\eta\in[-\eta_0,\eta_0]$,
\begin{equation}
\label{eq:Pi-est}
\left|\widetilde{\Pi}(\eta)
-\begin{pmatrix}1  & -1\\ \lambda_{1,c_0} & \lambda_{1,c_0} \end{pmatrix}\right|
+\left|\widetilde{\Pi}(\eta)^{-1}-\frac{1}{2\lambda_{1,c_0}}
\begin{pmatrix}\lambda_{1,c_0} & 1 \\ -\lambda_{1,c_0} & 1 \end{pmatrix}\right|
\lesssim |\eta|\,.
\end{equation}
If $\eta_0$ is sufficiently small, then $\widetilde{\Pi}(D_y)$ and its
inverse belong to $B(Y)$
and it follows from Claim~\ref{cl:k-decay} and the definitions of $\bb$
and $\bd$ that
\begin{equation}
  \label{eq:bx-bb} 
  \begin{split}
& \left\|\begin{pmatrix}\gamma_y(t,\cdot)\\  b(t,\cdot)\end{pmatrix}
-\begin{pmatrix}1 & -1 \\ \lambda_{1,c_0} & \lambda_{1,c_0} \end{pmatrix}
e^{t\lambda_{1,c_0}\sigma_3\pd_y}\bd(t,\cdot)\right\|_Y
\\ & \lesssim \|k(t,\cdot)\|_Y+\|\pd_y\bd(t,\cdot)\|_Y
 \lesssim  \eps\la t\ra^{-3/4}\,.
  \end{split}
\end{equation}
Moreover, we have $\|\bd(0)\|_{Y_1}+\|\chi(D_y)\bd(0)\|_{L^1}\lesssim \eps$.

\par
We will investigate the asymptotic behavior of solutions by using the
compactness argument in \cite{Karch}.  More precisely, we consider the
rescaled solution $\bd_\lambda(t,y)=\lambda\bd(\lambda^2t,\lambda y)$
and prove that for any $t_1$ and $t_2$ satisfying $0<t_1<t_2<\infty$,
\begin{equation}
\label{eq:bd-c}
\lim_{\lambda\to\infty}\sup_{t\in[t_1,t_2]}
\|\bd_\lambda(t,y)-\bd_\infty(t,y)\|_{L^2(\R)}=0\,,
\end{equation}
where $\bd_\infty(t,y)={}^t(d_{\infty,+}(t,y), d_{\infty,-}(t,y))$ 
and $d_{\infty,\pm}(t,y)$ are self-similar solutions of the Burgers' equation
\begin{equation}
\label{eq:Burgers}
\pd_td_\pm=\lambda_{2,c_0}\pd_y^2d_\pm \pm p_3\pd_yd_\pm^2\,,
\quad p_3=\frac{1}{2}\left(\lambda_{1,c_0}^2p_1+a_{15}(c_0)+2a_{23}(c_0)\right)\,,
\end{equation}
satisfying
\begin{equation}
\label{eq:ss}
\lambda\bd_\infty(\lambda^2t,\lambda y)=\bd_\infty(t,y)
\quad\text{for every $\lambda>0$.}  
\end{equation}

First, we will show that rescaled solutions $\bd_\lambda$ are
uniformly bounded with respect to $\lambda\ge1$.
\begin{lemma}
\label{lem:tbb-bound}
Let $\eps$ be as in Proposition~\ref{prop:poly}.
Then there exists a positive constants $C$ such that
for any $\lambda\ge1$ and $t\in(0,\infty)$,
\begin{gather}
  \label{eq:tbb-bound1}
\sum_{k=0,1}\|\pd_y^k\bd_\lambda(t,\cdot)\|_{L^2}\le C\eps t^{-(2k+1)/4} \,,
\quad \|\pd_y^2\bd_\lambda(t,\cdot)\|_{L^2}\le C\eps\lambda^{1/2}t^{-1}\,,
\\  \label{eq:tbb-bound2}
\left\|\pd_t\bd_\lambda(t,\cdot)\right\|_{H^{-2}}
\le C(t^{-3/4}+t^{-5/4})\eps\,.
\end{gather}
\end{lemma}
\begin{proof}
By Proposition~\ref{prop:poly} and \eqref{eq:bx-bb},
we have and \eqref{eq:tbb-bound1}.
Let
$$\widetilde{\mN}_\lambda(t,y)
=\lambda^2\widetilde{\mN}(\lambda^2t,\lambda y)\,,
\quad
\widetilde{\mN}_\lambda'(t,y)
=\lambda^3\widetilde{\mN}(\lambda^2t,\lambda y)\,,
\quad
\widetilde{\mN}_\lambda''(t,y)
=\lambda^2\widetilde{\mN}(\lambda^2t,\lambda y)\,.$$
Then
\begin{equation}
  \label{eq:rebd-eq}
\pd_t\bd_\lambda=\lambda_{2,c_0}\pd_y^2\bd_\lambda
+\lambda\sigma_3\pd_y\tilde{\omega}(\lambda^{-1}D_y)\bd_\lambda
+\widetilde{\mN}_\lambda'
+\pd_y(\widetilde{\mN}_\lambda+\widetilde{\mN}_\lambda'')\,,  
\end{equation}
and it follows from \eqref{eq:H}, \eqref{eq:mN2}--\eqref{eq:mN4}
and the fact that $\mathcal{A}_0(\eta)-\mathcal{A}_*(\eta)=O(\eta^4)$ that
\begin{align}
& \label{eq:wmN}
\|\widetilde{\mN}_\lambda(t,\cdot)\|_{L^2}
= \lambda^{3/2}\|\widetilde{\mN}(\lambda^2t,\cdot)\|_Y
\lesssim \eps^2t^{-3/4}\,,
\\ & \label{eq:wmN'}
\|\widetilde{\mN}\lambda(t,\cdot)\|_{L^1}
=\lambda^2\|\widetilde{\mN}'(\lambda^2t,\cdot)\|_{L^1}
\lesssim 
\left(e^{-\a h}\eps+\eps^2\right)\lambda^{-1/2}t^{-5/4}\,,
\\ & \label{eq:wmN''}
\|\widetilde{\mN}''_\lambda(t,\cdot)\|_{L^2}
=\lambda^{3/2}\|\widetilde{\mN}_1''(\lambda^2t,\cdot)\|_Y
\lesssim \eps\lambda^{3/2}(1+\lambda^2t)^{-1}
\lesssim \eps\lambda^{-1/4}t^{-7/8}\,.
\end{align}
Combining \eqref{eq:rebd-eq}--\eqref{eq:wmN''} with \eqref{eq:tbb-bound1},
we have \eqref{eq:tbb-bound2}.
\end{proof}
By the standard compactness argument, we have the following.
\begin{corollary}
\label{cor:tbb-bound}
There exists a sequence $\{\lambda_n\}_{n\ge1}$ satisfying
$\lim_{n\to\infty}\lambda_n=\infty$ and $\bd_\infty(t,y)$ such that
\begin{align*}
& \bd_{\lambda_n}(t,\cdot) \to \bd_\infty(t,\cdot)
\quad\text{weakly star in $L^\infty_{loc}((0,\infty);H^1(\R))$,}\\
& \pd_t\bd_{\lambda_n}(t,\cdot)\to \pd_t\bd_\infty(t,\cdot)
\quad\text{weakly star in $L^\infty_{loc}((0,\infty);H^{-1}(\R))$,}\\
\end{align*}
\begin{equation}
  \label{eq:condition}
 \sup_{t>0}t^{1/4}\|\bd_\infty(t)\|_{L^2}\le C\eps\,,
\end{equation}
where $C$ is a constant given in Lemma~\ref{lem:tbb-bound}.
Moreover, for any $R>0$ and $t_1$, $t_2$ with $0<t_1\le t_2<\infty$,
\begin{equation}
  \label{eq:aubli}
\lim_{n\to\infty}\sup_{t\in[t_1,t_2]}
\|\bd_{\lambda_n}(t,\cdot)-\bd_\infty(t,\cdot)\|_{L^2(|y|\le R)}=0\,.  
\end{equation}
\end{corollary}

To prove \eqref{eq:bd-c}, we need the following.
\begin{lemma}
\label{lem:outer-region}
Suppose that $\eps$ is sufficiently small.
Then for every $t_1$ and $t_2$ satisfying $0<t_1\le t_2<\infty$,
there exist a positive constant $C$ and a function
$\tilde{\delta}(R)$ satisfying $\lim_{R\to\infty}\tilde{\delta}(R)=0$
such that
$$\sup_{t\in[t_1,t_2]}\left\|\bd_\lambda(t,\cdot)\right\|_{L^2(|y|\ge R)}
\le C(\tilde{\delta}(R)+\lambda^{-1/4})
\quad\text{for  $\lambda\ge 1$.}$$
\end{lemma}
Since we can prove Lemma~\ref{lem:outer-region} in the same way as
\cite[Lemma~10.4]{Miz19}, we omit the proof.

The initial data of $\bd_\infty(t,y)$ at $t=0$ is a constant multiple of
the delta function. Using the change of variable
$$ \tbd(t,y)=(d_+(t,y),d_-(t,y))^T:=\bd(t,y)-\bar{\bd}(t,y)\,,\quad
\bar{\bd}(t,y)=-\int_t^\infty \widetilde{\mN}'(s,\cdot)\,ds\,,
$$
we can transform \eqref{eq:modeq-bd} into a conservative system
\begin{equation}
  \label{eq:modeq5}
\pd_t\tbd=\lambda_{2,c_0}\pd_y^2\tbd+\widetilde{\mN}'
+\pd_y(\widetilde{\mN}+\widetilde{\mN}'')+\pd_y^2\widetilde{\mN}'''\,,
\end{equation}
where $\widetilde{\mN}'''=\lambda_{2,c_0}\bar{\bd}
+\pd_y^{-1}\tilde{\omega}(D_y)\sigma_3\bd$.
\begin{lemma}
\label{lem:Burgers-ini}
\begin{equation}
  \label{eq:B-ini}
\lim_{t\downarrow0}\int_\R \bd_\infty(t,y)h(y)\,dy
=h(0)\int_\R\tbd(0,y)\,dy
\quad\text{for any $h\in H^2(\R)$.}
\end{equation}
\end{lemma}
Using \eqref{eq:wmN}--\eqref{eq:wmN''}, we can prove
Lemma~\ref{lem:Burgers-ini} in the same way as \cite{Karch}.
See also the proof of \cite[Lemma~10.3]{Miz19}.
\par
Combining Lemmas~\ref{lem:outer-region} and \ref{lem:Burgers-ini},
Corollary~\ref{cor:tbb-bound} with Claim~\ref{cl:b-capprox}
and \eqref{eq:bx-bb}, we have the following.
\begin{proposition}
   \label{prop:burgers}
Suppose $c_0>1$ and that \eqref{ass:S} holds for $c=c_0$.
Let $\Phi(t,x,y)$ be as in Theorem~\ref{thm:main}.
Then there exist positive constants $\eps_0$ and $C$ such that if 
$\eps:=\|(1+x^2+y^2)\mathcal{E}(\Phi_0)^{1/2}\|_{L^2(\R^2)} <\eps_0$,
then
\begin{equation}
  \label{eq:profile}
\left\|\begin{pmatrix} \gamma_y(t,\cdot) \\ \tc(t,\cdot)\end{pmatrix}
-\begin{pmatrix}  1 & -1 \\ \lambda_{1,c_0} & \lambda_{1,c_0}\end{pmatrix}
\begin{pmatrix}u_B^+(t,\cdot+\lambda_{1,c_0}t) \\ u_B^-(t,\cdot-\lambda_{1,c_0}t)
\end{pmatrix}
\right\|_{L^2(\R)}=o(\eps t^{-1/4})
\end{equation}
as $t\to\infty$, where 
$$u_B^\pm(t,y)=\frac{\pm \lambda_{2,c_0}}{p_3}
\frac{m_\pm H_{\lambda_{2,c_0}t}(y)}{1+m_\pm\int_0^{y}H_{\lambda_{2,c_0}t}(y_1)\,dy_1}\,,
$$
and $m_\pm\in(-2,2)$ are constants satisfying
$$\frac{\lambda_{2,c_0}}{p_3}\log\left(\frac{2\pm m_\pm}{2\mp m_\pm}\right)=
\int_\R d_\pm(0,y)\,dy\,.$$
\end{proposition}
See e.g. \cite[Proof of Theorem~1.4]{Miz19} for the proof.
\par
Now we are in a position to prove \eqref{eq:phase-lim}.
\begin{proof}[Proof of \eqref{eq:phase-lim}]
By  Lemmas~\ref{lem:decay-BB} and \ref{lem:fun-asymp},
$$\lim_{t\to\infty}\left\|U(t,0)\bb(0)+\mu(0)H_{\lambda_{2,c_0}t}*W_t*
\tk(0)\mathbf{e_1}\right\|_{L^\infty}=0\,,$$
and
\begin{equation*}
\lim_{t\to\infty}\left\|U(t,0)\bb(0)
-\gamma_{\infty,1}\mathbf{1}_{[-\lambda_{1,c_0}t,\lambda_{1,c_0}t]}(y)\mathbf{e_1}
\right\|_{L^\infty(|y\pm \lambda_{1,c_0}t|\ge \delta t)}=0
\quad\text{for any $\delta>0$,} 
\end{equation*}
where $\gamma_{\infty,1}=-(2\lambda_{1,c_0})^{-1}\mu(0)\int_\R \tk(0,y)\,dy$.
By Proposition~\ref{prop:poly}, \eqref{eq:Linfy=0-1}, \eqref{eq:Linfy=0-2}
and \eqref{eq:Linfy=0-3},
\begin{align*}
& \biggl\|\int_0^t U(t,s)\bigl(\mN_3(s)+\pd_y\mN_4(s)
+E_1\obu{\mN}_{U_1}(s)+\pd_y^2\oc{\mN}_{U_1}(s)\bigr)\,ds\biggr\|_{L^\infty}
\lesssim 
\eps \la t\ra^{-1/4} \log(2+t)\,.
\end{align*}
It follows from Lemmas~\ref{lem:decay-BB}--\ref{lem:phase-lim},
\eqref{eq:mN2} and \eqref{eq:obu-est2} that as $t\to\infty$,
\begin{align*}
&  \biggl\|\int_0^t \left\{U(t,s)-\mu(s)H_{\lambda_{2,c_0}(t-s)}*W_{t-s}*\right\}
\bigl(\mN_2(s)+E_2\obu{\mN}_{U_1}(s)\bigr)\,ds\biggr\|_{L^\infty}
\to 0\,,
\end{align*}
\begin{equation*}
\biggl\|\int_0^t U(t,s)\bigl(\mN_2(s)+E_2\obu{\mN}_{U_1}(s)\bigr)\,ds
-\gamma_{\infty,2}\mathbf{1}_{[-\lambda_{1,c_0}t,\lambda_{1,c_0}t]}(y)\mathbf{e_1}
\biggr\|_{L^\infty(|y\pm \lambda_{1,c_0}t|\ge \delta t)}\to0\,,    
\end{equation*}
for any $\delta>0$, where
$\gamma_{\infty,2}=(2\lambda_{1,c_0})^{-1}
\int_{\R_+\times\R}\mu(s)
\bigl(\mN_2(s)+\obu{\mN}_{U_1}(s)\bigr)\cdot\mathbf{e_2}\,dsdy$ and
\begin{equation}
  \label{eq:gamma2}
|\gamma_{\infty,2}|\lesssim e^{-\a h}\eps+\eps^2\,.  
\end{equation}
\par
Finally, we will prove
 \begin{equation}
   \label{eq:UmN1}
\lim_{t\to\infty}\left\|
\int_0^t U(t,s)\mN_1(s)\,ds\right\|_{L^\infty(|y\pm \lambda_{1,c_0}t|\ge\delta t)}=0
\quad\text{ for any $\delta>0$.}
 \end{equation}
Let $\mN_{1,0}=(n_{1,0},\pd_yn_{2,0})^T$ and
\begin{align*}
& n_{1,0}=\left(\lambda_{1,c_0}^2p_1+a_{15}(c_0)\right)\left\{
u_B^+(t,\cdot+\lambda_{1,c_0}t)^2+u_B^-(t,\cdot-\lambda_{1,c_0}t)^2\right\}\,,
\\ &
n_{2,0}=2\lambda_{1,c_0}a_{23}(c_0)\left\{u_B^+(t,\cdot+\lambda_{1,c_0}t)^2
-u_B^-(t,\cdot-\lambda_{1,c_0}t)^2\right\}\,.
\end{align*}
By Propositions~\ref{prop:poly} and \ref{prop:burgers},
\begin{align*}
IV(t):=\left\|n_1-n_{1,0}\right\|_{L^2}+\left\|n_2-n_{2,0}\right\|_{L^2} \lesssim \eps^2\delta(t)\la t\ra^{-3/4}\,,
\end{align*}
where $\delta(t)$ is a function that tends to $0$ as $t\to\infty$.
Hence it follows from 
Lemma~\ref{lem:decay-BB} that as $t\to\infty$,
\begin{align*}
& \left\|\int_0^t U(t,s)\left(\mN_1(s)-\mN_{1,0}(s)\right)\,ds\right\|_{L^\infty}
\lesssim \int_0^t \la t-s\ra^{-1/4}IV(s)\,ds\to0\,.
\end{align*}
Moreover, by Lemmas~\ref{lem:decay-BB} and \ref{lem:fun-asymp},
\begin{align*}
\left\|\int_0^t \left(U(t,s)-e^{(t-s)\mathcal{A}_*}\mu(s)\right)\mN_{1,0}(s)\,ds
\right\|_{L^\infty}
\lesssim & \sum_{j=1,2}\int_0^t \la t-s\ra^{-1/2}\|n_{j,0}(s)\|_Y\,ds\to0\,,
\end{align*}
\begin{align*}
  & \biggl\|\int_0^t \mu(s)\bigl\{e^{(t-s)\mathcal{A}_*}\mN_{1,0}(s)
-\frac{1}{2\lambda_{1,c_0}}\sum_\pm
 H_{\lambda_{2,c_0}(t-s)}(\cdot\pm\lambda_{1,c_0}(t-s))    
  \\  & \phantom{\biggl\|\int_0^t \bigl\{e^{(t-s)\mathcal{A}_*-\frac{1}{2\lambda_{1,c_0}}\sum_\pm}}
*\left(\lambda_{1,c_0}n_{1,0}(s)\mp n_{2,0}(s)\right)\bigr\}\,ds
\biggr\|_{L^\infty}
\\ \lesssim & \sum_{j=1,2}\int_0^t \la t-s\ra^{-1}\|n_{j,0}(s)\|_{Y_1}\,ds
 \lesssim  \eps^2\la t\ra^{-1/2}\log(2+t)\,.
\end{align*}
Let $\tilde{\delta}(t)$ be a function satisfying
$\lim_{t\downarrow0}\tilde{\delta}(t)=0$.
Since $|u^B_\pm(s,y)|\lesssim H_{\lambda_{2,c_0}s}(y)$, we have
for $y$ satisfying
$|y+\lambda_{1,c_0}t|\wedge|y-\lambda_{1,c_0}t|\ge
4\{\lambda_{2,c_0}t/\tilde{\delta}(t)\}^{1/2}$,
\begin{align*}
&  \left|\int_0^t \mu(s)H_{\lambda_{2,c_0}(t-s)}(\cdot\pm\lambda_{1,c_0}(t-s))
*\left(\lambda_{1,c_0}n_{1,0}(s)\mp n_{2,0}(s)\right)\,ds\right|
 \lesssim e^{-\tilde{\delta}(t)^{-1}}\to0\,,
\end{align*}
as $t\to\infty$. Combining the above, we have \eqref{eq:UmN1}.
Thus we complete the proof.
\end{proof}
\begin{proof}[Proof of Corollary~\ref{cor:instability2}]
 Let $\zeta\in C_0^\infty(-\eta_0,\eta_0)$ such that $\zeta(0)=1$ and let
$$U_0(x,y)=-\eps\zeta_{2,c_0}(z)(\mathcal{F}_\eta^{-1}\zeta)(y)\,.$$
Then $\bb(0)=\eps(1+\beta_1(c_0)^{-1}S^3_{21}(0))^{-1}
(\mathcal{F}_\eta^{-1}\zeta)(y)\mathbf{e_2}$ and
$$\gamma_{\infty,1}\gtrsim 
\int_\R\tb(0,y)\,dy=\eps(1+O(e^{-\a h}))\,.$$
Combining the above with \eqref{eq:gamma2}, we have
$\gamma_\infty\gtrsim \eps$, 
where $\gamma_\infty$ is a constant in \eqref{eq:phase-lim}.
Corollary~\ref{cor:instability2} follows immediately from
\eqref{eq:phase-lim} and the fact that $\gamma_\infty\gtrsim \eps$.
Thus we complete the proof.
\end{proof}
\bigskip

\appendix
\section{Miscellaneous estimates for operator norms}
\begin{claim}
\label{cl:B-1}
Suppose that $\a\in (0,1/\sqrt{b})$. Then
\begin{align}
\label{eq:pd-1}
& \a\|u\|_{L^2_\a(\R^2)}\le \|\pd_xu\|_{L^2_\a(\R^2)}\,,\\
\label{eq:B-1}
& \|\nabla^jB^{-1}f\|_\bX+\|\pd_x^jB_0^{-1}f\|_{L^2_\a(\R^2)}
\lesssim \|f\|_{L^2_\a(\R^2)}\quad\text{for $j=0$, $1$, $2$,}
\\ & \label{eq:L1}
     \|L_1f\|_\bX+\|L_1(0)f\|_\bX\lesssim \|f\|_\bX\,,\quad
     \|L_2f\|_{\bX_2}\lesssim \|f\|_\bX\,,
\\ & \label{eq:L1'}
     \|\mL_1f\|_\bX+\|\mL_1(0)f\|_\bX\lesssim \|f\|_\bX\,.
\end{align}
\end{claim}
\begin{proof}
Since $\|f\|_{L^2_\a(\R^2)}=\|\hat{f}(\xi+i\a,\eta)\|_{L^2(\R^2)}$,
 we can prove \eqref{eq:pd-1} and \eqref{eq:B-1}
by using the Plancherel theorem
(see \cite{MPQ13}). We can obtain \eqref{eq:L1} and \eqref{eq:L1'}
by using \eqref{eq:B-1}.
\end{proof}

\begin{claim}
  \label{cl:ker-B}
Let $\a\in(0,1/\sqrt{b})$ and $1<q<2$. Then there exists a positive constant $C$ such that for any $\gamma\in\R$,
\begin{gather}
  \label{eq:clB-1exp1}
\|B^{-1}g\|_{L^2_\a(\R^2)}\le C\|e^{\a x}g\|_{L^1(\R^2)+L^2(\R^2)}\,,
\\  \label{eq:clB-1exp2}  
\|e^{-\a|x-\gamma|}B^{-1}g\|_{L^2(\R^2)}\le C\|e^{-\a|x-\gamma|}g\|_{L^1(\R^2)+L^2(\R^2)}\,,
\\  \label{eq:clB-1exp3}
\|e^{\a (x-\gamma)}\nabla B^{-1}g\|_{L^2(\R^2)}\le C\|e^{\a(x-\gamma)}g\|_{L^q(\R^2)+L^2(\R^2)}\,,
\\  \label{eq:clB-1exp4}
 \|e^{-\a |x-\gamma|}\nabla B^{-1}g\|_{L^2(\R^2)}\le C\|e^{-\a|x-\gamma|}g\|_{L^q(\R^2)+L^2(\R^2)}\,.
\end{gather}
\end{claim}
\begin{proof}
The green kernel $K(x,y)$ of the operator $B$ satisfies that
for $(x,y)\ne(0,0)$,
\begin{equation}
\label{eq:kerB}
|K(x,y)| \lesssim e^{-\sqrt{x^2+y^2}/\sqrt{b}}
\max(1,\log(x^2+y^2)^{-1})\,,\quad
|\nabla K(x)|\lesssim \frac{e^{-\sqrt{x^2+y^2}/\sqrt{b}}}{\sqrt{x^2+y^2}}\,.
\end{equation}
Let $\widetilde{K}(x,y,x_1,y_1)=e^{\a(x-x_1)}K(x-x_1,y-y_1)$.
Since 
$$e^{\a x}(B^{-1}g)(x,y)=\int_{\R^2} \widetilde{K}(x,y,x_1,y_1)
e^{\a x_1}g(x_1,y_1)\,dx_1dy_1\,,$$
and $\sup_{x,y} \|\widetilde{K}(x,y,x_1,y_1)\|_{(L^1\cap L^2)(\R^2_{x_1,y_1})}
+\sup_{x_1,y_1} \|\widetilde{K}(x,y,x_1,y_1)\|_{(L^1\cap L^2)(\R^2_{x,y})}<\infty$
for $\a\in(0,1/\sqrt{b})$, we have \eqref{eq:clB-1exp1}.
We can prove the rest in the same way.
\end{proof}

\begin{claim}
  \label{cl:em-commu}
Let $\a\in(0,1/\sqrt{b})$. There exists a positive constant $C$ such that
for any $\gamma\in\R$,
\begin{align*}
  & \|[\pd_x, \sech\a (x-\gamma)]g\|_{L^2(\R^2)}
+\|[B^{-1},\sech\a (x-\gamma)]g\|_{L^2(\R^2)}
\le C\a\|e^{-\a|x-\gamma|}g\|_{L^2(\R^2)}\,.
\end{align*}
\end{claim}
\begin{proof}
Using \eqref{eq:kerB},
we can prove Claim~\ref{cl:em-commu} in the same way as
\cite[Claim~11.2]{MPQ13}.
\end{proof}

\section{Operator norms of $S^j_k$}
\label{ap:s}
\begin{claim}
\label{cl:S1}
Let $\a\in(0,\a_{c_0})$.
There exist positive constants $\eta_1$ and  $C$ such that
for $\eta_0\in (0,\eta_1]$, $j\in\Z_{\ge0}$, $k=1$, $2$, $t\ge0$ and
$f\in L^2(\R)$,
\begin{align*}
& \|\pd_y^jS_k^1[Q_c](f)(t,\cdot)\|_Y \le C\|e^{\a z}Q_{c_0}\|_{L^2}
\|\pd_y^j\wP_1f\|_Y \,,\\
& \|\pd_y^jS_k^1[Q_c](f)(t,\cdot)\|_{Y_1}
\le C\|e^{\a z}Q_{c_0}\|_{L^2}\|\pd_y^j\wP_1f\|_{Y_1}\,,\\
& \bigl[\pd_y, S^1_k[q_c]\bigr]=0\,.
\end{align*}
\end{claim}

\begin{claim}
  \label{cl:S2}
Let $\a\in(0,\a_{c_0})$.
There exist positive constants $\eta_1$, $\delta$ and $C$ such that
if $\eta_0\in(0,\eta_1]$ and $\bM_{c,\gamma}(T)\le \delta$,
then for $k=1$, $2$, $t\in[0,T]$ and $f\in L^2(\R)$,
\begin{equation*}
 \|S_k^2[Q_c](f)(t,\cdot)\|_{Y_1}  \le C
\sup_{|c\in[c_0-\delta,c_0+\delta]}\left(\|e^{\a z}Q_c\|_{L^2}
+\|e^{\a z}\pd_cQ_c\|_{L^2}\right)\|\tc\|_Y\|f\|_{L^2}\,,
\end{equation*}
\begin{align*}
& \|\pd_yS_k^2[Q_c](f)(t,\cdot)\|_{Y_1} \\ \le & C\sum_{i=0,1,2}
\sup_{c\in[c_0-\delta,c_0+\delta]}
\|e^{\a z}\pd_c^iQ_c\|_{L^2} (\|c_y\|_Y\|f\|_{L^2}+\|\tc\|_Y\|\pd_yf\|_{L^2})\,,    
\end{align*}
\begin{align*}
& \|S_k^2[Q_c](f)(t,\cdot)\|_Y  \le C
\sum_{0\le i\le 2}\sup_{c\in[c_0-\delta,c_0+\delta]} \|e^{\a z}\pd_c^iQ_c\|_{L^2}\|
\tc\|_{L^\infty}\|f\|_{L^2}\,,\\
& \|[\pd_y,S_k^2[Q_c]]f(t,\cdot)\|_{Y_1}  \le C
\sum_{0\le i\le 3}\sup_{c\in[c_0-\delta,c_0+\delta]} \|e^{\a z}\pd_c^iQ_c\|_{L^2}
\|c_y\|_Y\|f\|_{L^2}\,.
  \end{align*}
\end{claim}
We can prove Claims~\ref{cl:S1} and \ref{cl:S2} in exactly the same way
as \cite[Claims~B.1 and B.2]{Miz15}.
\par

\begin{claim}
  \label{cl:S3}
Let $\a\in(0,\a_{c_0})$. There exist positive constants $C$ and $\eta_1$
such that for $\eta\in(0,\eta_1]$, $h\ge 0$, $k=1$, $2$ and $t\ge0$,
\begin{gather*}
\|\chi(D_y)S^3_{k1}[p](f)(t,\cdot)\|_{L^1}
\le Ce^{-\a((c_0-1)t/2+h)}\|e^{\a z}p\|_{L^2}\|f\|_{L^1(\R)}\,,
\\
\|S^3_{k1}[p](f)(t,\cdot)\|_Y +\|S^3_{k2}[p](f)(t,\cdot)\|_Y
\le Ce^{-\a((c_0-1)t/2+h)}\|e^{\a z}p\|_{L^2}\|f\|_{L^2(\R)}\,,
\\ 
\|S^3_{k1}[p](f)(t,\cdot)\|_{Y_1}+\|S^3_{k2}[p](f)(t,\cdot)\|_{Y_1}
\le Ce^{-\a((c_0-1)t/2+h)}\|e^{\a z}p\|_{L^2}\|\wP_1f\|_{Y_1}\,.
\end{gather*}
\end{claim}
\begin{claim}
\label{cl:S4}
There exist positive constants $\eta_1$, $\delta$ and $C$ such that
if $\eta_0\in(0,\eta_1]$ and $\bM_{c,\gamma}(T)\le \delta$,
then for $k=1$, $2$, $t\in[0,T]$, $h\ge0$ and $f\in L^2$,
  \begin{gather*}
\|\chi(D_y)S^4_{k1}[p](f)(t,\cdot)\|_{L^1(\R)} \le
Ce^{-\a((c_0-1)t/2+h)}\|e^{\a z}p\|_{L^2}\|\tc\|_Y\|f\|_{L^2}\,,
\\ 
\|S^4_{k1}[p](f)(t,\cdot)\|_{Y_1}+\|S^4_{k2}[p](f)(t,\cdot)\|_{Y_1} \le
Ce^{-\a((c_0-1)t/2+h)}\|e^{\a z}p\|_{L^2}\|\tc\|_Y\|f\|_{L^2}\,.
\end{gather*}
\end{claim}
Let $S^j_k=S^j_{k1}-\pd_y^2S^j_{k2}$ for $j=5$, $6$ and
\begin{gather*}
S^5_{k1}(f)(t,y)=\wP_1\left(
\int_\R f(y)\wU_2(t,z,y)\cdot\pd_c\zeta_{k,c(t,y)}^*(z)\,dz\right)\,,\\
S^5_{k2}(f)(t,y)=\frac{1}{2\pi}\int_{-\eta_0}^{\eta_0}
\int_{\R^2} f(y_1)\wU_2(t,z,y_1)\cdot\pd_cg_{k1}^*(z,\eta,c(t,y_1))
e^{i(y-y_1)\eta}\,dz dy_1d\eta\,,\\
S^6_{k1}(f)(t,y)=\wP_1\left(\int_\R
f(y)\wU_2(t,z,y)\cdot\pd_z\zeta_{k,c(t,y)}^*(z)\,dz\right)\,,\\
S^6_{k2}(f)(t,y)=\frac{1}{2\pi}\int_{-\eta_0}^{\eta_0}\int_{\R^2}
f(y_1)\wU_2(t,z,y_1)\cdot\pd_zg_{k1}^*(z,\eta,c(t,y_1))
e^{i(y-y_1)\eta}\,dz dy_1d\eta\,.
\end{gather*}
\begin{claim}
  \label{cl:S5}
There exist positive constants $\eta_1$, $\delta$ and $C$ such that
if $\eta_0\in(0,\eta_1]$ and $\bM_{c,\gamma}(T)\le \delta$,
then for $k=1$, $2$, $t\in[0,T]$ and $f\in L^2$,
\begin{gather*}
\|\chi(D_y)S^5_{k1}(f)(t,\cdot)\|_{L^1(\R)}+\|\chi(D_y)S^6_{k1}(f)(t,\cdot)\|_{L^1(\R)}
\le C\|v_2(t)\|_X\|f\|_{L^2(\R)}\,,
\\ 
\sum_{j=5,6}\left( \|S^j_{k1}(f)(t,\cdot)\|_{Y_1}
+\|S^j_{k2}(f)(t,\cdot)\|_{Y_1}\right)
\le C\|v_2(t)\|_X\|f\|_{L^2}\,.
\end{gather*}
\end{claim}
We can prove Claims~\ref{cl:S3}--\ref{cl:S5} in the same way as
\cite[Claims~A.1--A.3]{Miz19}.
\bigskip

\section{Estimates of $R^j$}
\label{ap:r}

To start with, we estimate the difference between $b$ and $\tc$ and
the operator norms of $\delta B(c)$.
\begin{claim}
  \label{cl:b-capprox}
There exist positive constants $\delta$ and $C$ such that
if \linebreak $\sup_{t\in[0,T]}\|\tc(t)\|_Y \le \delta$, then for $t\in[0,T]$,
\begin{align*}
& \|b-\tc\|_Y\lesssim \|\tc\|_{L^\infty}\|\tc\|_Y\,,
\quad \|b-\tc\|_{Y_1}\lesssim \|\tc\|_Y^2\,,
\\ &
\|b_y-c_y\|_Y\lesssim \|\tc\|_{L^\infty}\|c_y\|_Y\,,\quad
\|b_y-c_y\|_{Y_1}\lesssim \|\tc\|_Y\|c_y\|_Y\,,
\\ &
 \|b_t-c_t\|_Y \lesssim \|\tc\|_{L^\infty}\|c_t\|_Y\,,
\quad  \|b_t-c_t\|_{Y_1} \lesssim \|\tc\|_Y\|c_t\|_Y\,,
\\ &
\|b_{yy}-c_{yy}\|_Y \lesssim \|\tc\|_{L^\infty}\|c_{yy}\|_Y
+\|c_y\|_{L^\infty}\|c_y\|_Y\,, \\ &
\|b_{yy}-c_{yy}\|_{Y_1} \lesssim \|\tc\|_Y\|c_{yy}\|_Y+\|c_y\|_Y^2\,,
\\ &
\left\|b-\tc-\frac12\rho''(c_0)\wP_1(\tc)^2\right\|_Y
\lesssim \|\tc\|_{L^\infty}^2\|\tc\|_Y\,,
\\ &
\left\|b-\tc-\frac12\rho''(c_0)\wP_1(\tc)^2\right\|_{Y_1}
\lesssim \|\tc\|_{L^\infty}\|\tc\|_Y^2\,.
\end{align*}
\end{claim}
\begin{proof}
Since
$b-\tc=\wP_1\left(\rho(c)-\rho(c_0)-\rho'(c_0)\tc\right)$,
we can prove Claim~\ref{cl:b-capprox} in exactly the same way as
\cite[Claim~D.6]{Miz15}.
\end{proof}

\begin{claim}
  \label{cl:B4}
There exist positive constants $C$ and $\delta$ such that if
$\bM_{c,\gamma}(T)\le \delta$, then for $t\in[0,T]$,
\begin{align*}
& \|B_4(c)-B_4(c_0)\|_{B(Y)} \le C\|\tc\|_{L^\infty}\,,
\\ &
\|B_4(c)-B_4(c_0)\|_{B(Y,Y_1)}
+\|\chi(D_y)(B_4(c)-B_4(c_0))\|_{B(Y,L^1)} \le C\|\tc\|_Y\,,
\\ & 
 \|\delta B(c)f\|_Y \le C\|\tc\|_{L^\infty}
(\|\pd_yf\|_{L^2}+\|c_y\|_Y\|f\|_{L^\infty})\,,
\\ & 
\|\delta B(c)f\|_{Y_1}+\|\chi(D_y)\delta B(c)f\|_{L^1} 
 \le C\|\tc\|_Y(\|\pd_yf\|_{L^2}+\|c_y\|_Y\|f\|_{L^\infty})\,.
\end{align*}
\end{claim}
\begin{proof}
The first two estimates follow immediately from the definition of $B_4$.
To prove the last two estimates, we use \eqref{eq:defdB},
the fact that $\pd_y^{-1}(I-\wP_1)\in B(Y)$ and
\begin{align*}
&(\wP_1B_1(c)^{-1}\wP_1)(\wP_1B_1(c)\wP_1)=
\wP_1-\wP_1B_1(c)^{-1}(I-\wP_1)B_1(c)\wP_1
\\ & = \wP_1-\wP_1\left(B_1(c)^{-1}-B_1(c_0)^{-1}\right)
(I-\wP_1)\left(B_1(c)-B_1(c_0)\right)\,.
\end{align*}
\end{proof}

Using Claims~\ref{cl:b-capprox} and \ref{cl:B4},
we have the following.
\begin{claim}
  \label{cl:Rg}
Let $R^{G,1}_{2,1}=\ta_{22}'(c)(c_y)^2$,
$R^{G,1}_{2,2}=\left(\ta_{22}(c)-\ta_{22}(c_0)\right)c_y-\ta_{22}(c_0)(b_y-c_y)$.
Then $R^{G,1}_2=\pd_y R^{G,1}_{2,2}-R^{G,1}_{2,1}$ and there exist positive constants
$\delta$ and $C$ such that if
$\bM_{c,\gamma}(T)\le \delta$, then for $t\in[0,T]$, 
\begin{gather*}
\|R^{G,1}_1\|_Y+\|R^{G,2}\|_Y\le C\bM_{c,\gamma}(T)^2\la t\ra^{-5/4}\,,
\\
\|R^{G,1}_{2,1}\|_{Y_1}+\|\chi(D)R^{G,1}_{2,1}\|_{L^1} \le
C\bM_{c,\gamma}(T)^2\la t\ra^{-3/2}\,,
\quad \|R^{G,1}_{2,2}\|_{Y_1}\le C\bM_{c,\gamma}(T)^2\la t\ra^{-1}\,,
\\
\|R^{G,2}_2\|_{Y_1}+\|\chi(D_y)R^{G,2}_2\|_{L^1} \le C\bM_{c,\gamma}(T)^2\la t\ra^{-5/4}\,.
\end{gather*}
\end{claim}
To prove the last estimate, we use the fact that $B_4(c)$ is an upper triangular
matrix.
\par

Let
\begin{align*}
R^{\ell_1,r}_{k1}=& \wS_1\left((b-\tc)\mathbf{e_1}\right)+
\wP_1\left\{\int_\R \pd_y\tilde{\ell}_{1r} \left(
\gamma_y\pd_z\zeta_{k,c(t,y)}^*(z)-c_y\pd_c\zeta_{k,c(t,y)}^*(z)
\right)\,dz\right\}\,,
\\ 
R^{\ell_1,r}_{k2}=&
\wP_1\left(\int_\R \pd_y\tilde{\ell}_{1r}\zeta_{k,c(t,y)}^*(z)\,dz\right)\,,
\\
R^{\ell_1,r}_{k3}=& 
\frac{1}{\sqrt{2\pi}}\wP_1\mF_\eta^{-1}
\la \pd_y^2\tilde{\ell}_{1r}, g_{k1}^*(z,\eta,c(t,y))e^{iy\eta}\ra\,.  
\end{align*}
Then $R^{\ell_1,r}_k=R^{\ell_1,r}_{k1}+\pd_yR^{\ell_1,r}_{k2}-\pd_y^2R^{\ell_1,r}_{k3}$.
\begin{claim}
  \label{cl:Rel1}
There exist positive constants $\delta$ and $C$ such that if
$\bM_{c,\gamma}(T)\le \delta$, then for $t\in[0,T]$, 
\begin{gather}
\label{eq:Rl1}
\|R^{\ell_1,1}\|_Y+\|\pd_yR^{\ell_1,2}\|_Y\le C
\bM_{c,\gamma}(T)(\bM_{c,\gamma}(T)+\bM_1(T)+\bM_2(T)^2)\la t\ra^{-5/4}\,,
\\ \label{eq:Rl1,r1}
\|R^{\ell_1,r}_{k1}\|_{Y_1}+ \|\chi(D_y)R^{\ell_1,r}_{k1}\|_{L^1}
\le C\bM_{c,\gamma}(T)^2\la t\ra^{-5/4}\,,
\\ \label{eq:Rl1,r2}
\|R^{\ell_1,r}_{k2}\|_Y+\|R^{\ell_1,r}_{k3}\|_Y
\le C\bM_{c,\gamma}(T)\la t\ra^{-1}\,.
\end{gather}
\end{claim}
\begin{proof}
We see that \eqref{eq:Rl1} immediately follows from
Claims~\ref{cl:S1}, \ref{cl:S2} and \ref{cl:cx_t-bound}.
Since
\begin{align*}
 \|\pd_y\tilde{\ell}_{1r}\|_\bX \lesssim &
\|\gamma_{yyy}\|_Y+\|c_{yyy}\|_Y+\|\pd_y(\gamma_y)^2\|_Y
+\|\pd_y(c_y\gamma_y)\|_Y+\|\pd_y(c_y)^2\|_Y
\\  \lesssim & \bM_{c,\gamma}(T)\la t\ra^{-1}\,,  
\end{align*}
we have \eqref{eq:Rl1,r1} and \eqref{eq:Rl1,r2}.  
\end{proof}
\begin{claim}
  \label{cl:Rl2}
Let $\a\in(0,\a_{c_0})$.
There exist positive constants $C$, $\delta$, $\eta_1$  and $h_0$ such that
such that if $\eta\in(0,\eta_1]$, $\bM_{c,\gamma}(T)\le \delta$ and
$h\ge h_0$, then for  $k=1$, $2$ and $t\in[0,T]$,
\begin{align*}
& \|R^{\ell_2,1}\|_{Y_1}+\|\chi(D_y)R^{\ell_2,1}\|_{L^1}+\|R^{\ell_2,2}\|_{Y_1}
\\ & \lesssim \bM_{c,\gamma}(T)(\bM_{c,\gamma}(T)+\bM_1(T)+\bM_2(T)^2)
e^{-\a\{(c_0-1)t/2+h\}}\,,
\end{align*}
$$\|R^{\ell_2,r}\|_Y \lesssim \bM_{c,\gamma}(T)e^{-\a\{(c_0-1)t/2+h\}}\,.$$
\end{claim}
\begin{proof}
By Claims~\ref{cl:S3} and \ref{cl:S4},
\begin{gather*}
\|\wS_{31}\|_{B(Y_1)}+\|\chi(D_y)\wS_{31}\|_{B(L^1(\R))}+\|\wS_{32}\|_{B(Y_1)}\lesssim e^{-\a\{(c_0-1)t/2+h\}}\,,\\
|A_1(t,0)|+\|A_2(t,D_y)\|_{B(Y_1)}\lesssim e^{-\a\{(c_0-1)t/2+h\}}\,.
\end{gather*}
Combining the above with Claims~\ref{cl:cx_t-bound}, \ref{cl:b-capprox}
and
\begin{align*}
& \|\pd_y\tilde{\ell}_{2r}\|_{L^2_\a(\R^2)}
 \lesssim \bM_{c,\gamma}(T)e^{-\a\{(c_0-1)t/2+h\}}\,,
\\ & \|\tilde{\ell}_{2N}\|_{L^1(\R_y;L^2_\a(\R_z))}\lesssim \bM_{c,\gamma}(T)^2
e^{-\a\{(c_0-1)t/2+h\}}\,,
\end{align*}
we have Claim~\ref{cl:Rl2}.
\end{proof}

Next, we will estimate $R^j_k$ ($0\le j\le 7$).
\begin{claim}
  \label{cl:R0}
There exist $R^0_{kj}$ ($k=1$, $2$, $j=0$, $1$, $2$)
satisfying $R^0_k=R^0_{k0}-2\pd_yR^0_{k1}-\pd_y^2R^0_{k2}$
and positive constants $C$ and $\delta$ such that if
$\bM_{c,\gamma}(T)\le\delta$, then for $t\in[0,T]$,
\begin{align*}
& \|R^0_{k0}\|_{Y_1}+\|\chi(D_y)R^0_{k0}\|_{L^1}+\|R^0_{k2}\|_{Y_1}
\le C  \bM_2(T)\bM_{c,\gamma}(T)\la t\ra^{-3/2}\,,
\\ & \|R^0_{k1}\|_{Y_1} \le C\bM_2(T)\bM_{c,\gamma}(T)\la t\ra^{-1}\,.
\end{align*}
\end{claim}
\begin{proof}
By \eqref{eq:gevp1}, \eqref{eq:gevp2} and \eqref{eq:orth},
$$II^0_k(t,\eta)=
\int_{\R^2} \wU_2\cdot\{\mL_{c(t,y)}^*-\mL_{c(t,y)}(\eta)^*\}
\left(g_k^*(z,\eta,c(t,y))e^{-iy\eta}\right)\,dzdy\,.$$
Let
$L_{12}(\eta)=B^{-1}B(\eta)^{-1}\{(b-a)\Delta+BA(\eta)\}E_{12}
-bV_{c(t,y)}^*B(\eta)^{-1}$,
$$ R^0_{k0}=\wP_1\left[\int_\R \wU_2\cdot\left\{
\mL_{c(t,y)}^*-\mL_{c(t,y)}(0)^*\right\}\zeta_{k,c(t,y)}^*(z)\,dz\right]\,,$$
\begin{multline*}
R^0_{k1}=\frac{1}{2\pi}\int_{-\eta_0}^{\eta_0} \int_{\R^2} e^{i(y-y_1)\eta}
\wU_2\cdot\bigl\{L_{12}(\eta)\pd_yg_k^*(z,\eta,c(t,y_1))
\\ +\pd_y\left(r_{c(t,y)}E_{12}B(\eta)^{-1}g_k^*(z,\eta,c(t,y_1))\right)\bigr\}
\,dzdy_1d\eta\,,
\end{multline*}
and $R^0_{k2}=\pd_y^{-2}(R^0_k-R^0_{k0}+2\pd_yR^0_{k1})$.
Since
\begin{align*}
& (\pd_y^2+\eta^2)\{g_k^*(z,\eta,c(t,y))e^{-iy\eta}\}
=[\pd_y^2, g_k^*(z,\eta,c(t,y)]e^{-iy\eta}\,,
\\ & \mL_{c(t,y)}^*-\mL_{c(t,y)}(\eta)^*
= L_{12}(\eta)(\pd_y^2+\eta^2) +(\pd_y^2+\eta^2)r_cE_{12}B(\eta)^{-1}\,,
\end{align*}
we have $\|R^0_{k1}\|_{Y_1} \lesssim 
\|\wU_2\|_{\bX}(\|\gamma_y\|_Y+\|c_y\|_Y)$ and
\begin{align*}
& \|R^0_{k0}\|_{Y_1}+\|\chi(D_y)R^0_{k0}\|_{L^1}+\|R^0_{k2}\|_{Y_1}
  \\ \lesssim & \|\wU_2\|_{\bX}
(\|\gamma_{yy}\|_Y+\|c_{yy}\|_Y+\|(c_y)^2\|_Y+\|c_y\gamma_y\|_Y+\|(\gamma_y)^2\|_Y)\,.
\end{align*}
Thus we complete the proof.
\end{proof}

\begin{claim}
  \label{cl:R1}
There exist positive constants $C$, $\delta$, $h_0$ and $R^1_{k,j}$ $(j=1,\,2)$
satisfying 
$$R^1_{k1}+\pd_yR^1_{k2}=
\wP_1\left(\int_\R \widetilde{N}_1\cdot \zeta_{k,c(t,y)}^*(z)\,dz\right)$$
such that if $\bM_{c,\gamma}(T)+\bM_\infty(T)\le\delta$ and $h\ge h_0$, then for $t\in[0,T]$,
\begin{align*}
& \|R^1_{k1}\|_{Y_1}+\|\chi(D_y)R^1_{k1}\|_{L^1}
\lesssim (\bM_1(T)+\bM_2(T))\bM_{c,\gamma}(T)^2\la t\ra^{-3/2}\,,
\\ & \|R^1_{k2}\|_{Y_1}+\|R^1_{k3}\|_{Y_1}
\lesssim (\bM_1(T)+\bM_2(T))\bM_{c,\gamma}(T)\la t\ra^{-1}\,,
\end{align*}
where
$R^1_{k3}=\int_{-\eta_0}^{\eta_0}
\int_{\R^2} \widetilde{N}_1\cdot g_{k1}^*(z,\eta,c(t,y))e^{-iy\eta}\,dzdy$.
\end{claim}
We remark that $R^1_k=R^1_{k1}+\pd_yR^1_{k2}-\pd_y^2R^1_{k3}$.
\begin{proof}[Proof of Claim~\ref{cl:R1}]
We have
\begin{multline} \label{eq:wN1}
\widetilde{N}_1
= -\tau_\gamma B^{-1}\{[\pd_y^2,\tau_{-\gamma}\varphi_c](U\cdot\mathbf{e_2})
+2(\pd_y\tau_{-\gamma} r_c)\pd_y(U\cdot\mathbf{e_1})\}\mathbf{e_2}
-\tau_\gamma N'(\Psi_c)U
\\ +\tau_\gamma B^{-1}\left([B,\tau_{-\gamma}](V_c\wU)
-\tau_{-\gamma}r_c[\pd_y^2,\tau_\gamma]E_{21}\wU\right)\,,
\end{multline}
and for example,
\begin{align*}
& \int_\R \tau_\gamma B^{-1}\{(\pd_y\tau_{-\gamma} r_c)\pd_y(U\cdot\mathbf{e_1})\}
\mathbf{e_2}\cdot\zeta_{k,c}^*\,dz= \pd_yV_1 -V_2\,,
\\ &
V_1=\int_\R  
B^{-1}(U\cdot\mathbf{e_1}\pd_y\tau_{-\gamma} r_c)
\tau_{-\gamma}\zeta_{k,c}^*\cdot\mathbf{e_2}\,dx\,,
\\ &
V_2=\int_\R \left\{
B^{-1}(U\cdot\mathbf{e_1}\pd_y^2\tau_{-\gamma} r_c)
\tau_{-\gamma}\zeta_{k,c}^*\cdot\mathbf{e_2}
+
B^{-1}(U\cdot\mathbf{e_1}\pd_y\tau_{-\gamma} r_c)
\pd_y\tau_{-\gamma}\zeta_{k,c}^*\cdot\mathbf{e_2}\right\}\,dx\,,
\end{align*}
\begin{align*}
 \|V_1\|_{Y_1} \lesssim &
(\|c_y\|_Y+\|\gamma_y\|_Y)(\|U_1\|_{\bW(t)}+\|U_2\|_\bX)
\\ \lesssim & \bM_{c,\gamma}(T)(\bM_1(T)+\bM_2(T))\la t\ra^{-1}\,,
\end{align*}
\begin{align*}
& \|V_2\|_{Y_1}+\|\chi(D_y)V_2\|_{L^1} \lesssim \\ &
(\|c_{yy}\|_Y+\|\gamma_{yy}\|_Y+\|(c_y)^2\|_Y+\|c_y\gamma_y\|_Y+\|(\gamma_y)^2\|_Y)
\|(\|U_1\|_{\bW(t)}+\|U_2\|_\bX)
\\ \lesssim & \bM_{c,\gamma}(T)(\bM_1(T)+\bM_2(T))\la t\ra^{-3/2}\,,
\end{align*}
and
$$\left\|\int_\R \tau_\gamma N'(\Psi_c)U\cdot \zeta_{k,c}^*\,dz\right\|_{L^1(\R_y)}
\lesssim e^{-\a\{(c_0-1)t/2+h\}}M_{c,\gamma}(T)(\bM_1(T)+\bM_2(T))\,.$$
We can estimate the rest in the same way.
\end{proof}

\begin{claim}
  \label{cl:R2}
Let
\begin{align*}
& R^2_{k1}=\wP_1\left(\int_\R N_2\cdot \zeta_{k,c(t,y)}^*(z)\,dz\right)\,,
\\
& R^2_{k2}=\frac{1}{2\pi}\int_{-\eta_0}^{\eta_0}
\int_{\R^2} N_2\cdot g_{k1}^*(z,\eta,c(t,y_1))e^{i(y-y_1)\eta}\,dzdy_1d\eta\,.
\end{align*}
Then $R^2_k=R^2_{k1}-\pd_y^2R^2_{k2}$ and there exist positive constants
$C$ and $\delta$ such that if
$\bM_{c,\gamma}(T)\le\delta$, then for $t\in[0,T]$,
\begin{align*}
\|R^2_{k1}\|_{Y_1}+\|\chi(D_y)R^2_{k1}\|_{L^1}+\|R^2_{k2}\|_{Y_1}
\lesssim \bM_2(T)(\bM_1(T)+\bM_2(T))\la t\ra^{-3/2}\,.
\end{align*}
\end{claim}

By the definition of $R^3$ and Lemma~\ref{lem:virial-a}, we have
\begin{equation}
 \label{eq:R3-est1}
  \|R^3\|_Y\lesssim \|U_1\|_{\bW(t)}\lesssim \bM_1(T)\la t\ra^{-2}\,.
\end{equation}
More precisely, we have the following.
\begin{claim}
  \label{cl:R3}
Let $\a\in(0,\a_{c_0})$. There exist positive constants $C$, $\delta$ and
$R^3_{2j}$ $(j=1, 2)$ satisfying
\begin{gather}
  \label{eq:R62}
  R^3_2+\frac{d}{dt}\wP_1\la \wU_1,\zeta_{2,c_0}^*\ra
=R^3_{21}-\pd_y^2R^3_{22}
\end{gather}
such that if $\bM_{c,\gamma}(T)\le \delta$, then for $t\in[0,T]$,
\begin{gather}
\|R^3_1\|_Y+\|R^3_2\|_Y+ \|R^3_{22}\|_Y \le C\bM_1(T)\la t\ra^{-2}\,,\\
\|R^3_{21}\|_{Y_1}+\|\chi(D_y)R^3_{21}\|_{L^1} \le
C\bM_1(T)(\bM_{c,\gamma}(T)+\bM_1(T)+\bM_2(T)^2)\la t\ra^{-9/4}\,.
\end{gather}
\end{claim}
\begin{proof}
Let
\begin{gather*}
R^{U_1,0}_k=\wP_1\left(
\int_\R V_{c_0}\wU_1\cdot\zeta_{k,c_0}^*\,dz\right)\,,
\quad
R^{U_1,1}_k=\wP_1\left(
\int_\R \widetilde{N}_3\cdot\zeta_{k,c(t,y)}^*\,dz\right)\,,
\\
R^{U_1,2}_k=\frac{1}{2\pi}\int_{-\eta_0}^{\eta_0}\int_{\R^2}
\widetilde{N}_3\cdot g_{k1}^*(z,\eta,c(t,y_1))e^{i(y-y_1)\eta}\,dzdy_1d\eta\,.
\end{gather*}
Then $R^3_k=R^{U_1,1}_k-\pd_y^2R^{U_1,2}_k$ and
\begin{align*}
& \|R^{U_1,1}_k\|_Y \lesssim \bM_1(T)\la t\ra^{-2}\,,
\quad
\|R^{U_1,2}_k\|_Y \lesssim (1+\bM_{c,\gamma}(T))\bM_1(T)\la t\ra^{-2}\,,
\\ &
\|R^{U_1,1}_k-R^{U_1,0}_k\|_{Y_1}
+\left\|\chi(D_y)(R^{U_1,1}_k-R^{U_1,0}_k)\right\|_{L^1}
\lesssim \bM_{c,\gamma}(T)\bM_1(T)\la t\ra^{-9/4}\,.
\end{align*}
Obviously, we cannot expect that $\|R^{U_1,0}_2\|_{Y_1}$ decays as $t\to\infty$.
The bad part of $R^{U_1,0}_2$ can be extracted as spatial and time derivatives of
$L^2$ functions.
In fact, we have
$R^{U_1,0}_2=R^{U_1}_a+\pd_y^2R^{U_1}_b$ with
$\|R^{U_1}_b\|_Y\lesssim \bM_1(T)\la t\ra^{-2}$ and
\begin{equation*}
R^{U_1}_a= \wP_1\left(
\int_\R V_{c_0}(0)\wU_1\cdot\zeta_{2,c_0}^*\,dz\right)\,.
\end{equation*}
As in \cite{Miz18}, we can decompose $R^{U_1}_a$ into a sum of an integrable
function and a time derivative of an $L^2$-function.
Since $\mL_{c_0}(0)^*\zeta_{2,c_0}^*=0$, it follows from \eqref{eq:wU1} that
\begin{align*}
\int_\R V_{c_0}(0)\wU_1\cdot\zeta_{2,c_0}^*\,dz
=&
-\la (c_0\pd_x+L_0)\wU_1,\zeta_{2,c_0}^*\ra
\\ =&
-\frac{d}{dt}\la \wU_1,\zeta_{2,c_0}^*\ra-\pd_y^2R^{U_1}_c+R^{U_1}_d\,,
\end{align*}
where
$R^{U_1}_c=\la L_1\wU_1,\zeta_{2,c_0}^*\ra$,
$R^{U_1}_d=\la N(U_1),\zeta_{2,c_0}^*\ra
+\left\la \wU_1, (\gamma_t\pd_z-[L,\tau_\gamma]\tau_{-\gamma})\zeta_{2,c_0}^*
\right\ra$ and 
\begin{align*}
& \|\wP_1R^{U_1}_c\|_Y \lesssim \bM_1(T)\la t\ra^{-2}\,,\quad
\|R^{U_1}_d\|_{L^1}
\lesssim \bM_1(T)(\bM_1(T)+\bM_{c,\gamma}(T)+\bM_2(T)^2)\la t\ra^{-9/4}\,.
\end{align*}
Thus we complete the proof.
\end{proof}
\par

Let $R^4_{kj}=S_{kj}^5(c_t)-S_{kj}^6(\gamma_t-\tc)$ for $j=1$, $2$.
Then $R^4_k=R^4_{k1}-\pd_y^2R^4_{k2}$.
By  Claims~\ref{cl:S4}, \ref{cl:S5} and \ref{cl:cx_t-bound},
we have the following.
\begin{claim}
  \label{cl:R4}
Let $\a\in(0,\a_{c_0})$. 
There exist positive constants $C$ and $\delta$ such that if
$\bM_{c,\gamma}(T)\le\delta$, then for $k=1$, $2$ and $t\in[0,T]$,
\begin{align*}
&  \|R^4_{k1}\|_{Y_1}+\|\chi(D_y)R^4_{k1}\|_{L^1}+\|R^4_{k2}\|_{Y_1}
\\ \le & C\bM_2(T)(\bM_{c,\gamma}(T)+\bM_1(T)+\bM_2(T)^2)\la t\ra^{-3/2}\,.
\end{align*}
\end{claim}

\begin{claim}
  \label{cl:R5}
Let $\a\in(0,\a_{c_0})$. 
There exist positive constants $\eta_1$, $C$, $\delta$ and $R^5_j$ $(j=1,\,2)$
such that if $\eta_0\in(0,\eta_1]$ and $\bM_{c,\gamma}(T)\le\delta$,
then for $k=1$, $2$ and $t\in[0,T]$, $R^5_k=R^5_{k1}+\pd_yR^5_{k2}$ and
\begin{equation}
\label{eq:R5}
\begin{split}
& \|R^5_{k1}\|_{Y_1}+\|\chi(D_y)R^5_{k1}\|_{L^1}\lesssim
 \la t\ra^{-3/2}\bM_{c,\gamma}(T)\bM_2(T)\,,
\\ &
\|R^5_{k2}\|_{Y_1}\lesssim \la t\ra^{-1}\bM_{c,\gamma}(T)\bM_2(T)\,.
\end{split}
\end{equation}
\end{claim}
\begin{proof}
Since $L^*=B^{-1}A\Delta E_{12}+E_{21}$ and $[\pd_x,\tau_{-\gamma(t,y)}]=0$,
\begin{align*}
[L^*,\tau_{\gamma(t,y)}]=& B^{-1}A[\pd_y^2,\tau_{\gamma(t,y)}]E_{12}
-aB^{-1}[\pd_y^2,\tau_{\gamma(t,y)}]\Delta E_{12}
\\ &+bB^{-1} [\pd_y^2,\tau_{\gamma(t,y)}]B^{-1}A\Delta E_{12}\,.
\end{align*}
Combining the above with
$\tau_{-\gamma(t,y)}[\pd_y^2, \tau_{\gamma(t,y)}]
=(\gamma_y)^2\pd_z^2+\gamma_{yy}\pd_z+2\gamma_y\pd_{zy}^2$,
we can find $R^5_{k1}$ and $R^5_{k2}$ satisfying 
$R^5_k=R^5_{k1}+\pd_yR^5_{k2}$ and \eqref{eq:R5}.
\end{proof}

\begin{claim}
  \label{cl:R6}
Let $R^{6,1}=[\pd_y^2,B_4(c)]r_6$ and $R^{6,2}=(B_4(c)-B_4(c_0))r_6$. 
Then $R^6=R^{6,1}+\pd_y^2R^{6,2}$ and there exist
positive constants $C$ and $\delta$ such that if 
$\bM_{c,\gamma}(T)\le \delta$, then for $t\in[0,T]$,
\begin{gather*}
\|R^{6,1}\|_{Y_1}+\|\chi(D_y)R^{6,1}\|_{Y_1} \le 
C\bM_{c,\gamma}(T)\left(\bM_{c,\gamma}(T)+\bM_1(T)+\bM_2(T)^2\right)\la t\ra^{-3/2}\,,
\\
\|R^{6,2}\|_{Y_1} \le C\bM_{c,\gamma}(T)
\left(\bM_{c,\gamma}(T)+\bM_1(T)+\bM_2(T)^2\right)\la t\ra^{-1}\,.
\end{gather*}
\end{claim}
\begin{proof}
Since
$\|r_6\|_Y
\lesssim \bM_{c,\gamma}(T)\left(\bM_{c,\gamma}(T)+\bM_1(T)+\bM_2(T)^2\right)\la t\ra^{-3/4}$ by Claims\linebreak\ref{cl:S1}, \ref{cl:cx_t-bound} and
\ref{cl:b-capprox}, we have Claim~\ref{cl:R6}.
\end{proof}
     
\begin{claim}
  \label{cl:R7}
There exist positive constants $C$, $\delta$  and $h_0$ such that if 
$\bM_{c,\gamma}(T)\le \delta$ and $h\ge h_0$, then for $t\in[0,T]$,
\begin{align*}
 \|R^7\|_{Y_1}+\|\chi(D_y)R^7\|_{L^1}
 \le & C\bM_{c,\gamma}(T)\left(\bM_{c,\gamma}(T)+\bM_1(T)+\bM_2(T)^2\right)
e^{-\a(\frac{c_0-1}{2}t+h)}\,.
\end{align*}
\end{claim}
Claim~\ref{cl:R7} follows immediately from
Claims~\ref{cl:S3}, \ref{cl:cx_t-bound}, \ref{cl:b-capprox} and \ref{cl:B4}.
\bigskip

\section{Estimates for $k(t,y)$}
\label{ap:k}
By Lemma~\ref{lem:virial-a},
the $L^2$-norm of $k(t,y)$ decays like $t^{-2}$ as $t\to\infty$.
\begin{claim}
  \label{cl:k-decay}
Suppose that $\sup_y|\gamma(t,y)|\le 1$ for $t\in[0,T]$.
There exist positive constants $\delta$ and $C$ such that if
$\|\la x\ra^2\mathcal{E}(U_0)^{1/2}\|_{L^2}<\delta$, then
\begin{equation*}
\|k(t,\cdot)\|_{L^2}\le C \la t \ra^{-2}\|\la x\ra^2\mathcal{E}(U_0)^{1/2}\|_{L^2}\quad\text{for $t\in[0,T]$.}  
\end{equation*}
\end{claim}
Next, we will give an upper bound of the growth rate of 
$\|k(t,y)\|_{L^1}$ when $U_0(x,y)$ is polynomially localized in $\R^2$.
\begin{claim}
  \label{cl:k-growth}
Suppose $\sup_y|\gamma(t,y)|\le 1$ for $t\in[0,T]$ and that $U_1$ is a solution of \eqref{eq:U1}.
Then there exists a positive constant $C$ such that
\begin{equation*}
\|\la \cdot \ra k(t,\cdot)\|_{L^2(\R)} \le C\la t \ra
\left(\left\|\la y\ra U_0\right \|_\bE+\left\|U_0\right \|_\bE^2\right)
\quad\text{for every $t\ge0$.}
\end{equation*}
\end{claim}
\begin{proof}
Since $U_1$ is a solution of \eqref{eq:U1}, 
it follows from \eqref{eq:energy-flux} that
 \begin{align*}
\frac{d}{dt}\int_{\R^2} \la y\ra^2\mathcal{E}(U_1)\,dxdy
=& -2\int_{\R^2} y\mathbf{e_2}\cdot\mathcal{F}(U_1)\,dxdy
\\ \lesssim & 
\left(\int_{\R^2} \la y\ra^2\mathcal{E}(U_1)\,dxdy\right)^{1/2}
(E(U_1)^{1/2}+E(U_1))\,.
 \end{align*}
Combining the above with \eqref{eq:energyU1}, we have
$$\left(\int_{\R^2} \la y\ra^2\mathcal{E}(U_1(t))\,dxdy\right)^{1/2}
\lesssim 
\left(\int_{\R^2} \la y\ra^2\mathcal{E}(U_0)\,dxdy\right)^{1/2}
+t(\|U_0\|_{\bE}+\|U_0\|_\bE^2)\,.$$
Since $\|\la y \ra k(t,y)\|_{L^2(\R_y)}\lesssim 
\|\la y\ra \mathcal{E}(U_1)^{1/2}\|_{L^2(\R^2)}$, we have
Claim~\ref{cl:k-growth}. 
\end{proof}
\bigskip

\section*{Acknowledgment}
T.~M. is supported by JSPS KAKENHI Grant Number 17K05332.
Y.~S. would like to express his gratitude to
Institute of Mathematics, Academia Sinica in Taiwan where he worked on this
research.

\end{document}